\documentclass[12pt]{iopart}
\usepackage[colorlinks=true,allcolors=blue]{hyperref}
\usepackage{amsthm}
\usepackage{graphicx}
\usepackage{iopams}
\usepackage{xr}
\usepackage{pdfpages}

\usepackage{etoolbox}
\patchcmd{\numparts}{\addtocounter{equation}{1}}{\refstepcounter{equation}}{}{}

\newcommand{\fbar}{\bar{f}}
\newcommand{\Zbar}{\bar{Z}}

\newcommand{\alphabar}{\bar{\alpha}}
\newcommand{\dd}[1]{\,\mathrm{d} #1}
\newcommand{\DD}{\mathrm{D}}
\newcommand{\defined}{:=}
\newcommand{\revdefined}{=:}
\newcommand{\Cr}{\Lambda}
\newcommand{\ccr}{\lambda}
\newcommand{\ccrt}{\ccr^*}
\newcommand{\thetat}{\theta^*}
\newcommand{\m}{\vec{m}}
\newcommand{\n}{\vec{n}}
\newcommand{\id}{\mathrm{id}}
\renewcommand{\t}{\mathrm{closed}}
\newcommand{\ev}{\lambda}

\newcommand{\D}{\mathbb{D}}
\renewcommand{\S}{\mathbb{S}^1}
\newcommand{\T}{\mathbb{T}}
\newcommand{\OT}{\T^N_{\mathrm{ordered}}}
\newcommand{\contF}{\mathcal{M}}
\newcommand{\MG}{\mathcal{G}}
\newcommand{\OCr}{V}
\newcommand{\Log}{\mathrm{Log\;}}
\newcommand{\rank}{\mathrm{rank\;}}
\renewcommand{\Re}{\mathrm{Re\;}}
\renewcommand{\Im}{\mathrm{Im\;}}

\newcommand{\abs}[1]{\left| #1 \right|}
\newcommand{\norm}[1]{\left|\left| #1 \right|\right|}
\newcommand{\mean}[1]{\left\langle #1 \right\rangle}
\renewcommand{\vec}[1]{\boldsymbol{#1}}
\newcommand{\cond}{\;;\,}
\newcommand{\set}[2]{\left\{#1\cond#2\right\}}
\newcommand{\M}[2]{G_{#1,#2}}
\newcommand{\level}[1]{\mathcal{L}_{#1}(\vec{\Cr})}
\newcommand{\diff}[2]{\frac{\mathrm{d}#1}{\mathrm{d}#2}}

\newcommand{\ie}{i.\,e.}
\newcommand{\eg}{e.\,g.}

\newtheorem{theorem}{Theorem}[section]
\newtheorem*{theoremInformal}{Theorem}

\newtheorem{lemma}[theorem]{Lemma}
\newtheorem{proposition}[theorem]{Proposition}
\newtheorem{definition}[theorem]{Definition}
\newtheorem*{remarks}{Remarks}
\newtheorem*{remark}{Remark}

\begin{document}
\title{Continua and persistence of periodic orbits in ensembles of oscillators}

\author{R Ronge$^1$, M A Zaks$^1$,	T Pereira$^{2,3}$}

\address{$^1$ Institut für Physik, Humboldt-Universit\"at zu Berlin, 12489 Berlin, Germany}
\address{$^2$ Institute of Mathematical and Computer Sciences, University of S\~ao Paulo, Brazil}
\address{$^3$ Department of Mathematics, Imperial College London, London SW7 2AZ, United Kingdom} 

\ead{robert.ronge@physik.hu-berlin.de}
	
\begin{abstract}
	Certain systems of coupled identical oscillators like the Kuramoto-Sakaguchi or the active rotator model possess the remarkable property of being Watanabe-Strogatz integrable.
	We prove that such systems, which couple via a global order parameter, feature a normally attracting invariant manifold that is foliated by periodic orbits.
	This allows us to study the asymptotic dynamics of general ensembles of identical oscillators by applying averaging theory. For the active rotator model, perturbations result in only finitely many persisting orbits, one of them giving rise to splay state dynamics. This sheds some light on the persistence and typical behavior of splay states previously observed.
\end{abstract}


\submitto{\NL}

\maketitle

\section{Introduction}

Studies of synchronization phenomena in ensembles of self-sustained oscillators or excitable elements have found numerous applications, from physics \cite{Cawthorne_1999,Pikovsky_Rosenblum_Kurths_2001} over engineering \cite{Strogatz_2005,Motter_2013} and neurobiology \cite{Singer_1993,Fries_2015} to social science \cite{Hegselmann_2002,Pluchino_2006}. For systems of \emph{identical} elements, an important class of phase models is of the form 
\begin{eqnarray}\label{eq:GeneralSystem}
	\dot{\phi}_j = f(Z)\rme^{\rmi\phi_j} + g(Z) + \fbar(Z)\rme^{-\rmi\phi_j}
\end{eqnarray}
where $\phi_j\in\S\defined\mathbb{R}/2\pi\mathbb{Z}$ with $j=1,\dots,N$ and
\begin{eqnarray*}
	Z = \frac{1}{N} \sum_{j=1}^N \rme^{\rmi\phi_j}
\end{eqnarray*}
is the Kuramoto order parameter \cite{Kuramoto_1975}. Hence every element $\phi_j$ couples to the rest of the ensemble via $Z$. The Kuramoto-Sakaguchi model for identical phase oscillators \cite{Sakaguchi_1986} as well as the active rotator model by Shinomoto and Kuramoto \cite{Shinomoto_1986} belong to this class. 

Systems of the form \eref{eq:GeneralSystem} have found particular interest since the seminal work \cite{Watanabe_Strogatz_1994} by Watanabe and Strogatz, showing that these systems possess $N-3$ conserved quantities due to a Möbius group symmetry \cite{Marvel_Mirollo_Strogatz_2009}. Thanks to this symmetry, the dynamics of \eref{eq:GeneralSystem} is determined by just three coupled differential equations in the Möbius group parameters for any $N\geq3$, a phenomenon known as Watanabe-Strogatz (WS) integrability. 

WS-integrability has lead to the understanding of multiple phenomena in ensembles of identically driven oscillators such as Chimera states \cite{Eldering_2021}, classification of attractors \cite{Engelbrecht_Mirollo_2014}, and asymptotic stability of periodic two-cluster states \cite{Gong_2019_2,Ronge_Zaks_2021}, to name a few.
WS-integrable systems can also feature splay states, which are of special interest because of their specific spatio-temporal symmetry \cite{Aronson_1991,Mirollo_1994} and are potential candidates for attractors for systems of type \eref{eq:GeneralSystem}. 

Equation \eref{eq:GeneralSystem} serves as a phase-reduced model for ensembles of general oscillatory or excitable elements to first-order Fourier mode \cite{Acebron_2005,Stankovski_2017}. Its Möbius group symmetry and subsequent partial integrability is the result of neglecting higher modes, which would otherwise break this symmetry. It remains an open problem as to how the WS-framework can be applied to general settings.

In this work, we focus on models of the type
\begin{equation}\label{eq:GeneralSystemPerturbed}
\eqalign{
	\dot{\phi}_j = f(Z)\rme^{\rmi\phi_j} + g(Z) + \fbar(Z)\rme^{-\rmi\phi_j} + \epsilon h(\phi_j) \\ 
	h(\phi_j) = \sum_{n=2}^{\infty} a_n \sin n \phi_j + b_n\cos n\phi_j 
}
\end{equation}
where the symmetry of \eref{eq:GeneralSystem} is broken by introducing higher harmonics in the \emph{on-site} dynamics. Note that the Fourier expansion of $h$ starts with the second modes since any zeroth or first mode perturbations can be absorbed by $f$ and $g$. 

Our motivation comes from the classic model of coupled identical active rotators
\begin{equation}\label{eq:ActiveRotatorsClassic}
	\dot{\phi}_j = \omega - \sin\phi_j+ \frac{\kappa}{N}\sum_{k=1}^N\sin(\phi_k-\phi_j)
\end{equation}
with $\abs{\omega}<1$ by Shinomoto and Kuramoto \cite{Shinomoto_1986} where the common fields in \eref{eq:GeneralSystem} are
\begin{eqnarray*}
	g(Z) = \omega \\
	f(Z) = -\frac{\rmi}{2}\left(1+\kappa\Zbar\right).
\end{eqnarray*}
In particular, this model is closely related to systems of coupled theta-neurons \cite{Laing_2018,Bick_2020}. Here, when we speak of an \emph{active rotator}, we mean an element that is \emph{excitable} \cite{Izhikevich_2007} and not already an oscillator by itself. Hence the sole interest in the case $\abs{\omega}<1$. It was argued in \cite{Zaks_Tomov_2016} that for $\kappa<-\sqrt{1-\omega^2}$, the system \eref{eq:ActiveRotatorsClassic} gives rise to a continuum of periodic orbits, \ie, when the coupling becomes sufficiently \emph{repulsive}. Families of periodic orbits are in fact a common occurrence in WS-theory as already noted in \cite{Watanabe_Strogatz_1994}. 

The classic active rotator model is a first order phase reduced description for general systems of coupled \emph{class I excitable elements}, characterized by being close to a saddle-node bifurcation on an invariant circle \cite{Izhikevich_2007}. The question arises what happens to the continuous family of periodic orbits from \cite{Zaks_Tomov_2016} if we break the symmetry of \eref{eq:ActiveRotatorsClassic} by, \eg, including higher-order Fourier modes in its on-site dynamics. We therefore investigate how the degenerate dynamics of \eref{eq:ActiveRotatorsClassic} change if we consider what we call \emph{generalized} active rotators \cite{Ronge_Zaks_2021,Ronge_Zaks_2021_2} instead, for which the equations of motion read
\begin{equation}\label{eq:ActiveRotatorPerturbed}
\eqalign{	
	\dot{\phi}_j = \omega - \sin\phi_j+ \epsilon h(\phi_j) + \frac{\kappa}{N}\sum_{k=1}^N\sin(\phi_k-\phi_j) \cr
	h(\phi_j) = \sum_{n=2}^{\infty} a_n \sin n \phi_j  + b_n\cos n\phi_j
}
\end{equation}
and which are of type \eref{eq:GeneralSystemPerturbed}. Below, we outline our main results.

\section{Informal statements of main results}

We first investigate the geometry of the continuum in \sref{sec:WSCase} which results in theorem~\ref{thm:ManifoldContinuum} and proposition~\ref{prop:SplayState}. Below, we give an informal statement for the special case of the active rotator model \eref{eq:ActiveRotatorsClassic}, discussed in \sref{sec:ARWS}.

\begin{theoremInformal}[Informal statement for active rotators]
	Consider a system of classic active rotators \eref{eq:ActiveRotatorsClassic} with $\abs{\omega}<1$. For $\kappa<\kappa_0=-\sqrt{1-\omega^2}$ and sufficiently large $N$, the system possesses a continuous family of periodic orbits where the union over these orbits forms a \emph{normally attracting invariant manifold} (NAIM). One of the orbits of this continuum features splay state dynamics.
\end{theoremInformal}

The orbits of the continuum can be parameterized by $N-3$ functionally independent cross-ratios $\vec{\ccr}$ which are preserved under the flow of \eref{eq:GeneralSystem}, see \cite{Marvel_Mirollo_Strogatz_2009}. The splay state corresponds thereby to special value $\vec{\ccr}=\vec{\ccrt}$, defined below. On introducing symmetry-breaking perturbations in the on-site dynamics in \eref{eq:GeneralSystemPerturbed}, the periodic orbits of the family generally vanish while the NAIM persists. Averaging the resulting perturbation term in the now time-dependent cross-ratios over an orbit of the unperturbed system results in an averaged system for which hyperbolic fixed points correspond to hyperbolic periodic orbits of the perturbed system. This leads to our second main result theorem~\ref{prop:Averaging} for which we again give an informal version below.

\begin{theoremInformal}
	The averaged dynamics of the cross-ratios possess a fixed point at $\vec{\ccrt}$, corresponding to the splay state in the unperturbed system.
\end{theoremInformal}

Whether $\vec{\ccrt}$ is hyperbolic depends on the specific setup of the model \eref{eq:GeneralSystemPerturbed}. However, we conjecture that splay state orbits are generically hyperbolic.

In \sref{sec:ARWS}, we apply our results to the active rotator models where we first
determine the critical coupling strength $\kappa_0$ of \eref{eq:ActiveRotatorsClassic} below which the NAIM-forming continuum arises and apply the averaging principle to the dynamics on the NAIM and conduct numerical experiments where we find that $\vec{\ccrt}$ is hyperbolic for the averaged dynamics so that the splay state is robust in \eref{eq:ActiveRotatorPerturbed}.

\paragraph{Strategy:} We outline the main steps to prove our main theorems~\ref{thm:ManifoldContinuum} and \ref{prop:Averaging}.:

\textbf{Step 1:} Since we are dealing with identically driven units, we can 
assume the elements to be in (strict) cyclic order which is preserved by the flow of \eref{eq:GeneralSystem}. In lemma~\ref{lm:WSChart}, we show that there exists a diffeomorphism from the space $\OT$ of phases $\vec{\phi}\in\T^N$ in strict cyclic order to the space of Möbius group parameters and the $N-3$ conserved cross-ratios of \eref{eq:GeneralSystem} which allows to write down the dynamics of $\vec{\phi}(t)$ in terms of these new coordinates.

\textbf{Step 2:} We show that there exists an open set of cross-ratios, for which the equations of motion for the system \eref{eq:GeneralSystem} in WS-variables can be approximated in $C^1$-norm by truncating higher-order terms in $N$. We then proceed in theorem~\ref{thm:ManifoldContinuum} by showing that there exists a family of periodic orbits for this truncated system, whose union possesses the desired NAIM structure by virtue of the persistence theorems for hyperbolic orbits \cite{Shilnikov2001}. 

\textbf{Step 3:} The dynamics on the NAIM describe the asymptotic dynamics of the perturbed system \eref{eq:GeneralSystemPerturbed}. For any smooth $h$, applying the method of averaging to the slowly-varying cross-ratio components of the vector field on the NAIM, we define an averaged perturbation function $\vec{F}_h$ which governs the averaged dynamics in $\vec{\ccr}$. Hyperbolic fixed points of this system determine persistent periodic orbits for the perturbed system \eref{eq:GeneralSystemPerturbed}. 

\textbf{Step 4:} Using cyclic permutation properties of the cross-ratios, we show in theorem~\ref{prop:Averaging} for $N=4$ that $\vec{F}_h$ vanishes for smooth choices of $h$ for the cross-ratios $\vec{\ccrt}$ that correspond to the splay state.  Afterwards, we show that this statement holds true for general $N\geq4$ by summing over the group action of cyclic permutations.

\section{Preliminaries}

\subsection{Normally attracting invariant manifolds}
Our definition of a normally attracting invariant manifold (NAIM) is according to \cite{Eldering_2013}. 
\begin{definition}[Normally attracting invariant manifold]
	\label{def:NAIM}
	Let $r\geq1$ and $\dot{\vec{x}}=\vec{f}(\vec{x})$ with $\vec{x}\in\mathbb{R}^n$ and $\vec{f}\in C^r$ be a dynamical system with flow $\Phi:\mathbb{R}\times\mathbb{R}^n\to\mathbb{R}^n$. A given $C^r$-submanifold $\contF\subset\mathbb{R}^n$ is then called an attracting invariant manifold of this system if it fulfills the following three criteria:
	\begin{enumerate}
		\item $\contF$ is invariant under the flow, \ie, $\Phi^t(\contF)=\contF$ $\forall t\in\mathbb{R}$.
		\item There exists a continuous splitting
		\begin{equation*}
			T_{\contF}\mathbb{R}^n=T\contF\oplus\mathcal{N}
		\end{equation*}
		of the tangent bundle $T\mathbb{R}^n$, restricted to $\contF$, into the tangent bundle $T\contF$ and a normal bundle $\mathcal{N}$ with continuous projections $\pi_{\contF}$ and $\pi_{\mathcal{N}}$. This splitting is invariant under the linearized flow $\DD\Phi^t=\DD\Phi^t_{\contF}\oplus \DD\Phi^t_{\mathcal{N}}$.
		\item There exist real numbers $a$ and $b$ with $a<r\,b\leq0$, and $C>0$ such that the following exponential growth conditions hold on $T_{\contF}\mathbb{R}^n$:
		\begin{eqnarray*}
			\forall t\leq0, (\vec{x},\vec{\nu})\in T\contF:& \norm{\DD\Phi^t_{\contF}(\vec{x})\cdot\vec{\nu}}&\leq C\,\rme^{bt}\norm{\vec{\nu}} \\
			\forall t\geq0, (\vec{x},\vec{\nu})\in \mathcal{N}:& \norm{\DD\Phi^t_{\mathcal{N}}(\vec{x})\cdot\vec{\nu}}&\leq C\,\rme^{at}\norm{\vec{\nu}}.
		\end{eqnarray*}
	\end{enumerate}
\end{definition}

A key feature of NAIMs is that they persist under small perturbations of $\vec{f}$, as is stated in the following theorem, cf. Theorem. 4.1 \cite{HirschPughShub2006}.
\begin{theorem}[Persistence of NAIMs]\label{thm:persistence}
	Let $\contF\subset\mathbb{R}^n$ be a compact attracting invariant manifold of the system $\dot{\vec{x}}=\vec{f}(\vec{x})$. Then, there exists an $\epsilon>0$ such that for any vector field $\vec{\tilde{f}}$ with $\norm{\vec{\tilde{f}}-\vec{f}}_{C^1}\leq\epsilon$, there exists a unique invariant $C^r$-manifold $\tilde{\contF}$ for $\vec{\tilde{f}}$ that is diffeomorphic to $\contF$, normally attracting, and $\mathcal{O}\left(\norm{\vec{\tilde{f}}-\vec{f}}_{C^1}\right)$-close to $\contF$.
\end{theorem}

We close this section with the following remarks.
\begin{remarks}~
	\begin{enumerate}
		\item If not stated otherwise, we employ the Euclidean norm $\norm{\cdot}$ and its induced norm.
		\item The persistence theorem~\ref{thm:persistence} for NAIMs also holds for manifolds with boundary as long as they are \emph{overflowing invariant}. In our case, the NAIMs have a boundary which is generally not overflowing invariant. This problem can be overcome by modifying $\vec{f}$ in a neighborhood of the boundary accordingly.
		\item We use Landau's big-$\mathcal{O}$ notation for order functions, \ie, $f(\epsilon)=\mathcal{O}(g(\epsilon))$ if there exist constants $K\geq0$ and $\epsilon_0\geq0$ such that $0\leq \abs{f(\epsilon)}\leq K\abs{g(\epsilon)}$ for all $0\leq\epsilon\leq\epsilon_0$.
	\end{enumerate}
\end{remarks}

\subsection{The averaging method}
To investigate the asymptotic dynamics of \eref{eq:GeneralSystemPerturbed}, we employ averaging theory. This approach relies on Theorem 7.9 in \cite{Chicone_2006} on averaging over periodic orbits, stated below.
\begin{theorem}\label{thm:AveragedPeriodicOrbits}
	Consider the system
	\begin{equation}\label{eq:AveragingTheoremFull}
	\eqalign{
		\dot{\vec{x}} = \epsilon\;\!\vec{F}(\vec{x},\psi) + \epsilon^2 \vec{F}_2(\vec{x},\psi,\epsilon) \cr
		\dot{\psi} =\Omega(\vec{x}) + \epsilon\,G(\vec{x},\psi,\epsilon)
	}
	\end{equation}
	where $\vec{x}\in\mathbb{R}^n$ and $\psi\in\S$ and assume that the function $\Omega(\vec{x})$ is bounded away from zero, \ie, $\Omega(\vec{x})>c>0$ for some number $c$ and all $\vec{x}\in\mathbb{R}^n$. If the averaged system
	\begin{equation*}
		\dot{\vec{y}} = \epsilon\;\!\vec{\hat{F}}(\vec{y})
	\end{equation*}
	with
	\begin{equation*}
		\vec{\hat{F}}(\vec{y}) \defined \frac{1}{2\pi} \int_0^{2\pi} \vec{F}(\vec{y},\psi) \dd{\psi}
	\end{equation*}
	possesses a hyperbolic fixed point $\vec{y}_0\in \mathbb{R}^n$ and $\epsilon>0$ is sufficiently small, \eref{eq:AveragingTheoremFull} possesses a periodic solution $t\mapsto\left(\vec{x}(t),\psi(t)\right)$ of the same stability type as $\vec{y}_0$.
\end{theorem}

\subsection{Watanabe-Strogatz integrability}
Watanabe and Strogatz showed that a class of systems of $N$ identically driven phase variables possess $N-3$ conserved quantities \cite{Watanabe_Strogatz_1994}. Later on, this finding was explained in a broader geometrical context in terms of Möbius transformations  \cite{Marvel_Mirollo_Strogatz_2009}. More specific, WS-integrable systems can be brought to the form 
\begin{equation}\label{eq:WSGeneral}
	\dot{\phi}_j = f\rme^{\rmi\phi_j} + g + \fbar \rme^{-\rmi\phi_j}
\end{equation}
with smooth functions $f:\T^N\to\mathbb{C}$ and $g:\T^N\to\mathbb{R}$ where $\T^N\defined(\S)^N$ denotes the $N$-dimensional torus and a bar denotes the complex conjugate. For short, we write $\vec{\phi}\defined(\phi_1,\dots,\phi_N)\in\T^N$.

The dynamics of the $\phi_j(t)$ in \eref{eq:WSGeneral} can be written as
\begin{eqnarray}\label{eq:MoebiusAction}
	\rme^{\rmi\phi_j(t)} = \M{\alpha(t)}{\psi(t)}\left(\rme^{\rmi\theta_j}\right),\quad j=1,\dots,N
\end{eqnarray}
where
\begin{equation*}
	\M{\alpha}{\psi}\left(z\right) \defined \frac{\alpha+ \rme^{\rmi\psi}z}{1+\alphabar\rme^{\rmi\psi}z}
\end{equation*}
for $z\in\mathbb{C}$ with parameters $\psi\in\S$ and $\alpha\in\D\defined\set{z\in\mathbb{C}}{\abs{z}<1}$. $G_{\alpha,\psi}$ is an element of the Möbius group
\begin{equation*}
	\mathcal{G}\defined\set{\M{\alpha}{\psi}:\D\to\D}{(\alpha,\psi)\in\D\times\S}.
\end{equation*}
where $\D\defined\set{z\in\mathbb{C}}{\abs{z}<1}$ denotes the complex unit disk. Generally, the Möbius group is the group of biholomorphic automorphisms of the augmented complex plane $\hat{\mathbb{C}}\defined\mathbb{C}\cup\{\infty\}$ \cite{Ahlfors_1953}. In the context of WS-theory, one considers the subgroup $\mathcal{G}$ of biholomorphic automorphisms of $\D$. By analytic continuation, these transformations map the boundary $\partial\D$ bijectively onto itself. In accordance to the literature on WS-theory we call $\MG$ \emph{the} Möbius group.

\Eref{eq:MoebiusAction} means that $\MG$ acts diagonally on $\T^N$. Setting in accordance to the literature for any $\vec{\theta}=(\theta_1,\dots,\theta_N)\in\T^N$
\begin{equation*}
	\rme^{\rmi\vec{\theta}}\defined \left( \rme^{\rmi\theta_1},\dots,\rme^{\rmi\theta_N}\right)
\end{equation*}
we write in a slight abuse of notation
\begin{equation}\label{eq:DiagonalAction}
	G_{\alpha,\psi}\left(\rme^{\rmi\vec{\theta}}\right)\defined\left(G_{\alpha,\psi}\left(\rme^{\rmi\theta_1}\right),\dots,G_{\alpha,\psi}\left(\rme^{\rmi\theta_N}\right)\right).
\end{equation}
The time evolution of the group parameters $\alpha$ and $\psi$ is then given by
\begin{equation}\label{eq:WSGeneralEqs}
	\eqalign{
	\dot{\alpha} = \rmi\left(f\alpha^2+g\alpha+\fbar\right) \\
	\dot{\psi} = f\alpha+ g + \fbar\alphabar
	}
\end{equation}
where $f$ and $g$ now depend implicitly on $\alpha$, $\psi$, and $\vec{\theta}$ by virtue of \eref{eq:MoebiusAction}.

\section{NAIMs for WS-integrable systems}\label{sec:WSCase}

The aim of this section is to establish a way in which the dynamics of a systems of identical phase variables $\vec{\phi}$ can be written in terms of the Möbius group parameters and the cross-ratios. The main result of this section is thus lemma~\ref{lm:WSChart} on the existence of a diffeomorphism between the set of all $\vec{\phi}\in\T^N$ in cyclic order and the space $\D\times\S\times \OCr$ where $\OCr$ is the $(N-3)$-dimensional space of conserved quantities.

\subsection{Watanabe-Strogatz variables}

Since all units are identically driven, the order of phases on $\S$ is preserved under the flow of \eref{eq:GeneralSystem}. We restrict our attention without loss of generality to the case where the $\phi_j$ are in (strict) cyclic order. This gives rise to the following definition.
\begin{definition}
	For fixed $N$, the space $\OT\subset\T^N$ is defined as
	\begin{equation*}
	\OT \defined \set{ \vec{\theta}\in\T^N}{\theta_1 <\dots <\theta_N<\theta_1+ 2\pi}.
	\end{equation*}
\end{definition}

The (diagonal) group action \eref{eq:DiagonalAction} of the Möbius group $\MG$ induces the following equivalence relation on $\OT$.

\begin{definition}
	Any two points $\vec{\vartheta},\vec{\theta}\in\OT$ are equivalent if and only if there exists a $\M{\alpha}{\psi}\in\MG$ such that $\rme^{\rmi\vec{\vartheta}} = \M{\alpha}{\psi}\left(\rme^{\rmi\vec{\theta}}\right)$ in which case we write $\vec{\vartheta}\sim\vec{\theta}$.
\end{definition}
That this is an equivalence relation is established below.
\begin{proposition}\label{prop:EquivalenceRelation}
	The relation $\sim$ is an equivalence relation on $\OT$.
\end{proposition}
\begin{proof}
	This follows immediately from the group properties of $\MG$ \cite{Ahlfors_1953}: Consider any three points $\vec{\vartheta}\sim\vec{\theta}\sim\vec{\phi}\in\OT$. Then, (i) the identity map $\M{0}{0}\in\MG$ guarantees that $\vec{\theta}\sim\vec{\theta}$, (ii) there exists a $\M{\alpha}{\psi}\in\MG$ such that $\rme^{\rmi\vec{\vartheta}} = \M{\alpha}{\psi}\left(\rme^{\rmi\vec{\theta}}\right)$. The inverse $\M{\alpha}{\psi}^{-1}\in\MG$ fulfills $\rme^{\rmi\vec{\theta}} = \M{\alpha}{\psi}^{-1}\left(\rme^{\rmi\vec{\vartheta}}\right)$ and $\vec{\theta}\sim\vec{\vartheta}$. (iii) there exists a $\M{\beta}{\chi}\in\MG$ with $\rme^{\rmi\vec{\theta}}=\M{\beta}{\chi}\left(\rme^{\rmi\vec{\phi}}\right)$ and $\rme^{\rmi\vec{\vartheta}} = \M{\alpha}{\psi}\circ\M{\beta}{\chi}\left(\rme^{\rmi\vec{\phi}}\right)$ where $\M{\alpha}{\psi}\circ\M{\beta}{\chi}\in\MG$ so that $\vec{\vartheta}\sim\vec{\phi}$.
\end{proof}

The resulting partition of $\OT$ in equivalence classes is invariant under the flow of \eref{eq:WSGeneral} according to \eref{eq:MoebiusAction}. We write
\begin{equation*}
	[\vec{\theta}]\defined\set{ \vec{\vartheta}\in\OT}{\exists \M{\alpha}{\psi}\in\MG : \rme^{\rmi\vec{\vartheta}} = \M{\alpha}{\psi}\left(\rme^{\rmi\vec{\theta}}\right)}
\end{equation*}
for its equivalence classes. 

The real-valued so-called \emph{cross-ratios}
\begin{equation*}
	\Cr_{p,q,r,s}(\vec{\theta}) \defined \frac {\big(\rme^{\rmi\theta_p}-\rme^{\rmi\theta_s}\big)\big( \rme^{\rmi\theta_q}-\rme^{\rmi\theta_r} \big)} {\big( \rme^{\rmi\theta_q}-\rme^{\rmi\theta_s} \big)\big(\rme^{\rmi\theta_p}-\rme^{\rmi\theta_r}\big)}
\end{equation*}
with $p,q,r,s=1,\dots,N$ are invariant under arbitrary Möbius transformations \cite{Ahlfors_1953} and are thus in particular conserved under the flow of \eref{eq:WSGeneral}. It was shown in \cite{Marvel_Mirollo_Strogatz_2009}, that all cross-ratios can be written in terms of $N-3$ functionally independent ones. To choose such a set, we make the following definition.

\begin{definition}
	Let the set $V\subset\mathbb{R}^{N-3}$ be defined as
	\begin{equation*}
		\OCr \defined \set{ \vec{\ccr}\in(0,1)^{N-3}}{1>\ccr_1>\dots>\ccr_{N-3}>0}.
	\end{equation*}
	The cross-ratio function $\vec{\Cr}:\OT\to\OCr$ is defined by
	\begin{equation}\label{eq:CrossRatios}
	\eqalign{
		\vec{\Cr}(\vec{\theta})&\defined\left( \Cr_1(\vec{\theta}),\dots,\Cr_{N-3}(\vec{\theta}) \right) \\
		\Cr_k(\vec{\theta})&\defined\Cr_{1,2,3,k+3}(\vec{\theta})= \frac{\big(\rme^{\rmi\theta_1}-\rme^{\rmi\theta_{k+3}}\big)\big( \rme^{\rmi\theta_2}-\rme^{\rmi\theta_3} \big)} {\big( \rme^{\rmi\theta_2}-\rme^{\rmi\theta_{k+3}} \big)\big(\rme^{\rmi\theta_1}-\rme^{\rmi\theta_3}\big)}
	}
	\end{equation}
	with $k=1,\dots,N-3$.
\end{definition}
The $N-3$ components $\Cr_k$ are functionally independent. We note some important properties for the spaces $\OT$ and $\OCr$ as well as the cross-ratio function $\vec{\Cr}$ below.
\begin{lemma}
	The following two assertions hold:
	\begin{enumerate}
		\item The set $\OT$ is invariant under the action of $\MG$.
		\item The cross-ratio \eref{eq:CrossRatios} defines a smooth function $\vec{\Cr}:\OT\to\OCr$.
	\end{enumerate}
\end{lemma}
\begin{proof}
	The first assertion follows from the fact that all Möbius transformations $\M{\alpha}{\psi}\in\MG$ are orientation preserving \cite{Ahlfors_1953}. Let $\vec{\theta}\in\OT$ and $\rme^{\rmi\vec{\vartheta}}=\M{\alpha}{\psi}\left(\rme^{\rmi\vec{\theta}}\right)$. By preservation of orientation, we have $\vec{\vartheta}\in\OT$.
	
	The second assertion follows from the fact that all derivatives of $\Cr_k$ exist because all $\theta_j$ are distinct and that
	\begin{eqnarray*}
		\DD_{\theta_{k+3}}\Cr_k(\vec{\theta}) = \frac{1}{2} \frac{\sin\frac{\theta_1-\theta_2}{2}\sin\frac{\theta_2-\theta_3}{2}}{\sin\frac{\theta_1-\theta_3}{2}\sin^2\frac{\theta_2-\theta_{k+3}}{2}}<0, \qquad& k=1,\dots,N-3
	\end{eqnarray*}
	so that each $\Cr_k$ is strictly monotonically decreasing in $\theta_{k+3}$ on $\OT$. In particular, we find $\Cr_{k}(\vec{\theta})>\Cr_{k+1}(\vec{\theta})$ on $\OT$. Additionally, we find that  $\lim_{\theta_{k+3}\downarrow\theta_3}\Cr_k(\vec{\theta})=1$ and $\lim_{\theta_{k+3}\uparrow\theta_1+2\pi}\Cr_k(\vec{\theta})=0$ so that $\vec{\Cr}$ indeed maps $\OT$ to $\OCr$.
\end{proof}

As a next step, we define a coordinate system on $\OT$ in terms of $\alpha$, $\psi$ and the cross-ratios $\vec{\ccr}\in \OCr$. We start with the following lemma.

\begin{lemma}\label{lm:CREquivalence}
	For any two $\vec{\vartheta},\vec{\theta}\in\OT$, we have
	\begin{equation*}
		\vec{\vartheta}\sim\vec{\theta} \Leftrightarrow \vec{\Cr}(\vec{\vartheta}) = \vec{\Cr}(\vec{\theta})
	\end{equation*}
	and in particular, we have
	\begin{equation*}
		[\vec{\theta}] = \set{\vec{\vartheta}\in \OT}{\vec{\Cr}(\vec{\vartheta})=\vec{\Cr}(\vec{\theta})},
	\end{equation*}
	\ie, the equivalence classes $[\vec{\theta}]$ are exactly the level sets of the cross-ratio function $\vec{\Cr}$.
\end{lemma}
\begin{proof}
	($\Rightarrow$): Since $\vec{\vartheta}\sim\vec{\theta}$, there exists a $\M{\alpha}{\psi}\in\MG$ with $\rme^{\rmi\vec{\vartheta}}=\M{\alpha}{\psi}\left(\rme^{\rmi\vec{\theta}}\right)$. Then $\vec{\Cr}(\vec{\vartheta})=\vec{\Cr}(\vec{\theta})$ follows the invariance of cross-ratios under arbitrary Möbius transformations.
	
	($\Leftarrow$): Let  $\vec{\Cr}(\vec{\vartheta})=\vec{\Cr}(\vec{\theta})$ for some $\vec{\vartheta},\vec{\theta}\in \OT$. We construct a suitable $\M{\alpha}{\psi}\in\MG$ for which $\M{\alpha}{\psi}\left(\rme^{\rmi\vec{\theta}}\right)=\rme^{\rmi\vec{\vartheta}}$ holds. Consider the two general Möbius transformations $\mu_{\vec{\theta}}:\hat{\mathbb{C}}\to\hat{\mathbb{C}}$ and $\mu_{\vec{\vartheta}}:\hat{\mathbb{C}}\to\hat{\mathbb{C}}$ with
	\begin{eqnarray*}
		\mu_{\vec{\theta}}(z) &\defined \frac{\Big(z-\rme^{\rmi\theta_1}\Big) \Big(\rme^{\rmi\theta_2}-\rme^{\rmi\theta_3}\Big)} {\Big(z-\rme^{\rmi\theta_2}\Big)\Big(\rme^{\rmi\theta_1}-\rme^{\rmi\theta_3}\Big)}, \qquad
		\mu_{\vec{\vartheta}}(z) &\defined \frac{\Big(z-\rme^{\rmi\vartheta_1}\Big)\Big(\rme^{\rmi\vartheta_2}-\rme^{\rmi\vartheta_3}\Big)} {\Big(z-\rme^{\rmi\vartheta_2}\Big)\Big(\rme^{\rmi\vartheta_1}-\rme^{\rmi\vartheta_3}\Big)}.
	\end{eqnarray*}
	These maps fulfill
	\begin{eqnarray*}
		\fl
		\mu_{\vec{\theta}}\left(\rme^{\rmi\theta_1}\right) = \mu_{\vec{\vartheta}}\left(\rme^{\rmi\vartheta_1}\right)=0, \qquad
		\mu_{\vec{\theta}}\left(\rme^{\rmi\theta_2}\right) = \mu_{\vec{\vartheta}}\left(\rme^{\rmi\vartheta_2}\right)=\infty, \qquad
		\mu_{\vec{\theta}}\left(\rme^{\rmi\theta_3}\right) = \mu_{\vec{\vartheta}}\left(\rme^{\rmi\vartheta_3}\right)=1
	\end{eqnarray*}
	so that the function $G=\mu_{\vec{\vartheta}}^{-1}\circ \mu_{\vec{\theta}}$ maps $\rme^{\rmi\theta_k}$ to $\rme^{\rmi\vartheta_k}$ for $k=1,2,3$. In particular, $G$ is a bijective conformal map from $\partial\D$ to $\partial\D$. It is also orientation preserving since  $\rme^{\rmi\theta_1},\rme^{\rmi\theta_2},\rme^{\rmi\theta_3}\in\partial\D$ and $\rme^{\rmi\vartheta_1},\rme^{\rmi\vartheta_2},\rme^{\rmi\vartheta_3}\in\partial\D$ are in the same cyclic order by assumption. $G$ is therefore an element of $\MG$. Our assumption $\Cr_k(\vec{\theta})=\Cr_k(\vec{\vartheta})$ implies that 
	\begin{equation*}
		\mu_{\vec{\theta}}\big(\rme^{\rmi \theta_{k+3}}\big)=\Cr_k(\vec{\theta}) = \Cr_k(\vec{\vartheta})=\mu_{\vec{\vartheta}}\big(\rme^{\rmi \vartheta_{k+3}}\big).
	\end{equation*}
	and $G\left(\rme^{\rmi \theta_k}\right)=\rme^{\rmi \vartheta_k}$ holds also for all $k=4,\dots,N$. We conclude that $G\left(\rme^{\rmi\vec{\theta}}\right) = \rme^{\rmi\vec{\vartheta}}$ which finishes the proof.
\end{proof}

Lemma~\ref{lm:CREquivalence} implies that we can parameterize the partition of $\OT$ in equivalence classes $[\vec{\theta}]\subset \OT$ via the cross-ratios $\vec{\ccr}\in \OCr$ such that we may set
\begin{equation*}
	[\vec{\theta}] \equiv \level{\vec{\ccr}} \defined \set{\vec{\vartheta}\in \OT}{\vec{\Cr}(\vec{\vartheta})=\vec{\ccr}}.
\end{equation*}
where $\vec{\ccr}=\vec{\Cr}(\vec{\theta})$. Each $\vec{\theta}\in \OT$ is an element of the level set $ \level{\vec{\ccr}}$ with $\vec{\ccr}=\vec{\Cr}(\vec{\theta})$. We have to introduce a suitable coordinate system on each set $ \level{\vec{\ccr}}$.

The parameters $\alpha$ and $\psi$ do not yet establish a coordinate system on $\level{\vec{\ccr}}$ but only parameterize how to transform any given point in $\level{\vec{\ccr}}$ to any other point in the same level set rather than uniquely determining where each point is located in $\level{\vec{\ccr}}$. This ambiguity comes from the fact that no a priori ``point-of-reference'' exists in $\level{\vec{\ccr}}$ relative to which all other points can are parameterized. To introduce such a unique point-of-reference $\vec{\Theta}(\vec{\ccr})$ for each $\level{\vec{\ccr}}$, we define $\vec{\Theta}:\OCr\to\OT$ by
\begin{equation}\label{eq:DefTheta}
\eqalign{
	\vec{\Theta}(\vec{\ccr})&\defined(\Theta_1(\vec{\ccr}),\dots,\Theta_N(\vec{\ccr})) \\
	\Theta_k(\vec{\ccr})&\defined
	\cases{
		-\pi+\frac{2\pi}{N}(k-1) &for $k=1,2,3$\\
		-\rmi\,\Log \frac{\rme^{\frac{2\pi\rmi}{N}}\left(\ccr_{k-3} +\ccr_{k-3}\,\rme^{\frac{2\pi \rmi}{N}}-1\right)}{-\ccr_{k-3}+(1-\ccr_{k-3})\,\rme^{\frac{2\pi \rmi}{N}}} &for $k=4,\dots,N$.
	}
}
\end{equation}
To see that this choice is suitable for our purposes, we need to show that it defines for each level set $ \level{\vec{\ccr}}$ a unique point-of-reference, \ie, that the image of $\vec{\Theta}$ intersects each $ \level{\vec{\ccr}}$ in exactly one point. We can define for each point $\vec{\theta}\in \OT$ its WS-coordinates $\left(\alpha,\psi,\vec{\ccr}\right)$
via the relation
\begin{equation*}
	\rme^{\rmi\vec{\theta}} = \M{\alpha}{\psi}\left(\rme^{\rmi\vec{\Theta}(\vec{\ccr})}\right)
\end{equation*}
where $\vec{\ccr}$ determines in \emph{which} level set $\level{\vec{\ccr}}$ the point $\vec{\theta}$ lies while $\alpha$ and $\psi$ determine \emph{where} $\vec{\theta}$ is located \emph{in} $\level{\vec{\ccr}}$ with respect to $\vec{\Theta}(\vec{\ccr})$. This is the purpose of the following lemma.

\begin{lemma}\label{lm:PointOfReference}
	The map $\vec{\Theta}:\OCr\to\OT$, defined by \eref{eq:DefTheta}, is smooth and a right inverse of the function $\vec{\Cr}$, \ie, $\vec{\Cr}\circ\vec{\Theta}(\vec{\ccr})=\vec{\ccr}$ and $\vec{\Theta}(\vec{\ccr})\in\level{\vec{\ccr}}$.
\end{lemma}
\begin{proof}
	$\vec{\Theta}$ is well-defined and smooth since the the numerator in the second line of \eref{eq:DefTheta} vanishes only for $\ccr_k=1/\left(1+\exp(2\pi\rmi/N)\right)\notin\mathbb{R}$ while the denominator vanishes only for $\ccr_k=\exp(2\pi\rmi/N)/\left(-1+\exp(2\pi\rmi/N)\right)\notin\mathbb{R}$. \Eref{eq:DefTheta} solves the equation $\vec{\Cr}(\vec{\Theta})=\vec{\ccr}$ for $\vec{\Theta}$ and so is the right inverse of $\vec{\Cr}$ by construction. Finally, we show that $\vec{\Theta}$ indeed maps $\OCr$ to $\OT$. For this, we show that for every $k>3$, $\Theta_{k}(\vec{\ccr})$ is strictly monotonically decreasing in $\ccr_{k-3}$ and that $\Theta_3(\vec{\ccr})>\Theta_4(\vec{\ccr})$. The assertion follows from the fact that $1>\ccr_1>\dots>\ccr_{N-3}>0$ and the limits
	\begin{eqnarray*}
		\lim_{\ccr_{1}\uparrow1}\Theta_4(\vec{\ccr}) = -\pi+\frac{2\pi}{N}2 \qquad\textrm{and}\qquad
		\lim_{\ccr_{N-3}\downarrow0}\Theta_N(\vec{\ccr}) = \pi.
	\end{eqnarray*}
	But since $N\geq4$, we find
	\begin{equation*}
	\DD_{\ccr_{k-3}}\Theta_{k}(\vec{\ccr}) = \frac{-2\sin\frac{2\pi}{N}}{1-\underbrace{2\ccr_{k-3}(1-\ccr_{k-3})}_{\in(0,1/2]}\underbrace{\left(1+\cos\frac{2\pi}{N}\right)}_{\in[1,2)}}<0
	\end{equation*}
	so that $\Theta_{k}$ is indeed strictly monotonically decreasing in $\ccr_{k-3}$ for all $k=4,\dots,N$.
\end{proof}

\begin{remark}
	In what follows, particular focus lies on the point
	\begin{equation}\label{eq:DefThetaTilde}
	\eqalign{
		\vec{\thetat}\defined\left(\thetat_1,\dots,\thetat_N\right)\qquad\textrm{with}\qquad
		\thetat_j \defined -\pi+\frac{2\pi}{N}(j-1)
	}
	\end{equation}
	for which we find
	\begin{equation}\label{eq:DefLambdaTilde}
	\eqalign{
		\vec{\ccrt}\defined\left(\ccrt_1,\dots,\ccrt_{N-3}\right)\qquad\textrm{with}\qquad
		\ccrt_k \defined \Cr_k(\vec{\thetat}) = \frac{\sin\frac{\pi(k+2)}{N}}{2\cos\frac{\pi}{N}\sin\frac{\pi(k+1)}{N}}.
	}
	\end{equation}
\end{remark}

We state the final result of this section on the existence of the well-defined WS-coordinate system on $\OT$ in the following proposition.
\begin{proposition}\label{lm:WSChart}
	The map $\m:\D\times\S\times \OCr\to \OT$ with
	\begin{equation*}
	\m(\alpha,\psi,\vec{\ccr})\defined -\rmi\,\Log \M{\alpha}{\psi}\left(\rme^{\rmi\vec{\Theta}(\vec{\ccr})}\right)=-\rmi\,\Log\frac{\alpha+\rme^{\rmi\psi}\rme^{\rmi\vec{\Theta}(\vec{\ccr})}}{1+\alphabar\,\rme^{\rmi\psi}\rme^{\rmi\vec{\Theta}(\vec{\ccr})}}
	\end{equation*}
	is a smooth diffeomorphism.
\end{proposition}
\begin{proof}	
	As a composition of the smooth maps $(\alpha,\psi,\vec{\theta})\mapsto\M{\alpha}{\psi}(\rme^{\rmi\vec{\theta}})$, $\vec{\Theta}$, and the $\Log$-function, $\m$ is smooth. We show that $\m$ is also bijective and $\DD\m$ has full rank.
	
	Injectivity of $\m$: Suppose that for $(\alpha,\psi,\vec{\ccr}),(\alpha',\psi',\vec{\ccr}')\in\D\times\S\times \OCr$ the equality $\m(\alpha,\psi,\vec{\ccr})=\m(\alpha',\psi',\vec{\ccr}')$ holds. We find
	\begin{eqnarray*}
		-\rmi\,\Log\M{\alpha}{\psi}\left(\rme^{\rmi\vec{\Theta}(\vec{\ccr})}\right) =-\rmi\,\Log\M{\alpha'}{\psi'}\left(\rme^{\rmi\vec{\Theta}(\vec{\ccr}')}\right) \\
		\Rightarrow\rme^{\rmi\vec{\Theta}(\vec{\ccr})} = \M{\alpha}{\psi}^{-1}\circ\M{\alpha'}{\psi'}\left(\rme^{\rmi\vec{\Theta}(\vec{\ccr}')}\right).
	\end{eqnarray*}
	By lemma~\ref{lm:CREquivalence}, this implies
	\begin{eqnarray*} 
		\vec{\Theta}(\vec{\ccr})\in[\vec{\Theta}(\vec{\ccr}')]= \level{\vec{\ccr}'}\quad 
		\Rightarrow\quad \vec{\ccr} =\vec{\Cr}\circ\vec{\Theta}(\vec{\ccr})= \vec{\ccr}'.
	\end{eqnarray*}
	Further, $\vec{\Theta}(\vec{\ccr})$ always possesses at least three distinct coordinates (e.g. $\Theta_1$, $\Theta_2$, and $\Theta_3$). But since a Möbius map is uniquely defined by the images of three distinct points, we have $\M{\alpha}{\psi}\left(\rme^{\rmi\vec{\Theta}(\vec{\ccr})}\right)=\M{\alpha'}{\psi'}\left(\rme^{\rmi\vec{\Theta}(\vec{\ccr})}\right) \Leftrightarrow \M{\alpha}{\psi}=\M{\alpha'}{\psi'}\Leftrightarrow(\alpha,\psi)=(\alpha',\psi')$ and $(\alpha,\psi,\vec{\ccr})=(\alpha',\psi',\vec{\ccr}')$.
	
	Surjectivity of $\m$: For any $\vec{\theta}\in \OT$, the cross-ratios  $\vec{\ccr}=\vec{\Cr}(\vec{\theta})\in \OCr$ are well defined. Since $\vec{\Theta}(\vec{\ccr})$ and $\vec{\theta}$ are both elements of $\level{\vec{\ccr}}=[\vec{\Theta}(\vec{\ccr})]=[\vec{\theta}]$, there exists a $\M{\alpha}{\psi}\in\MG$ with $\M{\alpha}{\psi}\left(\rme^{\rmi\vec{\Theta}(\vec{\ccr})}\right) = \rme^{\rmi\vec{\theta}}$. But this implies the existence of an $(\alpha,\psi,\vec{\ccr})\in\D\times\S\times \OCr$ which is mapped by $\m$ to $\vec{\theta}$.
	
	At last, we show that the derivative $\DD \m$ has full rank everywhere. For each $x\in\{\alpha,\alphabar,\psi,\ccr_1,\dots,\ccr_{N-3}\}$, let $\DD_x \m=(\DD_x \m_1,\dots,\DD_x \m_N)^T$ denote the respective column of $\DD \m=\left[\DD_{\alpha}\m,\DD_{\alphabar}\m,\DD_{\psi}\m,\DD_{\ccr_1}\m,\dots,\DD_{\ccr_{N-3}}\m\right]$. 
	From the identity
	\begin{equation*}
		\vec{\Cr}(\m(\alpha,\psi,\vec{\ccr})) = \vec{\ccr}
	\end{equation*}
	we note first that $\DD_{\vec{\theta}}\vec{\Cr}\cdot\DD_{\vec{\ccr}}\m=\id_{N-3}$ or more explicitly
	\begin{equation}\label{eq:leftInverse}
		\pmatrix{
			\DD_{\theta_1}\Cr_1 & \dots  & \DD_{\theta_N}\Cr_1 \cr
			\vdots              & \ddots & \vdots             \cr
			\DD_{\theta_1}\Cr_{N-3} & \dots  & \DD_{\theta_N}\Cr_{N-3}
		}\cdot
		\pmatrix{
			\DD_{\ccr_1}m_1 & \dots  & \DD_{\ccr_{N-3}}m_1 \cr
			\vdots          & \ddots & \vdots \cr
			\DD_{\ccr_1}m_N & \dots  & \DD_{\ccr_{N-3}}m_N 
		} = \id_{N-3}
	\end{equation}
	where $\id_{N-3}$ is the $(N-3)$-dimensional identity matrix. Thus, $\DD_{\vec{\ccr}}\m$ has a left inverse $\DD_{\vec{\theta}}\vec{\Cr}$ and therefore has in particular full column rank.
	
	Secondly, we have 
	\begin{eqnarray}\label{eq:tangential}
		\DD_{\vec{\theta}}\vec{\Cr}\cdot\DD_{\alpha}\m = \DD_{\vec{\theta}}\vec{\Cr}\cdot\DD_{\alphabar}\m = \DD_{\vec{\theta}}\vec{\Cr}\cdot\DD_{\psi}\m = 0
	\end{eqnarray}
	so the column vectors $\DD_{\alpha}\m$, $\DD_{\alphabar}\m$, and $\DD_{\psi}\m$ are orthogonal to the column vectors $\DD_{\theta_j}\vec{\Cr}$ and thus linearly independent of the column vectors $\DD_{\ccr_k}\m$: if any linear combination
	\begin{equation*}
		\vec{y} = a\,\DD_{\alpha}\m + b\,\DD_{\alphabar}\m + c\,\DD_{\psi}\m
	\end{equation*}
	was in the span of the $\DD_{\ccr_k}\m$, \ie, if we would have
	\begin{equation*}
		\vec{y} = \sum_{j=1}^{N-3} c_j \DD_{\ccr_j}\m
	\end{equation*}
	with $c_k\neq0$ for some $k\in\{1,\dots,N-3\}$, we would find
	\begin{eqnarray*}
		0 = \DD_{\vec{\theta}}\Cr_k\cdot\vec{y}
		= \sum_{j=1}^{N-3}c_j \DD_{\vec{\theta}}\Cr_k\cdot\DD_{\ccr_j}\m
		= c_k
	\end{eqnarray*}
	where the first equality follows from \eref{eq:tangential} and the last from \eref{eq:leftInverse}. This contradicts the assumption that $c_k\neq0$.
	
	The three vectors $\DD_{\alpha}\m$, $\DD_{\alphabar}\m$, and $\DD_{\psi}\m$ are linearly independent. To see this, consider the $(N-3)\times3$ matrix $A=\left[\DD_{\alpha}\m,\DD_{\alphabar}\m,\DD_{\psi}\m\right]$ with $\rank A\leq3$. For any matrix, its rank is equal to the largest order square submatrix with nonzero determinant \cite{Gantmacher_2000}. Let the submatrix $\hat{A}$ consist of the first three components of each column of $A$. It is of the form
	\begin{eqnarray*}
	\hat{A}
		= \pmatrix{
		\frac{\displaystyle\rmi}{\displaystyle\rme^{\rmi \psi}-\alpha} &
		\frac{\displaystyle\rmi \rme^{\rmi \psi }}{\displaystyle\alphabar  \rme^{\rmi
				\psi}-1} & \frac{\displaystyle\rme^{\rmi \psi} (\abs{\alpha}^2 -1)}{\displaystyle(\rme^{\rmi \psi}-\alpha)(\alphabar  \rme^{\rmi \psi}-1)} \cr
		\frac{\displaystyle\rmi}{\displaystyle\rme^{\rmi\psi}\zeta_N-\alpha} & \frac{\displaystyle\rmi
			\rme^{\rmi\psi}\zeta_N}{\displaystyle\alphabar \rme^{\rmi\psi}\zeta_N-1}
		& \frac{\displaystyle\rme^{\rmi\psi}\zeta_N(\abs{\alpha}^2 -1)}{\displaystyle( \rme^{\rmi\psi}\zeta_N-\alpha)(\alphabar \rme^{\rmi\psi}\zeta_N-1)} \cr
		\frac{\displaystyle\rmi}{\displaystyle\rme^{\rmi\psi}\zeta_N^2-\alpha} & \frac{\displaystyle\rmi
			\rme^{\rmi\psi}\zeta_N^2}{\displaystyle\alphabar \rme^{\rmi\psi}\zeta_N^2-1}
		& \frac{\displaystyle\rme^{\rmi\psi}\zeta_N^2(\abs{\alpha}^2 -1)}{\displaystyle(\rme^{\rmi\psi}\zeta_N^2-\alpha)(\alphabar  \rme^{\rmi\psi}\zeta_N^2-1)}
		},
	\end{eqnarray*}
	where we set $\zeta_N=\exp(2\pi\rmi/N)$ and has the determinant
	\begin{equation*}
		\det \hat{A} = \frac{X}
		{Y_1Y_2}
	\end{equation*}
	where
	\begin{eqnarray*}
		X = -(1-\zeta_N)^3 (1+\zeta_N)(1-\abs{\alpha}^2)
		\rme^{3\rmi\psi}\zeta_N \\
		Y_1 = (\alpha -\rme^{\rmi \psi })(\alpha-\rme^{\rmi\psi}\zeta_N)(\alpha-\rme^{\rmi\psi}\zeta_N^2) \\
		Y_2 = (1-\alphabar  \rme^{\rmi \psi })	
		(1-\alphabar \rme^{\rmi\psi}\zeta_N)
		(1-\alphabar \rme^{\rmi\psi}\zeta_N^2).
	\end{eqnarray*}
	This determinant is nonvanishing: By the triangle inequality, we have
	\begin{eqnarray*}
		\abs{Y_1} \leq (1+\abs{\alpha})^3 \qquad\textrm{and}\qquad
		\abs{Y_2} \leq (1+\abs{\alpha})^3
	\end{eqnarray*}
	while for $N\geq3$ and with $\abs{1\pm\rme^{\rmi\theta}}=\sqrt{2}\sqrt{1\pm\cos\theta}$ and $1-\cos\theta\geq\theta^2/8$ for all $\theta\in\left[-\frac{2\pi}{3},\frac{2\pi}{3}\right]$, we have
	\begin{eqnarray*}
		\abs{X} = \abs{1-\rme^{\frac{2\pi \rmi}{N}}}^3\abs{1+\rme^{\frac{2\pi \rmi}{N}}}\left(1-\abs{\alpha}^2\right)
		\geq 
		\left(\frac{\pi}{N}\right)^3\left(1-\abs{\alpha}^2\right)
	\end{eqnarray*}
	where in the third line we also used $1+\cos\theta\geq1/2$ for all $\theta\in\left[-\frac{2\pi}{3},\frac{2\pi}{3}\right]$. This implies
	\begin{equation*}
		\abs{\det\hat{A}} \geq \left(\frac{\pi}{N}\right)^3\frac{1-\abs{\alpha}^2}{(1+\abs{\alpha})^6}>0
	\end{equation*}%
	so that $\hat{A}$ has full rank 3 and thus $\rank A=3$. Hence, $\DD_{\alpha}\m$, $\DD_{\alphabar}\m$, and $\DD_{\psi}\m$ are linearly independent. From the inverse function theorem for smooth functions, since $\m$ is smooth, so is its inverse \cite{Rudin1964}. This finishes the proof.
\end{proof}

Lemma~\ref{lm:WSChart} implies that we can express the state of \emph{any} system of identical phase variables $\vec{\phi}(t)$ in terms of $\alpha(t)$, $\psi(t)$, and $\vec{\ccr}(t)$ to which we refer to in what follows as \emph{WS-variables}. The equations of motion can be expressed in WS-variables via $\m^{-1}$. From now on we write the Watanabe-Strogatz equations \eref{eq:WSGeneralEqs} in their full form
\begin{numparts}\label{eq:WSFull}
\begin{eqnarray}
	\dot{\alpha} &= \rmi\left(f(Z)\alpha^2+g(Z)\alpha+\fbar(Z)\right) \label{eq:alphaDot}\\
	\dot{\psi} &= f(Z)\alpha+ g(Z) + \fbar(Z)\alphabar \label{eq:psiDot}\\
	\dot{\vec{\ccr}} &= \vec{0}. \label{eq:lambdaDot}
\end{eqnarray}
\end{numparts}
Next, we discuss some properties of the level set $\level{\vec{\ccrt}}$ which will be useful later.

\subsection{The level set of uniform distributions}
On the level set $\level{\vec{\ccrt}}$, we can close \eref{eq:WSFull} by truncating higher-order terms in $N$ and solve the resulting equations. Under the right conditions, this yields a family of periodic orbits whose union forms a NAIM. Interpreting the truncated terms as $C^1$-small perturbations, this NAIM also exists for the original model. We start this section with a discussion of the Kuramoto order parameter \cite{Kuramoto_1975}
\begin{equation*}
	Z \defined \frac{1}{N}\sum_{j=1}^N \rme^{\rmi \phi_j} = \frac{1}{N} \sum_{j=1}^N \M{\alpha}{\psi}\left(\rme^{\rmi \theta_j}\right).
\end{equation*}

The key observation for systems of type \eref{eq:GeneralSystem} is that on the level set $\level{\vec{\ccrt}}$ the order parameter $Z$ can be approximated by $\alpha$ if $N$ is sufficiently large \cite{Pikovsky_2008,Eldering_2021}. Since $f$ and $g$ in \eref{eq:GeneralSystem} only depend on $Z$, replacing $Z$ by $\alpha$ decouples \eref{eq:alphaDot} from $\psi$ which simplifies the problem of finding periodic solution of \eref{eq:WSFull}. To use this fact within the framework of the theory of normally attracting invariant manifolds, we need to show that the resulting perturbation is small in $C^1$-norm. The main result of this section is an estimate of $Z$ and its partial derivatives for large $N$, which is used to estimate the error from truncating higher order terms in $N$ in \eref{eq:WSFull}. We start by writing $Z$ in terms of $\alpha$, $\psi$, and $\vec{\ccr}$. We first define the auxiliary function $Z_{\vec{\theta}}:\D\times\S\times\T^N\to\D$ with
\begin{eqnarray*}
	Z_{\vec{\theta}}(\alpha,\psi,\vec{\theta}) &\defined \frac{1}{N} \sum_{j=1}^N \frac{\alpha+\rme^{\rmi\psi} \rme^{\rmi\theta_j}} {1+\alphabar\,\rme^{\rmi\psi}\rme^{\rmi\theta_j}} \\
	&= \frac{1}{N} \sum_{j=1}^N (\alpha+\rme^{\rmi\psi}\rme^{\rmi\theta_j}) \sum_{k=0}^{\infty} (-\alphabar\,\rme^{\rmi\psi}\rme^{\rmi\theta_j})^k \\
	&= \sum_{k=0}^{\infty} \alpha (-\alphabar\,\rme^{\rmi\psi})^k\frac{1}{N}\sum_{j=1}^N \rme^{\rmi k\theta_j} + \rme^{\rmi\psi} (-\alphabar\,\rme^{\rmi\psi} )^k \frac{1}{N}\sum_{j=1}^N \rme^{\rmi(k+1)\theta_j}.
\end{eqnarray*}
Introducing the symbols
\begin{equation*}
	\mean{\rme^{\rmi  k\vec{\theta}}} \defined \frac{1}{N}\sum_{j=1}^N \rme^{\rmi  k\theta_j},
\end{equation*}
we arrive at the general expression
\begin{eqnarray*}
	Z_{\vec{\theta}}(\alpha,\psi,\vec{\theta}) &= \alpha\sum_{k=0}^{\infty}(-\alphabar\,\rme^{\rmi\psi})^k\mean{\rme^{\rmi k\vec{\theta}}} + \rme^{\rmi\psi} \sum_{k=0}^{\infty} (-\alphabar\,\rme^{\rmi\psi} )^k\mean{\rme^{\rmi(k+1)\vec{\theta}}}. \\
	&= \alpha+(1-\abs{\alpha}^2)\,\rme^{\rmi\psi}\sum_{k=1}^{\infty}(-\alphabar\,\rme^{\rmi\psi})^{k-1}\mean{\rme^{\rmi k\vec{\theta}}}.
\end{eqnarray*}
From this, we can define $Z:\D\times\S\times \OCr\to\D$ by setting $\vec{\theta}=\vec{\Theta}(\vec{\ccr})$ in the equation above so that we arrive at
\begin{equation}\label{eq:MoebiusMeanField}
\eqalign{
	Z(\alpha,\psi,\vec{\ccr}) &\defined Z_{\vec{\theta}}(\alpha,\psi,\vec{\Theta}(\vec{\lambda})) \\ &=  \alpha+(1-\abs{\alpha}^2)\,\rme^{\rmi\psi}\sum_{k=1}^{\infty}(-\alphabar\,\rme^{\rmi\psi})^{k-1} \mean{\rme^{\rmi k\vec{\Theta}(\vec{\ccr})}}.}
\end{equation}
For $\vec{\ccr}=\vec{\ccrt}$, the symbol $\mean{\rme^{\rmi k\vec{\Theta}(\vec{\ccr})}}$ and $Z$ can be written in simple form as shown next.
\begin{lemma}\label{lm:AveragedRootsOfUnity}
	Recall that $\vec{\thetat}=\vec{\Theta}(\vec{\ccrt})$ is defined by $\thetat_j = -\pi+2\pi(j-1)/N$ with $j=1,\dots,N$. For $\vec{\thetat}$, we have
	\begin{equation*}
		\mean{\rme^{\rmi k\vec{\thetat}}} = 
		\cases{
		(-1)^k & if $k\in N\mathbb{Z}$ \cr
		0 & else.
		}
	\end{equation*}
\end{lemma}
\begin{proof}
	Let $k\in N\mathbb{Z}$, \ie, $k=Nl$ for some integer $l$. Then
	\begin{equation*}
		\rme^{\rmi k\thetat_j}= \rme^{-\rmi\pi k+\frac{2\pi \rmi}{N} N l(j-1)} = \rme^{-\rmi\pi k}=(-1)^k
	\end{equation*}
	for all $j$ and $\mean{\rme^{\rmi k\vec{\thetat}}} = \frac{1}{N}\sum_{j=1}^N (-1)^k = (-1)^k$. On the other hand, let $k\notin N\mathbb{Z}$. We use the well known identity $\sum_{j=1}^{N} z^{j-1} = \frac{z^N-1}{z-1}$ for any $z\neq1$ which yields
	\begin{eqnarray*}
		\mean{\rme^{\rmi k\vec{\thetat}}} = \frac{1}{N} \sum_{j=1}^N \rme^{-\rmi\pi k+\frac{2\pi \rmi}{N}(j-1)k}=0
	\end{eqnarray*}
	since $\rme^{\frac{2\pi \rmi}{N}k}\neq1$ for $k\notin N\mathbb{Z}$.
\end{proof}

Each level set $\level{\vec{\ccr}}\subset\OT$ is diffeomorphic to the space $\D\times\S$ so that we identify $\level{\vec{\ccr}}$ with $\D\times\S\times\{\vec{\ccr}\}$:
\begin{equation*}
\level{\vec{\ccr}}\cong\D\times\S\times\{\vec{\ccr}\}.
\end{equation*}
Consider a smooth function
\begin{eqnarray*}
	\vec{F}:\D\times\S\times\OCr&\to\mathbb{R}^n.
\end{eqnarray*}
The derivative of its restriction $\vec{F}|_{\level{\vec{\ccr}}}$ to the level set $\level{\vec{\ccr}}$ is given by
\begin{equation*}
	\tilde{\DD}\vec{F}\defined\DD\vec{F}|_{\level{\vec{\ccr}}} \equiv \left(\DD_{\alpha}\vec{F},\DD_{\alphabar}\vec{F},\DD_{\psi}\vec{F}\right).
\end{equation*}
Let
\begin{equation*}
	\D_r\defined\set{\alpha\in\mathbb{C}}{|\alpha|<r}.
\end{equation*}
The following lemma gives an estimate to leading order in $N$ for $Z(\alpha,\psi,\vec{\ccr})-\alpha$ on the level set $\level{\vec{\ccrt}}$ of uniform distributions.

\begin{lemma}\label{lm:MeanFieldApproximationtheta}
	Consider $\overline{\D}_r$ for any $0<r<1$ and let
	\begin{eqnarray*}
		\eta:(\alpha,\psi,\vec{\ccr})&\mapsto Z(\alpha,\psi,\vec{\ccr})-\alpha\qquad\textrm{on}\qquad \D\times\S\times\OCr&\to\mathbb{C}
	\end{eqnarray*}
	Then,
	\begin{eqnarray}
		&\sup_{(\alpha,\psi)\in\overline{\D}_r\times\S} \abs{\eta(\alpha,\psi,\vec{\ccrt})} = \mathcal{O}(r^{N-1}) \label{eq:etaAbsBound}\\
		&\sup_{(\alpha,\psi)\in\overline{\D}_r\times\S} \norm{\tilde{\DD}\eta(\alpha,\psi,\vec{\ccrt})} = \mathcal{O}(Nr^{N-2}). \label{eq:etaDifBound}
	\end{eqnarray}
\end{lemma}
\begin{proof}
	We start with \eref{eq:etaAbsBound}. From \eref{eq:MoebiusMeanField}, $\vec{\thetat}=\vec{\Theta}(\vec{\ccrt})$, and lemma~\ref{lm:AveragedRootsOfUnity}, we infer
	\begin{eqnarray*}
		\eta(\alpha,\psi,\vec{\ccrt}) &= -\left(1-\abs{\alpha}^2\right)\,\rme^{\rmi\psi}\sum_{k=1}^{\infty}(\alphabar\,\rme^{\rmi\psi})^{kN-1} \\
		&= -\sum_{k=1}^{\infty} \alphabar^{kN-1} \rme^{\rmi kN\psi} + \alpha\sum_{k=1}^{\infty} \alphabar^{kN} \rme^{\rmi kN\psi}
	\end{eqnarray*}
	yielding
	\begin{eqnarray*}
		\sup_{(\alpha,\psi)\in\overline{\D}_r\times\S} \abs{\eta(\alpha,\psi,\vec{\ccrt})} 
		&=\sup_{(\alpha,\psi)\in\overline{\D}_r\times\S} \abs{\left(1-\abs{\alpha}^2\right)\,\rme^{\rmi\psi} \sum_{k=1}^{\infty} (\alphabar\,\rme^{\rmi\psi})^{kN-1}}\\
		&\leq \sup_{(\alpha,\psi)\in\overline{\D}_r\times\S} \left(1-\abs{\alpha}^2\right) \sum_{k=1}^{\infty} \abs{\alphabar\,\rme^{\rmi\psi}}^{kN-1} \\
		&\leq\sum_{k=1}^{\infty} r^{kN-1} = \mathcal{O}(r^{N-1}).
	\end{eqnarray*}
	so that \eref{eq:etaAbsBound} holds.
	
	To prove \eref{eq:etaDifBound}, we first compute the entries of $\tilde{\DD}\eta(\alpha,\psi,\vec{\ccrt})$ to leading order in $N$. From this, we can expand $\norm{\tilde{\DD}\eta(\alpha,\psi,\vec{\ccrt})}$ to leading order. For each $x\in\{ \alpha,\alphabar,\psi\}$, we write in a slight abuse of notation $\DD_x\eta = \Re \DD_x\eta+\rmi\,\Im\DD_x\eta$ for the respective column $(\Re \DD_x\eta,\Im\DD_x\eta)^T$ of $\DD\eta$.
	
	For the first column of $\DD\eta(\alpha,\psi,\vec{\ccrt})$, we find
	\begin{eqnarray*}
		\DD_{\alpha}\eta(\alpha,\psi,\vec{\ccrt}) &=
		\sum_{k=1}^{\infty} (\alphabar\,\rme^{\rmi\psi})^{kN} \\
		\abs{\DD_{\alpha}\eta(\alpha,\psi,\vec{\ccrt})}&\leq \sum_{k=1}^{\infty} \abs{\alpha}^{kN} \leq \sum_{k=1}^{\infty} r^{kN}
	\end{eqnarray*}
	for all $(\alpha,\psi)\in\overline{\D}_r\times\S$ and thus
	\begin{eqnarray*}
		\sup_{(\alpha,\psi)\in\overline{\D}_r\times\S}\abs{\Re\DD_{\alpha}\eta(\alpha,\psi,\vec{\ccrt})} &= \mathcal{O}(r^N) \\ \sup_{(\alpha,\psi)\in\overline{\D}_r\times\S}\abs{\Im\DD_{\alpha}\eta(\alpha,\psi,\vec{\ccrt})} &= \mathcal{O}(r^N).
	\end{eqnarray*}
	For the second column, we find
	\begin{eqnarray*}
		\DD_{\alphabar}\eta(\alpha,\psi,\vec{\ccrt}) &= 
		-\sum_{k=1}^{\infty} (kN-1)\,\alphabar^{kN-2}\,\rme^{\rmi kN\psi}
		+ \alpha\sum_{k=1}^{\infty} kN\,\alphabar^{kN-1}\,\rme^{\rmi kN\psi} \\
		\abs{\DD_{\alphabar}\eta(\alpha,\psi,\vec{\ccrt})}&\leq\; \sum_{k=1}^{\infty} \left[(kN-1)\abs{\alpha}^{kN-2} + kN\abs{\alpha}^{kN}\right] \\
		&\leq\; \sum_{k=1}^{\infty} \left[(kN-1)\,r^{kN-2}+kN\,r^{kN}\right]
	\end{eqnarray*}
	for all $(\alpha,\psi)\in\overline{\D}_r\times\S$ and
	\begin{eqnarray*}
		\sup_{(\alpha,\psi)\in\overline{\D}_r\times\S}\abs{\Re\DD_{\alphabar}\eta(\alpha,\psi,\vec{\ccrt})} &= \mathcal{O}(Nr^{N-2}) \\ \sup_{(\alpha,\psi)\in\overline{\D}_r\times\S}\abs{\Im\DD_{\alphabar}\eta(\alpha,\psi,\vec{\ccrt})} &= \mathcal{O}(Nr^{N-2}).
	\end{eqnarray*}
	Finally, we have
	\begin{eqnarray*}
		\DD_{\psi}\eta(\alpha,\psi,\vec{\ccrt})
		&= \sum_{k=1}^{\infty} \rmi kN\big[-\alphabar^{kN-1}+\alpha\alphabar^{kN}\big]\rme^{\rmi kN\psi} \\
		\abs{\DD_{\psi}\eta(\alpha,\psi,\vec{\ccrt})}
		&\leq \sum_{k=1}^{\infty} kN\big[ \abs{\alpha}^{kN-1} + \abs{\alpha}^{kN+1} \big] \\
		&\leq\sum_{k=1}^{\infty} kN\big[ r^{kN-1}+r^{kN+1}\big]
	\end{eqnarray*}
	for all $(\alpha,\psi)\in\overline{\D}_r\times\S$ so that
	\begin{eqnarray*}
		\sup_{(\alpha,\psi)\in\overline{\D}_r\times\S}\abs{\Re\DD_{\psi}\eta(\alpha,\psi,\vec{\ccrt})} &= \mathcal{O}(Nr^{N-1}) \\ \sup_{(\alpha,\psi)\in\overline{\D}_r\times\S}\abs{\Im\DD_{\psi}\eta(\alpha,\psi,\vec{\ccrt})} &= \mathcal{O}(Nr^{N-1}).
	\end{eqnarray*}
	Writing 
	\begin{equation*}
		\norm{A}_1 = \max_{1\leq j\leq n}\sum_{i=1}^m\abs{A_{ij}}
	\end{equation*}
	for the induced 1-norm for any $m\times n$ matrix $A$ and using the inequality $\norm{A}\leq\sqrt{n}\norm{A}_1$ for its Euclidean norm, we arrive at	
	\begin{eqnarray*}
		\norm{\tilde{\DD}\eta(\alpha,\psi,\vec{\ccrt})} \leq \sqrt{3} \max_{x\in\{\alpha,\alphabar,\psi\}}\left( \abs{\Re\tilde{\DD}_x\eta(\alpha,\psi,\vec{\ccrt})} + \abs{\Im\tilde{\DD}_x\eta(\alpha,\psi,\vec{\ccrt})}\right)
	\end{eqnarray*}
	and hence
	\begin{equation*}
	\norm{\tilde{\DD}\eta(\alpha,\psi,\vec{\ccrt})} = \mathcal{O}(Nr^{N-2}),
	\end{equation*}
	uniformly for all $(\alpha,\psi)\in\overline{\D}_r\times\S$ which proves \eref{eq:etaDifBound}.
\end{proof}

We are now ready to give the main result of this subsection which will be used in the proof of theorem~\ref{thm:ManifoldContinuum}.

\begin{lemma}\label{lm:VectorBounds}
	Let $0<r<1$ and consider the vector field $\vec{\tau}:\D\times\S\times\OCr\to\mathbb{R}^3$, defined by
	\begin{equation}\label{eq:VectorFieldOnLevelSet}
		\vec{\tau}(\alpha,\psi,\vec{\ccr}) = \left(G(Z(\alpha,\psi,\vec{\ccr}))-G(\alpha)\right)\cdot \vec{Y}(\alpha,\psi)
	\end{equation}
	where $G:\D\to\mathbb{C}$ and $\vec{Y}:\D\times\S\to\mathbb{R}^3$ are smooth. Then, given $\varepsilon>0$, there exists an $N_0\in\mathbb{N}$ such that for all $N\geq N_0$, there exists a $\delta$-neighborhood $\OCr_{\delta}(\vec{\ccrt})$ of $\vec{\ccrt}$ with
	\begin{eqnarray*}
		\sup_{(\alpha,\psi)\in\overline{\D}_r\times\S}\norm{\vec{\tau}(\alpha,\psi,\vec{\ccr})} +\norm{\tilde{\DD}\vec{\tau}(\alpha,\psi,\vec{\ccr})} &<\varepsilon
	\end{eqnarray*}
	for all $\vec{\ccr}\in\OCr_{\delta}(\vec{\ccrt})$.
\end{lemma}
\begin{proof}
	\textbf{Step 1:} Let $F:\D\to\mathbb{C}$ be defined as $F(Z)\defined G(Z)-G(\alpha)$. We start by bounding the function $F$ and its partial derivatives $\DD_xF$ with $x\in\{\alpha,\alphabar,\psi\}$. Since $\overline{\D}_r$ is convex and $F(\alpha)=0$, we find by the mean value theorem and with $\eta=Z-\alpha$ that
	\begin{equation}\label{eq:MeanFZ}
	\eqalign{
		F(Z) &= \int_0^1 \eta\cdot\DD_zF|_{z=\alpha+t\eta} + \bar{\eta}\cdot\DD_{\bar{z}}F|_{z=\alpha+t\eta}\dd{t} \\
		\abs{F(Z)} &\leq \abs{\eta}\cdot\int_0^1\abs{\DD_zF}_{z=\alpha+t\eta} + \abs{\DD_{\bar{z}}F}_{z=\alpha+t\eta}\dd{t}.
	}
	\end{equation}
	Since $F$ is smooth and $\overline{\D}_r$ is compact, there then exist constants $M_1,M_2>0$ such that $\abs{\DD_z F}<M_1$ and $\abs{\DD_{\bar{z}} F}<M_2$ for all $z\in\overline{\D}_r$ so that in particular, the integral above is of order $\mathcal{O}(1)$ and
	\begin{equation*}
		\abs{F(Z(\alpha,\psi,\vec{\ccrt}))} = \mathcal{O}(\abs{\eta(\alpha,\psi,\vec{\ccrt})}) = \mathcal{O}(r^{N-1})
	\end{equation*}
	for all $(\alpha,\psi)\in\overline{\D}_r\times\S$ by lemma~\ref{lm:MeanFieldApproximationtheta}. On the other hand, applying the product and chain rule to \eref{eq:MeanFZ}, we find
	\begin{eqnarray*}
		\DD_x F(Z) =&\quad \int_0^1 \DD_z\DD_z F|_{z=\alpha+t\eta}\,\DD_x(\alpha+t\eta)\dd{t}\cdot\eta \\
		&+\int_0^1 \DD_{\bar{z}}\DD_z F|_{z=\alpha+t\eta}\,\DD_x(\alphabar+t\bar{\eta})\dd{t}\cdot\eta \\
		&+\int_0^1 \DD_z \DD_{\bar{z}}F|_{z=\alpha+t\eta}\,\DD_x(\alpha+t\eta)\dd{t}\cdot\bar{\eta} \\
		&+\int_0^1 \DD_{\bar{z}}\DD_{\bar{z}}F|_{z=\alpha+t\eta}\,\DD_x(\alphabar+t\bar{\eta})\dd{t}\cdot\bar{\eta} \\
		&+\int_0^1 \DD_z F|_{\alpha+t\eta}\dd{t} \cdot \DD_x\eta+\int_0^1 \DD_{\bar{z}}F|_{z=\alpha+t\eta}\dd{t}\cdot\DD_x\bar{\eta}
	\end{eqnarray*}
	so that by the same argument as above and again with lemma~\ref{lm:MeanFieldApproximationtheta}, we find with $\abs{\DD_x(\alpha+t\eta(\alpha,\psi,\vec{\ccrt}))} = \mathcal{O}(1)+\mathcal{O}(Nr^{N-2})=\mathcal{O}(1)$
	\begin{eqnarray*}
		\abs{\DD_xF(Z(\alpha,\psi,\vec{\ccrt}))} &\leq \abs{\eta}\cdot \int_0^1\Big[ \abs{\DD_z\DD_zF}\!\,_{z=\alpha+t\eta}\cdot\abs{\DD_x(\alpha+t\eta)} +\\ 
		&\hspace{1.75cm}+ \abs{\DD_{\bar{z}}\DD_z F}\!\,_{z=\alpha+t\eta}\cdot\abs{\DD_x(\alphabar+t\bar{\eta})} + \\
		&\hspace{1.75cm}+ \abs{\DD_z\DD_{\bar{z}}F}\!\,_{z=\alpha+t\eta}\cdot\abs{\DD_x(\alpha+t\eta)} +\\
		&\hspace{1.75cm}+ \abs{\DD_{\bar{z}}\DD_{\bar{z}}F}\!\,_{z=\alpha+t\eta}\cdot\abs{\DD_x(\alphabar+t\bar{\eta})}\Big]\dd{t}\;+ \\
		&\hspace{1.75cm}+\abs{\DD_x\eta}\cdot\int_0^1\abs{\DD_zF}\!\,_{z=\alpha+t\eta}+\abs{\DD_{\bar{z}}F}\!\,_{z=\alpha+t\eta}\dd{t} \\
		&= \mathcal{O}(\abs{\eta(\alpha,\psi,\vec{\ccrt})}) + \mathcal{O}(\abs{\DD_x\eta(\alpha,\psi,\vec{\ccrt})}) = \mathcal{O}(Nr^{N-2}).
	\end{eqnarray*}
	
	\textbf{Step 2:} Since $\vec{Y}(\alpha,\psi)$ does not depend on $N$, we have $\norm{\vec{Y}(\alpha,\psi)} = \mathcal{O}(1)$ and $\norm{\DD_x\vec{Y}(\alpha,\psi)}_1 = \mathcal{O}(1)$ for all $(\alpha,\psi)\in\overline{\D}_r\times\S$ so that
	\begin{equation*}
		\norm{\vec{\tau}(\alpha,\psi,\vec{\ccrt})} = \abs{F(Z(\alpha,\psi,\vec{\ccrt}))}\cdot\norm{\vec{Y}(\alpha,\psi)} = \mathcal{O}(r^{N-1})
	\end{equation*}
	and, by virtue of the triangle inequality and the previous step,
	\begin{eqnarray*}
		\fl\norm{\tilde{\DD}\vec{\tau}(\alpha,\psi,\vec{\ccrt})}_1 &= \max_{x\in\{\alpha,\alphabar,\psi\}} \norm{\DD_xF(Z(\alpha,\psi,\vec{\ccrt}))\cdot\vec{Y}(\alpha,\psi) + F(Z(\alpha,\psi,\vec{\ccrt}))\cdot\DD_x\vec{Y}(\alpha,\psi)}_1 \\
		&=\mathcal{O}(Nr^{N-2})
	\end{eqnarray*}
	holds so that
	\begin{equation*}
		\norm{\tilde{\DD}\vec{\tau}(\alpha,\psi,\vec{\ccrt})} \leq\sqrt{3}\norm{\tilde{\DD}\vec{\tau}(\alpha,\psi,\vec{\ccrt})}_1 = \mathcal{O}(Nr^{N-2})
	\end{equation*}
	for all $(\alpha,\psi)\in\overline{\D}_r\times\S$. 
	
	\textbf{Step 3:} From step 2, we infer that there exists an $N_0$ such that $\abs{\vec{\tau}(\alpha,\psi,\vec{\ccr})}<\varepsilon/2$ and $\norm{\tilde{\DD}\vec{\tau}(\alpha,\psi,\vec{\ccr})}<\varepsilon/2$, uniformly on $\overline{\D}_r\times\S\times\{\vec{\ccrt}\}$ for all $N\geq N_0$. By smoothness of $\vec{\tau}$, it follows that for each $(\alpha,\psi,\vec{\ccrt})\in \overline{\D}_r\times\S\times\{\vec{\ccrt}\}$, there exist open neighborhoods $W(\alpha,\psi)\subseteq\D\times\S$ of $(\alpha,\psi)$ and $V_{\delta'}(\vec{\ccrt})\subseteq\OCr$ of $\vec{\ccrt}$ with $\delta'=\delta'(\alpha,\psi)>0$ such that this inequality also holds on $W(\alpha,\psi)\times V_{\delta'}(\vec{\ccrt})$. Covering $\overline{\D}_r\times\S\times\{\vec{\ccrt}\}$ with these open sets, there exists a finite subcover since $\overline{\D}_r\times\S\times\{\vec{\ccrt}\}$ is compact. Setting $\delta\defined\min \delta'$ over this subcover,  we find 
	\begin{equation*}
		\abs{\vec{\tau}(\alpha,\psi,\vec{\ccr})}+ \norm{\tilde{\DD}\vec{\tau}(\alpha,\psi,\vec{\ccr})}<\varepsilon
	\end{equation*}
	for each $(\alpha,\psi,\vec{\ccr})\in \overline{\D}_r\times\S\times V_{\delta}$ which proves the assertion.
\end{proof}

\subsection{Existence of the NAIM}\label{sec:Continuum}

We are now ready to prove our first main result: the theorem on the existence of a continuum of periodic orbits, whose union forms a NAIM.

\begin{theorem}\label{thm:ManifoldContinuum} 
	Consider the system \eref{eq:WSFull}
	where $f:\D\to\mathbb{C}$ and $g:\D\to\mathbb{R}$ are smooth functions of the Kuramoto order parameter $Z$. Further, let the closed equation
	\begin{equation*}
		\dot{\alpha} = i \left( f(\alpha)\alpha^2 + g(\alpha)\alpha+ \fbar(\alpha)\right)
	\end{equation*}
	possess a unique stable fixed point $\alpha_0\in\D$ for which $f(\alpha_0)\neq0$.

	Then, there exists a $\delta$-neighborhood $\overline{V}_{\delta}(\vec{\ccrt})$ of $\vec{\ccrt}$ such that for sufficiently large $N$ and every $\vec{\ccr}\in \overline{V}_{\delta}(\vec{\ccrt})$, there exists a unique periodic orbit $\mathcal{C}_{\vec{\ccr}}\subset\level{\vec{\ccr}}$ which is exponentially stable in $\level{\vec{\ccr}}$. Additionally, the union
	\begin{equation*}
		\contF_{\delta} \defined \bigcup_{\vec{\ccr}\in \overline{\OCr}_{\delta}}\mathcal{C}_{\vec{\ccr}} \subset \OT
	\end{equation*}
	forms a compact NAIM of dimension $N-2$ with invariant boundary.
\end{theorem}
\begin{proof}
	\textbf{Step 1:} Substituting $Z$ by $\alpha$ in \eref{eq:WSFull} yields the closed system
	\begin{numparts}\label{eq:truncatedWS}
	\begin{eqnarray}
		\dot{\alpha} &= i \left( f(\alpha)\alpha^2 + g(\alpha)\alpha+ \fbar(\alpha) \right) \label{eq:WSTruncAlpha}\\
		\dot{\psi} &= f(\alpha)\alpha+ g(\alpha) + \fbar(\alpha)\alphabar \label{eq:WSTruncPsi}\\
		\dot{\vec{\ccr}} &= \vec{0}. \label{eq:WSTruncLambda}
	\end{eqnarray}
	\end{numparts}
	By assumption, $\alpha_0$ is a stable fixed point of \eref{eq:WSTruncAlpha}. If $\Omega \defined f(\alpha_0)\alpha_0+ g(\alpha_0) + \fbar(\alpha_0)\alphabar_0\neq0$, this gives rise to the periodic solution so $(\alpha_0,\Omega\, t,\vec{\ccr})$ for \eref{eq:WSTruncAlpha}-\eref{eq:WSTruncLambda} with an exponentially stable periodic orbit $\mathcal{C}^{\t}_{\vec{\ccr}}$. From the definition of $\Omega$, we then infer
	\begin{eqnarray*}
		\alpha_0\Omega &= f(\alpha_0)\alpha_0^2 + g(\alpha_0)\alpha_0+ \fbar(\alpha_0)\abs{\alpha_0}^2.
	\end{eqnarray*}
	Adding and subtracting $\fbar(\alpha_0)$, this results in
	\begin{equation*}
		\alpha_0 \Omega = \underbrace{f(\alpha_0)\alpha_0^2+g(\alpha_0)\alpha_0+\fbar(\alpha_0)}_{=0 \textrm{ by \eref{eq:WSTruncAlpha}}} + \fbar(\alpha_0)\left(\abs{\alpha_0}^2-1\right)
	\end{equation*}
	so that $\Omega=0$ implies $f(\alpha_0)=0$ and conversely, $f(\alpha_0)\neq0$ implies $\Omega\neq0$.
	
	\textbf{Step 2:} We prove that for any $\delta>0$, the union 
	\begin{equation*}
		\contF^{\t}_{\delta}\defined\bigcup_{\vec{\ccr}\in \overline{\OCr}_{\delta}}\mathcal{C}^{\t}_{\vec{\ccr}}=\{\alpha_0\}\times\S\times\overline{\OCr}_{\delta}
	\end{equation*}
	forms a smooth compact NAIM with invariant boundary for \eref{eq:truncatedWS} in $\D\times\S\times \OCr$ by checking conditions (i)-(iii) of definition~\ref{def:NAIM}. As a product of the three smooth compact manifolds $\{\alpha_0\}$,  $\S$, and $\overline{\OCr}_{\delta}$, it is itself an $(N-2)$-dimensional smooth compact submanifold of $\D\times\S\times \OCr$. It is invariant by construction so that condition (i) is fulfilled.
	
	Let $\vec{p}=(\alpha_0,\psi,\vec{\ccr})\in\contF^{\t}_{\delta}$ denote any point on $\contF$ and consider any tangential vector $\vec{v}\in T_{\vec{p}}\OT$ at $\vec{p}$ with $\vec{v}=(v_{\alpha},v_{\psi},v_{\vec{\ccr}})^T$ where $v_{\alpha}\in \mathbb{R}^2$, $v_{\psi}\in\mathbb{R}^1$, and $v_{\vec{\ccr}}\in\mathbb{R}^{N-3}$ denote the respective tangential components with respect to the variables $\alpha$, $\psi$, and $\vec{\ccr}$. In particular, since the dynamics of $\alpha$ decouples from $\psi$ and $\vec{\ccr}$ in \eref{eq:truncatedWS}, the dynamics of $v_{\alpha}$ decouples from $v_{\psi}$ and $v_{\vec{\ccr}}$ on $\contF_{\delta}^{\t}$ so that $\mathcal{N}$ stays invariant under the linearized flow $\DD\Phi^t$. Then, $T\contF_{\delta}^{\t}$ also stays invariant and (ii) is fulfilled. The flow is not contracting in $v_{\psi}$ and $v_{\vec{\ccr}}$ since $\vec{\ccr}$ stays constant and $\psi=\Omega t$ so that we have $b=0$ while it is exponentially contracting in normal direction $v_{\alpha}$ (with contraction rates $\ev^{\pm}<0$ which are the eigenvalues of the Jacobian for \eref{eq:WSTruncAlpha}. Let $a=\max(\ev^+,\ev^-)<0$. Then $a<0=bm$ for $m>0$ and (iii) is also fulfilled and $\contF_{\delta}^{\t}$ is a NAIM.
		
	\textbf{Step 3:} We proceed by showing that \eref{eq:WSFull} equally possesses a continuous family of periodic orbits. By closing \eref{eq:WSFull} to get equations \eref{eq:truncatedWS}, we truncated the vector field 
	\begin{equation*}
		\vec{\tau} = 
		F(Z) \pmatrix{ \rmi\alpha^2 \cr \alpha \cr \vec{0}}
		+G(Z)\pmatrix{ \rmi\alpha \cr 1 \cr \vec{0}}
		+\bar{F}(Z)\pmatrix{ \rmi \cr \alphabar \cr \vec{0}}
	\end{equation*}
	with
	\begin{eqnarray*}
		F(Z) &\defined f(Z)-f(\alpha) \\
		G(Z) &\defined g(Z)-g(\alpha)
	\end{eqnarray*}
	from \eref{eq:WSFull}. Identifying $\level{\vec{\ccr}}\cong\D\times\S\times\{\vec{\ccr}\}$ with the space $\D\times\S$, we truncate
	\begin{equation}\label{eq:WSTruncTerms2}
		\vec{\tau}|_{\level{\vec{\ccr}}} = 
		F(Z) \pmatrix{ \rmi\alpha^2 \cr \alpha}
		+G(Z)\pmatrix{ \rmi\alpha \cr 1 }
		+\bar{F}(Z)\pmatrix{ \rmi \cr \alphabar }
	\end{equation}
	in every level set $\level{\vec{\ccr}}$ for any given $\vec{\ccr}\in\OCr$ where each of the three terms on the right hand side is of the form \eref{eq:VectorFieldOnLevelSet}. Choose $1>r>\abs{\alpha_0}$. In $\D\times\S$, we can identify all $\mathcal{C}_{\vec{\ccr}}^{\t}$ with the stable limit cycle $\{\alpha_0\}\times\S\subset\D_r\times\S$, so there exists an $\varepsilon>0$ such that for every perturbation of the vector field that has $C^1$-norm smaller than $\varepsilon$ in $\D_r\times\S$, the orbits $\mathcal{C}_{\vec{\ccr}}^{\t}$ persist \cite{Shilnikov2001}. But by lemma~\ref{lm:VectorBounds}, there exists an $N_0\in\mathbb{N}$ and a $\delta$-neighborhood $\overline{\OCr}_{\delta}(\vec{\ccrt})$ of $\vec{\ccrt}$ such that for all $N\geq N_0$
	\begin{equation*}
		\abs{\vec{\tau}(\alpha,\psi,\vec{\ccr})}+\norm{\tilde{\DD}\vec{\tau}(\alpha,\psi,\vec{\ccr})} < \varepsilon
	\end{equation*}
	uniformly on $\D_r\times\S\times\overline{\OCr}_{\delta}$ and thus, the $\mathcal{C}_{\vec{\ccr}}^{\t}$ persist, \ie, the full system \eref{eq:WSFull} possesses a periodic orbit $\mathcal{C}_{\vec{\ccr}}\subset\D_r\times\S\times\{\vec{\ccr}\}\subset\level{\vec{\ccr}}$ for every $\vec{\ccr}\in\overline{\OCr}_{\delta}$. 
	
	We can further choose $\delta$ such that the boundary of the union $\contF_{\delta}$ over these orbits is composed of orbits $\mathcal{C}_{\vec{\ccr}}$ with $\vec{\ccr}\in\partial\overline{V}_{\delta}$ and is invariant. We show that $\contF_{\delta}$ is a smooth compact manifold.
	
	\textbf{Step 4:} Note that for $\vec{\ccr}\in\overline{\OCr}_{\delta}$, there exists a smooth immersion $\iota_{\vec{\ccr}}:\mathcal{C}_{\vec{\ccr}}^{\t}\to\D_r\times\S\times\{\vec{\ccr}\}$ whose image $\iota_{\vec{\ccr}}(\mathcal{C}_{\vec{\ccr}}^{\t})=\mathcal{C}_{\vec{\ccr}}$ lies $\varepsilon$-close to $\mathcal{C}_{\vec{\ccr}}^{\t}$ and that $\iota_{\vec{\ccr}}$ depends smoothly on $\vec{\ccr}$ because the truncated terms \eref{eq:WSTruncTerms2} are smooth in $\vec{\ccr}$. Hence, the map
	\begin{eqnarray*}
		\iota:\contF_{\delta}^{\t}&\to\D\times\S\times\OCr \\
		\iota:(\alpha_0,\psi,\vec{\ccr})&\mapsto(\iota_{\vec{\ccr}}(\alpha_0,\psi),\vec{\ccr}),
	\end{eqnarray*}
	is (i) smooth, (ii) one-to-one on its image $\contF_{\delta}$, and (iii) its derivative has full rank $N-2$, in other words, $\iota$ is a smooth immersion and $\contF_{\delta}$ is a compact smooth invariant manifold which is $\mathcal{O}(\varepsilon)$-close to $\contF_{\delta}^{\t}$.
	
	\textbf{Step 5:} We show that $\contF_{\delta}$ is normally hyperbolic: 
	For any fixed $\vec{\ccr}\in\OCr$, let $\Phi^t_{\vec{\ccr}}$ denote the flow on the level set $\level{\vec{\ccr}}$. Then, the flow on $\OT$ is of the form
	\begin{eqnarray*}
		(\alpha,\psi,\vec{\ccr})\mapsto(\Phi^t_{\vec{\ccr}}(\alpha,\psi),\vec{\ccr})
	\end{eqnarray*}
	and is smooth in $\vec{\ccr}$ since the vector field on $\OT$ is smooth. The linearized flow at $\vec{p}=(\alpha,\psi,\vec{\ccr})\in\mathcal{C}_{\vec{\ccr}}\subset\contF_{\delta}\subset\OT$ reads
	\begin{equation*}
		\DD\Phi^t(\vec{p}) = 
		\pmatrix{
		\tilde{\DD}\Phi_{\vec{\ccr}}^t(\vec{p}) & \DD_{\vec{\ccr}}\Phi_{\vec{\ccr}}^t(\vec{p}) \cr
		0 & \id_{N-3}}.
	\end{equation*}
	Let $\mu^{\pm}<0$ denote the two nonzero contraction rates of $\mathcal{C}_{\vec{\ccr}}\subset\level{\vec{\ccr}}$ and $\vec{v}^{\pm}=(v^{\pm}_{\alpha},v^{\pm}_{\psi},\vec{0})$ denote the corresponding eigenvectors of $\DD\Phi^t(\vec{p})$ and let $\mu^0=0$ denote the vanishing contraction rate in tangent direction to $\mathcal{C}_{\vec{\ccr}}$. The remaining $N-3$ eigenvectors of $\DD\Phi^t(\vec{p})$ are also tangent vectors of $\contF_{\delta}$ at $\vec{p}$ and have nonvanishing components in $\vec{\ccr}$-direction since $\contF_{\delta}$ lies transversal to each $\level{\vec{\ccr}}$ that it intersects for sufficiently small $\varepsilon>0$. Since $\vec{\ccr}$ is constant under the flow, the contraction rates in the remaining $N-3$ tangent directions are also zero. Because $\Phi^t$ is smooth in $\vec{\ccr}$, the $\vec{v}^{\pm}$ depend smoothly on $\vec{p}$ so that condition (ii) of definition~\ref{def:NAIM} is readily fulfilled. For the numbers $a$, $b$, and $C$ from condition (iii), we find $b=0$, $a=\max_{\vec{p}\in\contF_{\delta}}(\mu^{\pm}(\vec{p}))<0 = m b$ for $m>0$ and $C=1$ so that $\contF_{\delta}$ is a smooth $m$-normally attracting invariant manifold of \eref{eq:WSFull}.
\end{proof}

We note that $\contF_{\delta}$ is of the form
\begin{equation*}
	\contF_{\delta} = \set{	\alpha(\psi,\vec{\ccr}),\psi,\vec{\ccr}}{\psi\in\S, \vec{\ccr}\in\overline{\OCr}_{\delta}}
\end{equation*}
where $\norm{\alpha(\psi,\vec{\ccr})-\alpha_0}_{C^1}\ll1$ uniformly on $\S\times\overline{\OCr}_{\delta}$.

\subsection{Existence of splay states}

Next, we study the dynamics on the periodic orbit $\mathcal{C}_{\vec{\ccrt}}\subset\level{\vec{\ccrt}}$ of \eref{eq:GeneralSystem}, characterized in theorem~\ref{thm:ManifoldContinuum}. Recall that a \emph{splay state} $\vec{\phi}(t)$ is a $T$-periodic solution for which there exists a $T$-periodic function $\varphi:\mathbb{R}\to\S$ such that
\begin{equation*}
	\phi_j(t) = \varphi\left(t+j\frac{T}{N}\right)
\end{equation*}
holds, \ie, the time series for each of the units $\phi_j$ are copies of each other, shifted multiples of $T/N$ in time. The following assertion holds.
\begin{proposition}\label{prop:SplayState}
	For the continuum $\contF_{\delta}$ from theorem~\ref{thm:ManifoldContinuum}, the periodic solution $(\alpha(t),\psi(t),\vec{\ccrt})$ of \eref{eq:WSFull} with periodic orbit $\mathcal{C}_{\vec{\ccrt}}\subset\level{\vec{\ccrt}}$ yields a splay state $\vec{\phi}(t)$ of \eref{eq:GeneralSystem}.
\end{proposition}
\begin{proof}
	Recall that $\thetat_j = -\pi + 2\pi(j-1)/N$. 	For the closed system \eref{eq:truncatedWS}, we found the periodic solution $(\alpha_0,\Omega\,t, \vec{\ccrt})$ with period $T=2\pi/\Omega$. Its phase dynamics $\vec{\phi}^{\t}(t)$ can via the diffeomorphism $\m$ which yields
	\begin{eqnarray*}
		\rme^{\rmi\phi^{\t}_j(t)} \equiv \frac{\alpha_0 + \rme^{\rmi\Omega t + i\thetat_j}}{1 + \alphabar_0\, \rme^{\rmi\Omega t + i\thetat_j}}
		= \frac{\alpha_0 + \rme^{\rmi\Omega\left(t + j\frac{T}{N}\right)+i\thetat_N}}{1 + \alphabar_0\, \rme^{\rmi\Omega\left(t + j\frac{T}{N}\right)+i\thetat_N}}
		\equiv \rme^{\rmi\phi^{\t}_N\left(t+j\frac{T}{N}\right)}
	\end{eqnarray*}
	so that $\phi^{\t}_j(t) = \phi^{\t}_N\left(t + jT/N\right)$ holds and $\vec{\phi}^{\t}(t)$ is in fact a splay state (with $\varphi=\phi_N$). We now assert that the periodic solution $(\alpha(t),\psi(t),\vec{\ccrt})$ for the true system \eref{eq:WSFull} from theorem~\ref{thm:ManifoldContinuum} also yields a splay state $\vec{\phi}(t)$ of \eref{eq:GeneralSystem}. For this to be true, the following condition must hold:
	\begin{equation*}
		\rme^{\rmi\phi_j(t)} \equiv 	\frac{\alpha(t)+\rme^{\rmi\psi(t)+i\thetat_j}}{1+\alphabar(t)\rme^{\rmi\psi(t)+i\thetat_j}} =
		\frac{\alpha\left(t+j\frac{T}{N}\right)+\rme^{\rmi\psi\left(t+j\frac{T}{N}\right)+i\thetat_N}}{1+\alphabar\left(t+j\frac{T}{N}\right)\rme^{\rmi\psi\left(t+j\frac{T}{N}\right)+i\thetat_N}} \equiv \rme^{\rmi\phi_N\left(t+j\frac{T}{N}\right)}
	\end{equation*}
	which is true if the spatio-temporal symmetry
	\begin{eqnarray*}
		\alpha\left(t+j\frac{T}{N}\right) = \alpha(t), \qquad
		\psi\left(t+j\frac{T}{N}\right) = \psi(t) + j\frac{2\pi}{N}
	\end{eqnarray*}
	with $j=1,\dots,N$ holds for $(\alpha(t),\psi(t),\vec{\ccrt})$.
	To see that this is the case, we note first that \eref{eq:alphaDot} and \eref{eq:psiDot} are equivariant under any transformation that keeps $Z$ and $\alpha$ invariant. On $\level{\vec{\ccrt}}$, the Kuramoto order parameter is given by
	\begin{equation*}
		Z(\alpha,\psi,\vec{\ccrt}) = \alpha- \big(1-\abs{\alpha}^2\big)\sum_{k=1}^{\infty} (\alphabar)^{kN-1}\rme^{\rmi kN\psi}
	\end{equation*}
	so that on $\level{\vec{\ccrt}}$, $Z(\alpha,\psi,\vec{\ccrt})$ and $\alpha$ are invariant under the action of the finite group $\Gamma$ of transformations
	\begin{eqnarray*}
		(\alpha,\psi) &\mapsto \left(\alpha, \psi + j\frac{2\pi}{N}\right)
	\end{eqnarray*}
	with $j=1,\dots,N$. The splay state nature of the $\vec{\phi}(t)$ for the full system is then a consequence of the fact that for hyperbolic orbits, spatio-temporal symmetries are robust under perturbations that leave the system equivariant under the finite group action of the symmetry group of the unperturbed system \cite{GolubitskyStewart2003}.
\end{proof}

In the next section, we discuss the introduction of symmetry-breaking perturbations and how averaging theory can be employed to approximate the dynamics on $\contF_{\delta}$ for such perturbations.

\section{Introducing perturbations}\label{sec:Averaging}

We consider now model \eref{eq:GeneralSystemPerturbed} where we assume that $h$ is smooth and normalized:
\begin{equation*}
	\norm{h}_{L^2} \defined \sqrt{\sum_{n=2}^{\infty} \abs{a_n}^2+\abs{b_{n}}^2} = 1.
\end{equation*}

\subsection{Averaging principle for WS-integrable systems}

The techniques that we employed to prove theorem~\ref{thm:ManifoldContinuum} allow the construction of a NAIM for cross-ratios near $\vec{\ccrt}$. However, normally attracting manifolds can be extended and numerical results indicate that for the classic active rotator model \eref{eq:ActiveRotatorsClassic}, this extension is ``large'' in the sense that periodic orbits exist for all $\vec{\ccr}\in\OCr$ so that $\contF$ intersects every $\level{\vec{\ccr}}$ \cite{Zaks_Tomov_2016,Ronge_Zaks_2021_2}. From now on, we drop the index $\delta$ and consider the largest possible extension $\contF$ of the NAIM of theorem~\ref{thm:ManifoldContinuum} such that $\contF$ still consists of periodic orbits $\mathcal{C}_{\vec{\ccr}}\subset\level{\vec{\ccr}}$, exponentially stable in $\level{\vec{\ccr}}$ and neutrally stable in normal direction to $\level{\vec{\ccr}}$. The following results apply as well but not exclusively to $\contF_{\delta}$ from theorem~\ref{thm:ManifoldContinuum}.

In a slight abuse of notation, we denote by $\contF_{\epsilon}$ the perturbed NAIM for \eref{eq:GeneralSystemPerturbed} with $\epsilon\neq0$ and the original NAIM for \eref{eq:GeneralSystem} as $\contF$. Our main objective is to determine what happens to the infinitely many periodic orbits when $\epsilon\neq0$. To determine robust periodic orbits and their stability, we develop a criterion based on averaging theory \cite{Sanders_Verhulst_Murdock_2007}.

Let $\n=(n_{\alpha},n_{\psi},n_{\vec{\ccr}})=\m^{-1}$ denote the inverse of the smooth diffeomorphism from lemma~\ref{lm:WSChart}. In particular, we have $n_{\ccr}=\vec{\Cr}$. For definiteness, we write
\begin{equation*}
	\vec{h}(\vec{\phi})\defined(h(\phi_1),\dots,h(\phi_N))
\end{equation*}
for the diagonal action of $h$ on $\T^N$. Form the chain rule, we conclude that for every $x\in\{\alpha,\psi,\vec{\ccr}\}$, we have
\begin{equation*}
	\dot{x}=\DD n_x\cdot\dot{\vec{\phi}} = \left(\DD n_x\cdot \vec{h}\right)(\vec{\phi}) \defined \sum_{j=1}^N h(\phi_j)\cdot\DD n_x(\vec{\phi})
\end{equation*}
at the point $\vec{\phi}$ and equation \eref{eq:GeneralSystemPerturbed} can be written in WS-variables as
\begin{eqnarray*}
	\dot{\alpha} &= \rmi\left(f(Z)\alpha^2+g(Z)\alpha+\fbar(Z)\right)  + \epsilon \left(\DD n_{\alpha} \cdot \vec{h} \right) \circ\m(\alpha,\psi,\vec{\ccr}) \\
	\dot{\psi} &= \left(f(Z)\alpha+ g(Z) + \fbar(Z)\alphabar\right) + \epsilon \left(\DD n_{\psi}\cdot \vec{h}\right)\circ\m(\alpha,\psi,\vec{\ccr})\\
	\dot{\vec{\ccr}} &= \epsilon\left(\DD\vec{\Cr}\cdot \vec{h}\right)\circ\m(\alpha,\psi,\vec{\ccr}).
\end{eqnarray*}
We treat the $\mathcal{O}(\epsilon)$-terms above as a perturbation of the WS-integrable system \eref{eq:WSFull}. The perturbation terms in the first two equations are unproblematic in terms of WS-theory since they leave the level sets invariant. Only the last equation makes the system nonintegrable by making the level sets $\level{\vec{\ccr}}$ noninvariant.

The following theorem constitutes our second main result.

\begin{theorem}\label{prop:Averaging}
	For fixed $N$, consider the system \eref{eq:GeneralSystemPerturbed}
	for $\vec{\phi}=(\phi_1,\dots,\phi_N)\in\OT$ where $h$ is smooth and $\norm{h}_{L^2}=1$. For $\epsilon=0$, let
	\begin{equation*}
		\contF = \bigcup_{\vec{\ccr}\in W} \mathcal{C}_{\vec{\ccr}}
	\end{equation*}
	denote the largest extension of the NAIM from theorem~\ref{thm:ManifoldContinuum} in $\OT$ such that $W$ is the largest open set for which for every $\vec{\ccr}\in W\subset\OCr$, there exists a periodic orbit $\mathcal{C}_{\vec{\ccr}}$ which is exponentially stable in $\level{\vec{\ccr}}$ and $\contF$ is a NAIM. For every $\vec{\ccr}\in W$, fix a $\vec{\theta}=\vec{\theta}(\vec{\ccr})\in\mathcal{C}_{\vec{\ccr}}$ and let $\vec{\phi}_{\vec{\ccr}}(t)$ denote the solution of \eref{eq:GeneralSystem} with initial condition $\vec{\phi}_{\vec{\ccr}}(0)=\vec{\theta}$. Then, the following statements hold true:
	\begin{enumerate}
		\item The function
		\begin{equation}\label{eq:AveragedF}
			\vec{F}_h(\vec{\ccr}) \defined \frac{1}{T(\vec{\ccr})} \int_0^{T(\vec{\ccr})} \left(\DD\vec{\Cr}\cdot\vec{h}\right)\circ\vec{\phi}_{\vec{\ccr}}(t) \dd{t},
		\end{equation}
		where $T(\vec{\ccr})$ is the period of $\vec{\phi}_{\vec{\ccr}}$, is continuously differentiable, well-defined, and independent of the choice of $\vec{\theta}(\vec{\ccr})\in\mathcal{C}_{\vec{\ccr}}$.
		
		\item For $N=4$, the function $F_h$ fulfills
		\begin{equation*}
			F_{h}(\ccr)=-F_h(1-\ccr)
		\end{equation*}
		and in particular, we have $F_h(\ccrt)=0$ with $\ccrt=1/2$.
		
		\item For any $N\geq4$,
		\begin{equation*}
			\vec{F}_h(\vec{\ccrt}) = \vec{0}.
		\end{equation*}
	\end{enumerate}
\end{theorem}

We discuss the significance of the function $\vec{F}_h$ for studying the asymptotic dynamics of the perturbed system \eref{eq:GeneralSystemPerturbed} and the implications from statements 2. and 3. above for the existence of robust splay states in it in the supplementary material.

\subsection{Proof of Theorem~\ref{prop:Averaging}}

\subsubsection{Proof of statement 1.}

Let $\vec{\phi}_{\vec{\ccr}}$ be the periodic solution of \eref{eq:GeneralSystemPerturbed} with period $T(\vec{\ccr})>0$ and initial condition $\vec{\phi}_{\vec{\ccr}}(0)=\vec{\theta}=\vec{\theta}(\vec{\ccr})\in\mathcal{C}_{\vec{\ccr}}$. Then, the average of $\DD\vec{\Cr}\cdot\vec{h}$ over $\mathcal{C}_{\vec{\ccr}}$ in the right hand side of \eref{eq:AveragedF} exists. For any $\vec{\theta}'\in\mathcal{C}_{\vec{\ccr}}$, let $(\alpha'(t),\psi'(t),\vec{\ccr})$ be the solution of \eref{eq:WSFull} with $\m(\alpha'(0),\psi'(0),\vec{\ccr})=\vec{\theta}'$ and set $\vec{\phi}_{\vec{\ccr}}'(t)\defined\m(\alpha'(t),\psi'(t),\vec{\ccr})$. Then, there exists a $\tau=\tau(\vec{\theta}')$ such that $\vec{\phi}_{\vec{\ccr}}'(t) = \vec{\phi}_{\vec{\ccr}}(t+\tau)$. But since we average over a full period, the integral is invariant under any shift $\tau$ and thus is independent of the choice of $\vec{\theta}$. $\vec{F}_h$ is continuously differentiable because both $\vec{\Cr}$ and $\vec{h}$ are smooth and $\vec{\phi}_{\vec{\ccr}}$ and $T(\vec{\ccr})$ depend smoothly on the system parameter $\vec{\ccr}$ \cite{Shilnikov2001}.

\subsubsection{Proof of statement 2.}
\begin{proof}
	For $N=4$, the level sets of $\Cr$ in $\OT$ are parameterized by a single cross-ratio $\Cr(\vec{\theta})\defined\Cr_1(\vec{\theta})\defined\Cr_{1,2,3,4}(\vec{\theta})$. Consider the cyclic permutation
	\begin{equation*}
		\sigma = \pmatrix{ 
		1&2&3&4 \cr 
		4&1&2&3 
		}
	\end{equation*}
	and set for any $\vec{\theta}=(\theta^1,\theta^2,\theta^3,\theta^4)\in\OT$
	\begin{eqnarray*}
		\sigma\vec{\theta} = (\theta^{\sigma(1)},\theta^{\sigma(2)},\theta^{\sigma(3)},\theta^
		{\sigma(4)}) = (\theta^4,\theta^1,\theta^2,\theta^3) \in \OT.
	\end{eqnarray*}
	For fixed $\vec{\theta}=\vec{\theta}(\ccr)\in\mathcal{C}_{\ccr}$, let
	\begin{equation*}
		\vec{\phi}_{\ccr}(t,\vec{\theta}) =  \left(\phi_{\ccr}^1(t,\vec{\theta}),\phi_{\ccr}^2(t,\vec{\theta}),\phi_{\ccr}^3(t,\vec{\theta}),\phi_{\ccr}^4(t,\vec{\theta}) \right)
	\end{equation*}
	denote the solution of the unperturbed system \eref{eq:GeneralSystem} with initial condition $\vec{\phi}_{\ccr}(0)=\vec{\theta}$. In the same spirit as above, we set
	\begin{eqnarray*}
		\sigma\vec{\phi}_{\ccr}(t,\sigma\vec{\theta}) = \left(\phi_{\ccr}^4(t,\vec{\theta}),\phi_{\ccr}^1(t,\vec{\theta}),\phi_{\ccr}^2(t,\vec{\theta}),\phi_{\ccr}^3(t,\vec{\theta}) \right) 
	\end{eqnarray*}
	for the cyclic permutation of $\vec{\phi}_{\ccr}(t,\vec{\theta})$ which again is a solution of \eref{eq:GeneralSystem} since the functions $f$ and $g$ in it depend solely on $Z$ which is invariant under permutations of phases. The periodic orbit of this new solution does in general not lie in $\mathcal{L}_{\ccr}(\Cr)$ because permutations of the components of any $\vec{\theta}$ transform the cross-ratios $\Cr(\vec{\theta})$ \cite{Ahlfors_1953}. Indeed, for $\vec{\theta}\in\mathcal{L}_{\ccr}(\Cr)$, we have $\sigma\vec{\theta}\in\mathcal{L}_{1-\ccr}(\Cr)$ because
	\begin{equation}\label{eq:CrossRatioTransform}
		\eqalign{
		\Cr(\sigma\vec{\theta}) &= \Cr_{1,2,3,4}(\sigma\vec{\theta}) \\
		&= \frac{\big(\rme^{\rmi \theta_4}-\rme^{\rmi \theta_3}\big)\big(\rme^{\rmi \theta_1}-\rme^{\rmi \theta_2}\big)}{\big(\rme^{\rmi \theta_1}-\rme^{\rmi \theta_3}\big)\big(\rme^{\rmi \theta_4}-\rme^{\rmi \theta_2}\big)} \\
		&= \Cr_{4,1,2,3}(\vec{\theta}) \equiv 1-\Cr_{1,2,3,4}(\vec{\theta}) \\
		&= 1-\Cr(\vec{\theta}).
	}
	\end{equation}
	Since  $\sigma\vec{\phi}_{\ccr}(0,\vec{\theta})=\sigma\vec{\theta}=\vec{\phi}_{1-\ccr}(0,\sigma\vec{\theta})$, we have
	\begin{equation}\label{eq:EquivalentOrbits}
		\sigma\vec{\phi}_{\ccr}(t,\vec{\theta}) = \vec{\phi}_{1-\ccr}(t,\sigma\vec{\theta})
	\end{equation}
	by uniqueness of solutions and it follows that the periodic orbit of $\sigma\vec{\phi}_{\ccr}$ is $\mathcal{C}_{1-\ccr}$. Note that from this we also read
	\begin{equation}\label{eq:Period}
		T(\vec{\ccr})=T(1-\vec{\ccr}),
	\end{equation}
	\ie, the orbits $\mathcal{C}_{\ccr}$ and $\mathcal{C}_{1-\ccr}$ have the same period. From \eref{eq:CrossRatioTransform}, we read
	\begin{eqnarray*}
		\left(\DD\Cr_{1,2,3,4}\cdot \vec{h}\right)(\sigma\vec{\theta}) &= \DD_{\vartheta_1}\Cr_{1,2,3,4}(\vec{\vartheta})\cdot h(\vartheta_1) + \DD_{\vartheta_2}\Cr_{1,2,3,4}(\vec{\vartheta})\cdot h(\vartheta_2) + \\ &\quad+\DD_{\vartheta_3}\Cr_{1,2,3,4}(\vec{\vartheta})\cdot h(\vartheta_3) + \DD_{\vartheta_4}\Cr_{1,2,3,4}(\vec{\vartheta})\cdot h(\vartheta_4)\Big|_{\vec{\vartheta}=\sigma\vec{\theta}} \\
		&=\DD_{\theta_4}\Cr_{4,1,2,3}(\vec{\theta})\cdot h(\theta_4) + \DD_{\theta_1}\Cr_{4,1,2,3}(\vec{\theta})\cdot h(\theta_1) + \\ &\quad+\DD_{\theta_2}\Cr_{4,1,2,3}(\vec{\theta})\cdot h(\theta_2) + \DD_{\theta_3}\Cr_{4,1,2,3}(\vec{\theta})\cdot h(\theta_3) \\
		&=\;\left( \DD\Cr_{4,1,2,3}\cdot\vec{h} \right)(\vec{\theta})
	\end{eqnarray*}
	and find, using \eref{eq:EquivalentOrbits}, \eref{eq:Period}, and \eref{eq:CrossRatioTransform},
	\begin{eqnarray*}
		F_h(1-\ccr) &= \frac{1}{T(1-\ccr)} \int_0^{T(1-\ccr)}\left( \DD\Cr_{1,2,3,4}\cdot \vec{h} \right)\circ\vec{\phi}_{1-\ccr}(t,\sigma\vec{\theta})\dd{t} \\
		&= \frac{1}{T(\ccr)}\int_0^{T(\ccr)} \left( \DD\Cr_{1,2,3,4}\cdot \vec{h} \right)\circ\sigma\vec{\phi}_{\ccr}(t,\vec{\theta})\dd{t} \\
		&=\frac{1}{T(\ccr)}\int_0^{T(\ccr)}\left(\DD\Cr_{4,1,2,3}\cdot \vec{h} \right)\circ\vec{\phi}_{\ccr}(t,\vec{\theta}) \dd{t} \\
		&=-\frac{1}{T(\ccr)}\int_0^{T(\ccr)}\left(\DD\Cr_{1,2,3,4}\cdot \vec{h}\right)\circ\vec{\phi}_{\ccr}(t,\vec{\theta})\dd{t} \\
		&=-F_h(\ccr)
	\end{eqnarray*}
	which in particular yields
	\begin{equation*}
		F_h(1/2) = 0.
	\end{equation*}
	Since $\ccrt=1/2$ by \eref{eq:DefLambdaTilde} for $N=4$, $F_h$ indeed vanishes at $\ccrt$.
\end{proof}

\subsubsection{Proof of statement 3.}

Before we come to the case of general $N$, it is instructional to show for $N=4$ again, that $F_h(\ccrt)=\vec{0}$. This is done through the following calculation, using the facts that the average over a full period of $\mathcal{C}_{\ccrt}$ is invariant under shifts in time, that any splay state is of the form $\phi_j(t)=\varphi(t+jT/N)$ with \emph{some} $T$-periodic function $\varphi$ such that for a splay, a shift in time by $T(\ccrt)/N$ is equivalent to a cyclic permutation of units, and equation \eref{eq:CrossRatioTransform}:
\begin{eqnarray*}
	F_h(\ccrt) &= \frac{1}{T(\ccrt)}\int_0^{T(\ccrt)} \left(\DD\Cr_{1,2,3,4}\cdot \vec{h}\right) \circ\vec{\phi}_{\ccrt}(t)\dd{t} \\&=\frac{1}{T(\ccrt)}\int_0^{T(\ccrt)} \left( \DD\Cr_{1,2,3,4}\cdot \vec{h}\right)\circ\vec{\phi}_{\ccrt}\left(t-\frac{T(\ccrt)}{4}\right)\dd{t} \\
	&=\frac{1}{T(\ccrt)}\int_0^{T(\ccrt)} \left( \DD\Cr_{1,2,3,4}\cdot \vec{h}\right)\circ\sigma\vec{\phi}_{\ccrt}(t)\dd{t} \\
	&=\frac{1}{T(\ccrt)}\int_0^{T(\ccrt)} \left( \DD\Cr_{4,1,2,3}\cdot \vec{h}\right)\circ\vec{\phi}_{\ccrt}(t)\dd{t} \\
	&=\frac{1}{T(\ccrt)}\int_0^{T(\ccrt)} \left( -\DD\Cr_{1,2,3,4}\cdot \vec{h}\right)\circ\vec{\phi}_{\ccrt}(t)\dd{t} \\
	&=-F_h(\ccrt).
\end{eqnarray*} 
Note that we do not use the fact that the splay state lies in $\mathcal{L}_{\ccrt}(\Cr)$. This means that the integral above always vanishes for \emph{any} splay state of $N=4$ units \emph{regardless} of whether the governing equations of motion are WS-integrable or whether its orbit lies in $\mathcal{L}_{\ccrt}(\Cr)$. However, this is not the case if one considers $N>4$ where one has to explicitly make use of the fact that the splay state lies in $\level{\vec{\ccrt}}$ as we show below.

Before we come to the proof, we need to discuss our choice \eref{eq:CrossRatios} for the cross-ratio coordinates $\vec{\Cr}$ on $\OT$. The purpose of these coordinates was to parameterize the level sets $\level{\vec{\ccr}}$ that partition $\OT$ and $\contF$ such that for every $\vec{\theta}\in\OT$, $\vec{\ccr}=\vec{\Cr}(\vec{\theta})$ uniquely determines in which set $\level{\vec{\ccr}}$ the point $\vec{\theta}$ lies. But \eref{eq:CrossRatios} is not the only possible choice. In fact, every set of $N-3$ functionally independent cross-ratios $\Cr_{p,q,r,s}$ will do the job and indeed, the authors of \cite{Marvel_Mirollo_Strogatz_2009} used the cross-ratios
\begin{equation}\label{eq:MMS_cross_ratios}
	\Cr_{k,k+1,k+2,k+3}(\vec{\theta}) = \frac{\left(\rme^{\rmi \theta_k}-\rme^{\rmi \theta_{k+3}}\right)\left(\rme^{\rmi \theta_{k+1}}-\rme^{\rmi \theta_{k+2}}\right)} {\big(\rme^{\rmi \theta_k}-\rme^{\rmi \theta_{k+2}}\big)\big(\rme^{\rmi \theta_{k+1}}-\rme^{\rmi \theta_{k+3}}\big)} 
\end{equation}
with $k=1,\dots,N-3$ to parameterize the partition of $\OT$ in terms of level sets. While our previous choice \eref{eq:CrossRatios} was more suitable for the proof of thorem~\ref{thm:ManifoldContinuum}, the cross-ratio functions \eref{eq:MMS_cross_ratios} are suited for the analysis of splay states since they capture the spatio-temporal relation between any four consecutive phases $(\phi_k(t),\phi_{k+1}(t),\phi_{k+2}(t),\phi_{k+3}(t))$ for the splay state:
\begin{eqnarray*}
	\phi_{k+j}(t)=\phi_k(t+jT(\vec{\ccrt})/N), \qquad j=1,2,3.
\end{eqnarray*}
Let us therefore from now define the cross-ratio functions $\vec{\Cr}$ as
\begin{equation}\label{eq:NewCrossRatios}
	\eqalign{
	\vec{\Cr}(\vec{\theta}) &\defined\left(\Cr_1(\vec{\theta}),\dots,\Cr_{N-3}(\vec{\theta})\right) \\
	\Cr_k(\vec{\theta})&\defined\Cr_{k,k+1,k+2,k+3}(\vec{\theta}) = \frac{\big(\rme^{\rmi \theta_k}-\rme^{\rmi \theta_{k+3}}\big)\big(\rme^{\rmi \theta_{k+1}}-\rme^{\rmi \theta_{k+2}}\big)} {\big(\rme^{\rmi \theta_k}-\rme^{\rmi \theta_{k+2}}\big)\big(\rme^{\rmi \theta_{k+1}}-\rme^{\rmi \theta_{k+3}}\big)}
	}
\end{equation}
with $k=1,\dots,N-3$. In these new cross-ratio coordinates, the leaf of uniform distributions $\level{\vec{\ccrt}}$ is determined by
\begin{equation*}
	\vec{\ccrt}\defined\vec{\Cr}(\vec{\thetat}) = \left(\ccrt_1,\dots,\ccrt_1\right)
\end{equation*}
where $\ccrt_1$ is given through \eref{eq:DefLambdaTilde} with $k=1$.
For any given periodic orbit $\mathcal{C}_{\vec{\ccr}}\subset\mathcal{L}_{\vec{\ccr}}(\vec{\Lambda})$, hyperbolic fixed points of the averaged system
\begin{equation*}
	\dot{\vec{\ccr}}=\vec{F}_h(\vec{\ccr})
\end{equation*}
where $\vec{F}_h(\vec{\ccr})=((F_h)_1(\vec{\ccr}),\dots,(F_h)_{N-3}(\vec{\ccr}))$ is defined by
\begin{equation*}
	(F_h)_k(\vec{\ccr})\defined \frac{1}{T(\vec{\ccr})} \int_0^{T(\vec{\ccr})} \left( \DD\Cr_{k,k+1,k+2,k+3}\cdot \vec{h} \right)\circ\vec{\phi}_{\vec{\ccr}}(t) \dd{t}
\end{equation*}
for $k=1,\dots,N-3$, correspond to periodic orbits of the original system. With these remarks, we are ready to prove the last statement of theorem~\ref{prop:Averaging}.

\begin{proof}
	Consider the cyclic permutation
	\begin{equation*}
		\sigma = \pmatrix{ 
		1 & 2 & \dots  & N \cr 
		N & 1 & \dots  & N-1 
		}
	\end{equation*}
	and as in the previous section, let for fixed $\vec{\theta}\in\mathcal{C}_{\vec{\ccrt}}$
	\begin{equation*}
		\vec{\phi}_{\vec{\ccrt}}(t) = \left( \phi_{\vec{\ccrt}}^1(t),\dots,\phi_{\vec{\ccrt}}^N(t) \right)
	\end{equation*}
	denote the the splay state solution of \eref{eq:GeneralSystem} with initial condition $\vec{\phi}_{\vec{\ccrt}}(0)=\vec{\theta}$ and set 
	\begin{eqnarray*}
		\sigma\vec{\phi}_{\vec{\ccrt}}(t) \defined \left( \phi_{\vec{\ccrt}}^{\sigma(1)}(t),\phi_{\vec{\ccrt}}^{\sigma(2)},\dots,\phi_{\vec{\ccrt}}^{\sigma(N)}(t) \right).
	\end{eqnarray*}
	Using that for splay states $\phi_{\vec{\ccrt}}^{k+1}(t)=\phi_{\vec{\ccrt}}^k(t+T(\vec{\ccrt})/N)$ holds, we then find
	\begin{equation}\label{eq:SplayStateSpatioTemporal}
		\eqalign{
		\sigma\vec{\phi}_{\vec{\ccrt}}(t) &= \left( \phi_{\vec{\ccrt}}^{\sigma(1)}(t),\phi_{\vec{\ccrt}}^{\sigma(2)},\dots,\phi_{\vec{\ccrt}}^{\sigma(N)}(t) \right) \\
		&=\left( \phi_{\vec{\ccrt}}^{N}(t),\phi_{\vec{\ccrt}}^1(t),\dots,\phi_{\vec{\ccrt}}^{N-1}(t) \right) \\
		&=\left(\phi_{\vec{\ccrt}}^1\left(t-\frac{T(\vec{\ccrt})}{N}\right),\dots,\phi_{\vec{\ccrt}}^N\left(t-\frac{T(\vec{\ccrt})}{N}\right)\right) \\
		&=\vec{\phi}_{\vec{\ccrt}}\left(t-\frac{T(\vec{\ccrt})}{N}\right).
		}
	\end{equation}
	Note that for the choice \eref{eq:NewCrossRatios} for the cross-ratios, we have
	\begin{equation}\label{eq:NewCrossRatioTransform}
		\eqalign{
		\Cr_{k+1,k+2,k+3,k+4}(\sigma\vec{\theta}) &= \frac{\big(\rme^{\rmi \theta_{\sigma(k+1)}}-\rme^{\rmi \theta_{\sigma(k+4)}}\big) \big(\rme^{\rmi \theta_{\sigma(k+2)}}-\rme^{\rmi \theta_{\sigma(k+3)}}\big)}{\big(\rme^{\rmi \theta_{\sigma(k+1)}}-\rme^{\rmi \theta_{\sigma(k+3)}}\big) \big(\rme^{\rmi \theta_{\sigma(k+2)}}-\rme^{\rmi \theta_{\sigma(k+4)}}\big)} \\
		&= \frac{\big(\rme^{\rmi \theta_{k}}-\rme^{\rmi \theta_{k+3}}\big) \big(\rme^{\rmi \theta_{k+1}}-\rme^{\rmi \theta_{k+2}}\big)}{\big(\rme^{\rmi \theta_{k}}-\rme^{\rmi \theta_{k+2}}\big) \big(\rme^{\rmi \theta_{k+1}}-\rme^{\rmi \theta_{k+3}}\big)} \\
		&= \Cr_{k,k+1,k+2,k+3}(\vec{\theta})
		}
	\end{equation}
	for all $\vec{\theta}\in\OT$. Because of the spatio-temporal symmetry \eref{eq:SplayStateSpatioTemporal} of the splay state and using \eref{eq:NewCrossRatioTransform}, we can for $k=1,\dots,N$ (where we set $N+1\equiv1$, $N+2\equiv2$, and $N+3\equiv3$) write
	\begin{eqnarray*}
		(F_h)_{k+1}(\vec{\ccrt}) &\defined \frac{1}{T(\vec{\ccrt})} \int_0^{T(\vec{\ccrt})} \left( \DD\Cr_{k+1,k+2,k+3,k+4}\cdot \vec{h} \right)\circ\vec{\phi}_{\vec{\ccrt}}(t)\dd{t} \\
		&=\frac{1}{T(\vec{\ccrt})} \int_0^{T(\vec{\ccrt})} \left( \DD\Cr_{k+1,k+2,k+3,k+4}\cdot \vec{h} \right)\circ\vec{\phi}_{\vec{\ccrt}}\left(t-\frac{T(\vec{\ccrt})}{N}\right)\dd{t} \\
		&=\frac{1}{T(\vec{\ccrt})} \int_0^{T(\vec{\ccrt})} \left( \DD\Cr_{k+1,k+2,k+3,k+4}\cdot \vec{h} \right)\circ\sigma\vec{\phi}_{\vec{\ccrt}}(t)\dd{t} \\
		&=\frac{1}{T(\vec{\ccrt})} \int_0^{T(\vec{\ccrt})} \left( \DD\Cr_{k,k+1,k+2,k+3}\cdot \vec{h} \right)\circ\vec{\phi}_{\vec{\ccrt}}(t)\dd{t} \\
		&=(F_h)_{k}(\vec{\ccrt}).
	\end{eqnarray*}
	This means that in particular, the $N-3$ components (\ie, $k=1,\dots,N-3$) of $\vec{F}_h(\vec{\ccrt})$ coincide. We also have for every $j=1,\dots,N-3$
	\begin{eqnarray*}
		(F_h)_j(\vec{\ccrt}) &= \frac{1}{N} \sum_{k=1}^N (F_h)_k(\vec{\ccrt}) \\
		&= \frac{1}{NT(\vec{\ccrt})} \int_0^{T(\vec{\ccrt})} \left( \sum_{k=1}^N \DD\Cr_{k,k+1,k+2,k+3}\cdot\vec{h} \right) \circ\vec{\phi}_{\vec{\ccrt}}(t)\dd{t} \\
		&= \frac{1}{NT(\vec{\ccrt})} \int_0^{T(\vec{\ccrt})} \sum_{k=1}^N \Big( \DD_{\theta_{k}}\Cr_{k,k+1,k+2,k+3}(\vec{\theta})\cdot h(\theta_{k}) + \\
		&\hspace{3.45cm}+ \DD_{\theta_{k+1}}\Cr_{k,k+1,k+2,k+3}(\vec{\theta})\cdot h(\theta_{k+1}) + \\
		&\hspace{3.45cm}+ \DD_{\theta_{k+2}}\Cr_{k,k+1,k+2,k+3}(\vec{\theta})\cdot h(\theta_{k+2}) + \\
		&\hspace{3.45cm}+\DD_{\theta_{k+3}}\Cr_{k,k+1,k+2,k+3}(\vec{\theta})\cdot h(\theta_{k+3})\Big)\Big|_{\vec{\theta} = \vec{\phi}_{\vec{\ccrt}}(t)}\dd{t}.
	\end{eqnarray*}
	Rearranging this sum by collecting all terms that contain derivatives with respect to a given $\theta_k$ yields	
	\begin{equation}\label{eq:vanishingIntegral}
		\eqalign{
		(\hat{F}_h)_j(\vec{\ccrt})&= \frac{1}{NT(\vec{\ccrt})} \int_0^{T(\vec{\ccrt})} \sum_{k=1}^N \DD_{\theta_k}\Big(\Cr_{k,k+1,k+2,k+3}(\vec{\theta}) + \\
		&\hspace{4cm}+\Cr_{k-1,k,k+1,k+2}(\vec{\theta})+ \\
		&\hspace{4cm}+\Cr_{k-2,k-1,k,k+1}(\vec{\theta})+ \\
		&\hspace{4cm}+
		\Cr_{k-3,k-2,k-1,k}(\vec{\theta})\Big)\cdot h(\theta_k) \Big|_{\vec{\theta}=\vec{\phi}_{\vec{\ccrt}}(t)}\dd{t}.
		}
	\end{equation}
	All terms in the sum in \eref{eq:vanishingIntegral} vanish so the integral yields zero. To see this, observe that for any $\vec{\theta}\in\mathcal{L}_{\vec{\ccrt}}(\vec{\Lambda})$, there exists $(\alpha,\psi)\in\D\times\S$ such that $\rme^{\rmi \vec{\theta}}=\M{\alpha}{\psi}\left(\rme^{\rmi \vec{\thetat}}\right)$ with $\vec{\thetat}$ from \eref{eq:DefThetaTilde}, yielding
	\begin{eqnarray*}
		\DD_{\theta_k} \Cr_{k-3,k-2,k-1,k}(\vec{\theta}) &= \rmi\rme^{\rmi \theta_k} \frac{\left( \rme^{\rmi \theta_{k-3}} - \rme^{\rmi \theta_{k-2}} \right) \left( \rme^{\rmi \theta_{k-2}} - \rme^{\rmi \theta_{k-1}} \right)}{\left( \rme^{\rmi \theta_{k-3}} - \rme^{\rmi \theta_{k-1}} \right)\left( \rme^{\rmi \theta_{k-2}} - \rme^{\rmi \theta_{k}} \right)^2} \\
		&= \frac{\rmi\rme^{\rmi \frac{2\pi(k-1)}{N}-\rmi\psi}\left( \rme^{\rmi \frac{2\pi k}{N}+\rmi\psi} - \rme^{\rmi \frac{2\pi}{N}}\alpha \right) \left( \rme^{\rmi \frac{2\pi}{N}}-\rme^{\rmi \frac{2\pi k}{N}+\rmi\psi} \alphabar \right)} {\left( 1-\rme^{\rmi \frac{2\pi}{N}}\right)\left(1+\rme^{\rmi \frac{2\pi}{N}}\right)^3\left(1-\abs{\alpha}^2 \right)} \\
		&= -\rmi\rme^{\rmi \theta_k} \frac{\left( \rme^{\rmi \theta_{k+3}} - \rme^{\rmi \theta_{k+2}} \right) \left( \rme^{\rmi \theta_{k+2}} - \rme^{\rmi \theta_{k+1}} \right)}{\left( \rme^{\rmi \theta_{k+3}} - \rme^{\rmi \theta_{k+1}} \right)\left( \rme^{\rmi \theta_{k+2}} - \rme^{\rmi \theta_{k}} \right)^2} \\
		&= -\DD_{\theta_k}\Cr_{k,k+1,k+2,k+3}(\vec{\theta})
	\end{eqnarray*}
	so that the first and last term in each summand in \eref{eq:vanishingIntegral} cancel out. We also have
	\begin{eqnarray*}
		\fl\DD_{\theta_k} \Cr_{k-2,k-1,k,k+1}(\vec{\theta}) &= \rmi\rme^{\rmi \theta_k} \frac{\left( \rme^{\rmi \theta_{k-1}} - \rme^{\rmi \theta_{k-2}} \right) \left( \rme^{\rmi \theta_{k+1}} - \rme^{\rmi \theta_{k-2}} \right)}{\left( \rme^{\rmi \theta_{k-1}} - \rme^{\rmi \theta_{k+1}} \right)\left( \rme^{\rmi \theta_{k}} - \rme^{\rmi \theta_{k-2}} \right)^2} \\
		&= -\frac{\rmi\left( 1+\rme^{\rmi \frac{2\pi}{N}}+\rme^{\rmi \frac{4\pi}{N}}\right)\left( 1-\rme^{-\rmi\frac{2\pi (k-1)}{N}-\rmi\psi}\alpha \right) \left( \rme^{\rmi \frac{2\pi}{N}}-\rme^{\rmi \frac{2\pi k}{N}+\rmi\psi}\alphabar \right)} {\left( 1-\rme^{\rmi \frac{2\pi}{N}} \right) \left( 1+\rme^{\rmi \frac{2\pi}{N}} \right)^3\left(1-\abs{\alpha}^2\right)} \\
		&= -\rmi\rme^{\rmi \theta_k} \frac{\left( \rme^{\rmi \theta_{k-1}} - \rme^{\rmi \theta_{k+2}} \right) \left( \rme^{\rmi \theta_{k+1}} - \rme^{\rmi \theta_{k+2}} \right)}{\left( \rme^{\rmi \theta_{k-1}} - \rme^{\rmi \theta_{k+1}} \right)\left( \rme^{\rmi \theta_{k}} - \rme^{\rmi \theta_{k+2}} \right)^2} \\
		&=-\DD_{\theta_k} \Cr_{k-1,k,k+1,k+2}(\vec{\theta})
	\end{eqnarray*}
	so that the second and third term in the summand also cancel and the integrands for all components of $\vec{F}_h(\vec{\ccrt})$ vanish. This finishes the proof of statement 3. of theorem~\ref{prop:Averaging}.
\end{proof}


\section{Application: repulsively coupled identical active rotators}\label{sec:ARWS}

Finally, we apply theorems~\ref{thm:ManifoldContinuum} and \ref{prop:Averaging} to our original motivational model of identical classic and generalized active rotators \eref{eq:ActiveRotatorsClassic} and \eref{eq:ActiveRotatorPerturbed}, starting with the classic case.

\subsection{Classic active rotators}

Recall that in order for \eref{eq:ActiveRotatorsClassic} to describe a system of classic active rotators and not a full blown oscillators, $\abs{\omega}<1$ holds. First, we determine the NAIM from theorem~\ref{thm:ManifoldContinuum}.

We start by writing down \Eref{eq:ActiveRotatorsClassic} in terms of WS-variables. Comparing with \eref{eq:GeneralSystem}, we have $g(Z)=\omega$ and $f(Z)=\frac{\rmi}{2}\left(1+\kappa\Zbar\right)$ so that in WS-variables, we have
\begin{eqnarray*}
	\dot{\alpha} &= -\frac{1}{2}(1+\kappa\Zbar)\alpha^2 + \rmi\omega\alpha+ \frac{1}{2}(1+\kappa Z) \\
	\dot{\psi} &= \frac{\rmi}{2}(1+\kappa\Zbar)\alpha + \omega - \frac{\rmi}{2}(1+\kappa Z)\alphabar \\
	\dot{\vec{\ccr}} &= \vec{0}.
\end{eqnarray*}
Replacing $Z$ by $\alpha$ then yields the closed (truncated) equation
\begin{eqnarray}
	\dot{\alpha} = -\frac{1}{2}(1+\kappa\alphabar)\alpha^2+\rmi\omega\alpha +\frac{1}{2}(1+\kappa\alpha) \label{eq:truncatedAlphaDot}
\end{eqnarray}
for $\dot{\alpha}$. We proceed by determining the fixed point $\alpha_0$ for \eref{eq:truncatedAlphaDot} in theorem~\ref{thm:ManifoldContinuum}. Setting $\alpha=r\rme^{\rmi \beta}$ and subsequently $\dot{\alpha}=\dot{r}\rme^{\rmi \beta}+\rmi\dot{\beta}r\rme^{\rmi \beta}$ yields
\begin{eqnarray*}
	\fl\dot{r}+ \rmi\dot{\beta}r 
	&= \left(-\frac{1}{2}r^2\cos\beta- \frac{\kappa}{2}r^3+\frac{1}{2}\cos\beta+\frac{\kappa}{2}r\right)+ \rmi\left(-\frac{1}{2}r^2\sin\beta+\omega r-\frac{1}{2}\sin\beta\right).
\end{eqnarray*}
The fixed point condition implies
\begin{eqnarray*}
	(r^2-1)\cos\beta+(r^2-1)\kappa r = 0\qquad\textrm{and}\qquad
	\frac{1}{2}(r^2+1)\sin\beta = \omega r
\end{eqnarray*}
and thus, since $r<1$,
\begin{equation}\label{eq:trigbeta}
	\eqalign{
	\cos\beta = -\kappa r\qquad \textrm{and}\qquad
	\sin\beta = \frac{2\omega r}{(1+r^2)}.
	}
\end{equation}
Eliminating the trigonometric terms and setting $x=r^2$ yields a cubic equation in $x$:
\begin{equation}\label{eq:cubic_equation}
	0 = \kappa^2x^3 + (2\kappa^2-1)x^2+(\kappa^2+4\omega^2-2)x-1.
\end{equation}
We make the following claim:

\begin{lemma}
	The cubic equation \eref{eq:cubic_equation} has for $\omega^2<1$ exactly one real root $x\in(0,1)$ if $\kappa^2>1-\omega^2$ and no real roots in $(0,1)$ if $\kappa^2<1-\omega^2$.
	\label{lm:cubic_equation}
\end{lemma}
\begin{proof}
	We solve \eref{eq:cubic_equation} for $\kappa^2$ and find
	\begin{equation*}
		\kappa^2 = \frac{1}{x} - \frac{4\omega^2}{(1+x)^2}.
	\end{equation*}
	We show that for any $\omega^2<1$ the map $(0,1)\ni x\mapsto 1/x - 4\omega^2/(1+x)^2\in(1-\omega^2,\infty)$ is a bijection from which the claim follows because bijectivity implies conversely that for any $\omega^2<1$ and $\kappa^2>1-\omega^2$ there exists a unique $x\in(0,1)$ that solves \eref{eq:cubic_equation} and there is no such $x$ for $\kappa^2<1-\omega^2$.
	\\
	Injectivity: The map is differentiable in $(0,1)$ so that
	\begin{eqnarray*}
		\diff{\kappa^2}{x} = -\frac{1}{x^2} + \frac{8\omega^2}{(1+x)^3} &< 0  \quad
		\Leftrightarrow\quad \frac{(1+x)^3}{8x^2} &> \omega^2.
	\end{eqnarray*}
	But $\omega^2<1<(1+x)^3/8x^2$ holds because $(1+x)^3/8x^2$ is strictly monotonically decreasing in $(0,1)$ since $\dd{}/\dd{x}\left((1+x)^3/8x^2\right)=(x-2)(1+x)^2/8x^3<0\;\forall x\in(0,1)$ and its infimum is $\lim_{x\to1}(1+x)^3/8x^2=1$. Hence $\kappa^2$ as a function of $x$ is strictly monotonically decreasing and therefore injective.
	\\
	Surjectivity: Because $\kappa^2$ is as a function of $x$ continuous and strictly monotonically decreasing, we find that its image $(a,b)$ is given by
	\begin{eqnarray*}
		a &= \lim_{x\to1}\frac{1}{x}-\frac{4\omega^2}{(1+x)^2} = 1-\omega^2 \qquad \textrm{and}\qquad
		b &= \lim_{x\to0}\frac{1}{x}-\frac{4\omega^2}{(1+x)^2} = \infty
	\end{eqnarray*}
	so that the map is indeed surjective and thus bijective.
\end{proof}

Using \eref{eq:trigbeta}, we write
\begin{equation*}
	\alpha_0=r(\cos\beta+ i\sin\beta) = -\kappa r^2 + 2\rmi \frac{\omega r^2}{1+r^2}
\end{equation*}
and find from lemma~\ref{lm:cubic_equation} that as long as $\omega\neq0$
\begin{eqnarray*}
	f(\alpha_0) &= \frac{\rmi}{2}\left(1+\kappa\alphabar_0\right)
	= \frac{\omega\kappa r^2}{1+r^2} + \frac{\rmi}{2}\left(1-\kappa^2r^2\right)
	\neq0
\end{eqnarray*}
since $\kappa^2>1-\omega^2>0$ and so we have
\begin{equation*}
	\dot{\psi} = \omega - \Im\alpha_0 = \omega\left( 1 - \frac{2r^2}{1+r^2} \right) = \omega \frac{1-r^2}{1+r^2} \revdefined \Omega\neq0.
\end{equation*}
Note that for $\omega=0$, the contours $\mathcal{C}_{\vec{\ccr}}$ and therefore the manifold $\contF_{\delta}$ consist of fixed points with two stable and $N-2$ neutral directions.

We determine the stability of $\alpha_0$. Again, we treat $\alpha$ and $\alphabar$ as independent variables in $\D$ so that the dynamics in $\D$ is given by
\begin{eqnarray*}
	\dot{\alpha} &= -\frac{1}{2}(1+\kappa\alphabar)\alpha^2+\rmi\omega\alpha+\frac{1}{2}(1+\kappa\alpha).
\end{eqnarray*}
Determining the Jacobian of the right hand side of this equation yields
\begin{equation}\label{eq:alphaJacobian}
	\frac{\partial(\dot{\alpha},\dot{\alphabar})}{\partial(\alpha,\alphabar)} = 
	\pmatrix{
	-(1+\kappa\alphabar)\alpha+\rmi\omega+\frac{\kappa}{2} & -\frac{\kappa}{2}\alpha^2 \cr
	-\frac{\kappa}{2}\alphabar^2 & -(1+\kappa\alpha)\alphabar -\rmi\omega+\frac{\kappa}{2}
	}.
\end{equation}
At the fixed point $\alpha_0=r\rme^{\rmi \beta}$, using equations \eref{eq:trigbeta}, the eigenvalues of \eref{eq:alphaJacobian} read
\begin{equation}\label{eq:alpha_eigenvalues}
	\ccr_{\pm} = \frac{\kappa}{2} \pm \sqrt{\frac{\kappa^2}{4}x^2 - \omega^2\left(\frac{1-x}{1+x}\right)^2}.
\end{equation}
It follows that $\mathrm{sign}\,\Re\,\ccr_{\pm}=\mathrm{sign}\,\kappa$:  We have two cases, that of $\ccr_{\pm}\in\mathbb{R}$ and that of $\ccr_{\pm}\notin\mathbb{R}$. In the first case, the radicand in \eref{eq:alpha_eigenvalues} is nonnegative and we have
\begin{equation*}
	\frac{\kappa^2}{4}>\frac{\kappa^2}{4}x^2-\omega^2\left(\frac{1-x}{1+x}\right)^2
\end{equation*}
since $x\in(0,1)$ and the subtraction of a nonnegative term and hence $\mathrm{sign}\,\Re\ccr_{\pm}=\mathrm{sign}\,\kappa$. In the second case, the radicand is negative and the statement holds trivially.

To conclude, the above results on the existence and stability of $\alpha_0$ show that the system \eref{eq:ActiveRotatorsClassic} fulfills the requirements of theorem~\ref{thm:ManifoldContinuum} if $0<\abs{\omega}<1$ and $\kappa<-\sqrt{1-\omega^2}$. This establishes the final result of this section:

\begin{theorem}\label{thm:ManifoldContinuumAR} 
	Let  $\vec{\phi}(t)=(\phi_1(t),\dots,\phi_N(t))\in \OT$ obey
	\eref{eq:ActiveRotatorsClassic} with parameters $0<\omega^2<1$ and $\kappa<-\sqrt{1-\omega^2}$. Then, for every $\delta>0$, there exists an $N_0\in\mathbb{N}$ such that for all $N\geq N_0$, there exists a closed $\delta$-neighborhood $\overline{\OCr}_{\delta}=\overline{\OCr}_{\delta}(\vec{\ccrt})\subset \OCr$ of $\vec{\ccrt}$, where for every $\vec{\ccr}\in \overline{\OCr}_{\delta}$, there exists a unique periodic orbit $\mathcal{C}_{\vec{\ccr}}\subset\level{\vec{\ccr}}$ which is exponentially stable in $\level{\vec{\ccr}}$. The union
	\begin{equation*}
		\contF(\delta) \defined \bigcup_{\vec{\ccr}\in \overline{\OCr}_{\delta}}\mathcal{C}_{\vec{\ccr}} \subset \OT
	\end{equation*}
	forms a compact normally attracting manifold of dimension $N-2$. Further, the periodic solution $\vec{\phi}_{\vec{\ccrt}}(t)$ with orbit $\mathcal{C}_{\vec{\ccrt}}\subset\contF_{\delta}$ is a splay state.
\end{theorem}
\begin{proof}
	This is a corollary of theorem~\ref{thm:ManifoldContinuum} and proposition~\ref{prop:SplayState}.
\end{proof}

\subsection{Generalized active rotators}

\begin{figure}
	\centering
	\includegraphics[width=\linewidth]{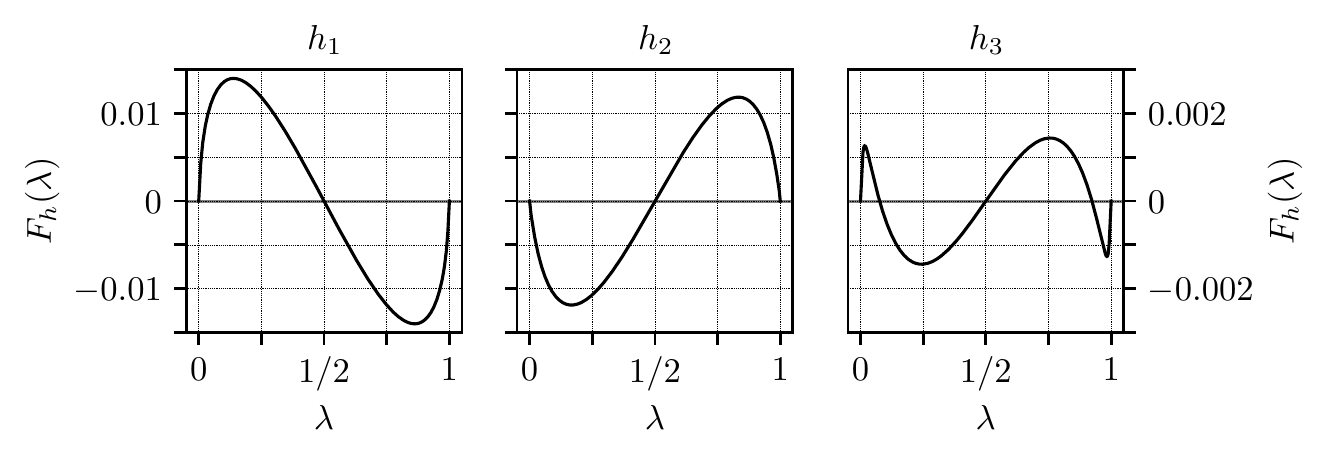}
	\caption{Three examples for $F_h(\ccr)$ for $N=4$ Active Rotators obeying \eref{eq:GeneralSystemPerturbed} with $\omega=0.8$ and $\kappa=-0.7$. The three choices for $h$ are $h_1:\phi\mapsto\sin2\phi$, $h_2:\phi\mapsto-\cos2\phi$, and $h_3=(3/5) h_1+(4/5) h_2$. Note the additional zeros for the case of $h_3$.}
	\label{fig:averagedperturbations4}
\end{figure}

We apply the averaging method from \sref{sec:Averaging} to the generalized active rotator model \eref{eq:ActiveRotatorPerturbed} to determine hyperbolic orbits. We focus on the case of $N=4$ rotators for which we are dealing with a single cross-ratio $\ccr=\Cr_{1,2,3,4}(\vec{\theta})$ and one-dimensional averaged dynamics $\dot{\ccr}=F_h(\ccr)$.

Concerning the hyperbolicity of the splay state, \fref{fig:averagedperturbations4} illustrates how $F_h$ vanishes at $\ccrt$ in accordance with statement 2. of theorem~\ref{prop:Averaging} and while $\DD F_h(\ccrt)\neq0$ so that $\ccrt$ becomes a hyperbolic fixed point for the averaged dynamics and $\mathcal{C}_{\epsilon,\ccrt}$ a robust periodic orbit. In the figure, we plot $F_h$ for three different choices of $h$, the last one being a linear combination of the former two. The splay state persists under all three perturbations. Note that from the numerical results, we have
\begin{equation*}
	\lim_{\ccr\to0}F_h(\ccr) = \lim_{\ccr\to1}F_h(\ccr) = 0
\end{equation*}
which, if we include these limit cases, indicates the existence of additional periodic orbits which must be clustered. Indeed, at least for the generalized active rotator model \eref{eq:ActiveRotatorPerturbed} with $N=4$, the values $\ccr=0$ and $\ccr=1$ correspond to periodic two-cluster states \cite{Ronge_Zaks_2021, Ronge_Zaks_2021_2} where the ensemble splits in two clusters of two units each. From \fref{fig:averagedperturbations4}, we read that these two-cluster states are also exponentially (un)stable periodic solutions, depending on the sign of $\epsilon$. Note that for the perturbation function $h_3$ we find two further zeros which indicate the existence of additional periodic orbits besides the splay state which are of broken spatio-temporal symmetry \cite{Ronge_Zaks_2021_2}. The graph of $F_{h_3}$ is a linear combination of the two graphs of from the first two Panels in \fref{fig:averagedperturbations4} just like $h_3$ is a linear combination of $h_1$ and $h_2$. This relation and its consequences on whether splay states are robust under generic perturbations are further elaborated in the supplementary material.

\section{Conclusion and outlook}
Families of periodic orbits are a common occurrence in Watanabe-Strogatz integrable systems. For the case that all involved common fields only depend on the Kuramoto order parameter, we showed that the union of these orbits further possesses the structure of a normally attracting invariant manifold and that one of the orbits is a splay state. To determine which of the periodic orbits survive for the perturbed system, we developed a criterion by applying the method of averaging to the dynamics on the NAIM.

One of the orbits that formed the invariant manifold of the unperturbed system features splay state dynamics due to equivariance under cyclic permutations. We showed that averaging over the splay state orbit of the unperturbed system yields a fixed point in the averaged dynamics for the perturbed system. Whether it also hyperbolic depends on the specific setup of the perturbed system and in particular on the choice for the perturbation function.

Applying our results to the active rotator model by Shinomoto and Kuramoto, our observations are two-fold: (i) We determined the critical coupling strength below which the model possesses the family of periodic orbits. (ii) Applying the averaging principle to the perturbed or generalized active rotator model, numerical evidence indicates that averaging over the unperturbed splay state indeed results in a hyperbolic fixed point for the averaged dynamics and thus the splay state is generically hyperbolic for ensembles of identical generalized active rotators. It also illustrates the versatility of the proposed procedure as an easy-to-implement tool to study perturbations to WS-integrable systems.

In restricting attention to the case of solely $Z$-dependent common fields $f$ and $g$, we excluded, \eg, systems of identically driven elements with external forcing \cite{Baibolatov_2009} or ensembles that consist of subpopulations \cite{Hong_2011}, both of which are WS-integrable. The question thus arises whether, and if, how the results, presented here, may be generalized to such and other WS-integrable systems.

\ack
We thank Zheng Bian, Jaap Eldering, and Edmilson Roque dos Santos for valuable discussions and comments. This work has been performed within the scope of the IRTG 1740/TRP 2015/50122-0 and funded by the DFG and FAPESP. TP was also supported in part by FAPESP Cemeai Grant No. 2013/07375-0, by Serrapilheira Institute (Grant No. Serra-1709-16124) and is a Newton Advanced Fellow of the Royal Society NAF\textbackslash R1\textbackslash180236.




\section*{Supplementary material}

# Continua and persistence of periodic orbits in ensembles of oscillators - Supplementary Material

**R Ronge**[1], **M A Zaks**[1], **T Pereira**[2,3]

[1] Institut für Physik, Humboldt-Universität zu Berlin, 12489 Berlin, Germany
[2] Institute of Mathematical and Computer Sciences, University of São Paulo, Brazil
[3] Department of Mathematics, Imperial College London, London SW7 2AZ, United Kingdom

E-mail: robert.ronge@physik.hu-berlin.de

## S.1. Determining perturbed asymptotic dynamics

The function $\boldsymbol{F}_h$ from theorem 5.1 can determine the asymptotic dynamics for the perturbed system (2), due to the normal attractiveness of the manifolds $\mathcal{M}$ and $\mathcal{M}_\epsilon$.

First, we note that for every $N \geq N_0$, with $N_0$ from theorem 4.12, there exists an $\epsilon_0 = \epsilon_0(N) > 0$ such that the system (2) possesses for all $|\epsilon| < \epsilon_0$ a NAIM $\mathcal{M}_\epsilon$ by means of the persistence theorem 3.2 which is diffeomorphic and $\mathcal{O}(\epsilon)$-close to $\mathcal{M}$ because the perturbation term in (2) is of order $\mathcal{O}(\epsilon)$ in $C^1$-norm

We want to apply the averaging theorem 3.3 to determine the asymptotic dynamics of (2) by averaging over the degenerate dynamics of (1) on the NAIM $\mathcal{M}$. For this, we note first that $\mathcal{M}$ is locally of the form $\mathcal{M} \cong \mathbb{S}^1 \times W$, i.e., every point $\boldsymbol{\theta} \in \mathcal{M}$ is uniquely determined by (i) its cross-ratio $\boldsymbol{\lambda} = \boldsymbol{\Lambda}(\boldsymbol{\theta})$ which determines in which orbit $\mathcal{C}_{\boldsymbol{\lambda}} \subset \mathcal{M}$ it lies and (ii) its position on $\mathcal{C}_{\boldsymbol{\lambda}}$ which is determined by its phase $s = S(\boldsymbol{\theta}) \in \mathbb{S}^1$. The phase $s$ becomes uniquely defined, once we define a Poincaré section $\Sigma \subset \mathcal{M}$ which lies transversal to each orbit $\mathcal{C}_{\boldsymbol{\lambda}} \subset \mathcal{M}$, such that for every $\boldsymbol{\theta} \in \Sigma$, we set $S(\boldsymbol{\theta}) = 0$. If we choose $\Sigma \subset \mathcal{M}$ to be a smooth submanifold of $\mathcal{M}$, the function $S$ is also smooth. This results in a chart

$$\chi : \boldsymbol{\theta} \mapsto (s, \boldsymbol{\lambda}) = (S(\boldsymbol{\theta}), \boldsymbol{\Lambda}(\boldsymbol{\theta}))$$

for $\mathcal{M}$ for which (16), restricted to $\mathcal{M}$, reads

$$\dot{s} = \frac{2\pi}{T(\boldsymbol{\lambda})}$$
$$\dot{\boldsymbol{\lambda}} = \boldsymbol{0}.$$

By assumption, $T(\boldsymbol{\lambda})$ is bounded from above for all $\boldsymbol{\lambda} \in W$ and so there exists a $0 < c < 2\pi/T(\boldsymbol{\lambda})$.

We want to introduce compatible coordinates on $\mathcal{M}_\epsilon$ and determine the equations of motion on it in terms of these coordinates. Noting that $\mathcal{M}$ is the union of periodic orbits $\mathcal{C}_{\boldsymbol{\lambda}}$, each stable in its respective level set $\mathcal{L}_{\boldsymbol{\lambda}}(\boldsymbol{\Lambda})$, there exists an open neighborhood $U(\mathcal{M})$ of $\mathcal{M}$ such that $\chi$ can be extended to

$$\chi : U \to \mathbb{S}^1 \times V$$
$$\chi : \boldsymbol{\vartheta} \mapsto (S(\boldsymbol{\vartheta}), \boldsymbol{\Lambda}(\boldsymbol{\vartheta}))$$

which is as smooth as the vector field in (2), see [1]. For $|\epsilon|$ sufficiently small, $\mathcal{M}_\epsilon$ lies within $U(\mathcal{M})$ and in particular, every point $\boldsymbol{\vartheta} \in \mathcal{M}_\epsilon$ is uniquely determined by the coordinates

$$(s, \boldsymbol{\lambda}) = \chi(\boldsymbol{\vartheta}) = (S(\boldsymbol{\vartheta}), \boldsymbol{\Lambda}(\boldsymbol{\vartheta}))$$

because every $\boldsymbol{\vartheta} \in \mathcal{M}_\epsilon$ lies in exactly one isochron in the level set $\mathcal{L}_{\boldsymbol{\lambda}}(\boldsymbol{\Lambda})$. There further exists a near-identity map

$$\boldsymbol{\rho} : \mathcal{M} \to \mathcal{M}_\epsilon$$
$$\boldsymbol{\rho} : \boldsymbol{\theta} \mapsto \boldsymbol{\vartheta} = \boldsymbol{\theta} + \epsilon \boldsymbol{P}(\boldsymbol{\theta}, \epsilon)$$

which can be extend to a diffeomorphism

$$\varrho : U(\mathcal{M}) \to U(\mathcal{M}_\epsilon)$$

between $U(\mathcal{M})$ and an open neighborhood $U(\mathcal{M}_\epsilon)$ of $\mathcal{M}_\epsilon$. In particular, we have $\varrho|_\mathcal{M} = \rho$ and further, $\varrho$ is a near-identity map around $\mathcal{M}$, as well, so that we find

$$\varrho(\boldsymbol{\theta}) = \boldsymbol{\theta} + \epsilon \boldsymbol{P}(\boldsymbol{\theta}, \epsilon) \tag{S.1}$$

for $\boldsymbol{\theta} \in \mathcal{M}$. Writing (2) as

$$\dot{\boldsymbol{\vartheta}} = \boldsymbol{G}_\mathrm{WS}(\boldsymbol{\vartheta}) + \epsilon \boldsymbol{h}(\boldsymbol{\vartheta})$$

$$\boldsymbol{G}_\mathrm{WS}(\boldsymbol{\vartheta}) := \begin{pmatrix} f(Z)\mathrm{e}^{\mathrm{i}\vartheta_1} + g(Z) + \bar{f}(Z)\mathrm{e}^{-\mathrm{i}\vartheta_1} \\ \vdots \\ f(Z)\mathrm{e}^{\mathrm{i}\vartheta_N} + g(Z) + \bar{f}(Z)\mathrm{e}^{-\mathrm{i}\vartheta_N} \end{pmatrix} \tag{S.2}$$

for $\boldsymbol{\vartheta} \in \mathcal{M}_\epsilon$ and $\boldsymbol{\theta} = \varrho^{-1}(\boldsymbol{\vartheta})$, we can Taylor-expand

$$\begin{aligned} \boldsymbol{G}_\mathrm{WS}(\boldsymbol{\vartheta}) &= \boldsymbol{G}_\mathrm{WS}(\boldsymbol{\theta} + \epsilon \boldsymbol{P}(\boldsymbol{\theta}, \epsilon)) \\ &= \boldsymbol{G}_\mathrm{WS}(\boldsymbol{\theta}) + \epsilon \mathrm{D}\boldsymbol{G}_\mathrm{WS}(\boldsymbol{\theta}) \cdot \boldsymbol{P}(\boldsymbol{\theta}, \epsilon) + \mathcal{O}(\epsilon^2) \end{aligned} \tag{S.3}$$

and

$$\begin{aligned} \boldsymbol{h}(\boldsymbol{\vartheta}) &= \boldsymbol{h}(\boldsymbol{\theta} + \epsilon \boldsymbol{P}(\boldsymbol{\theta}, \epsilon)) \\ &= \boldsymbol{h}(\boldsymbol{\theta}) + \epsilon \mathrm{D}\boldsymbol{h}(\boldsymbol{\theta}) \cdot \boldsymbol{P}(\boldsymbol{\theta}, \epsilon) + \mathcal{O}(\epsilon^2). \end{aligned} \tag{S.4}$$

For the phase $s = S(\boldsymbol{\vartheta})$, we find with (S.1) and (S.3),(S.4) and by the chain rule

$$\begin{aligned} \dot{s} &= \mathrm{D}S(\boldsymbol{\vartheta}) \cdot \dot{\boldsymbol{\vartheta}} \\ &= \mathrm{D}S(\boldsymbol{\vartheta}) \cdot [\boldsymbol{G}_\mathrm{WS}(\boldsymbol{\vartheta}) + \epsilon \boldsymbol{h}(\boldsymbol{\vartheta})] \\ &= \mathrm{D}S(\boldsymbol{\theta}) \cdot \boldsymbol{G}_\mathrm{WS}(\boldsymbol{\theta}) + \mathcal{O}(\epsilon) \\ &= \frac{2\pi}{T \circ \boldsymbol{\Lambda}(\boldsymbol{\theta})} + \mathcal{O}(\epsilon) \end{aligned} \tag{S.5}$$

because for the unperturbed system (1) we have by definition of the phase

$$\dot{S}(\boldsymbol{\theta}) = \mathrm{D}S(\boldsymbol{\theta}) \cdot \dot{\boldsymbol{\theta}} = \mathrm{D}S(\boldsymbol{\theta}) \cdot \boldsymbol{G}_\mathrm{WS}(\boldsymbol{\theta}) = \frac{2\pi}{T \circ \boldsymbol{\Lambda}(\boldsymbol{\theta})}.$$

For the cross-ratio $\boldsymbol{\lambda} = \boldsymbol{\Lambda}(\boldsymbol{\vartheta})$, we have by the same argument and because the WS-integrable term $\boldsymbol{G}_\mathrm{WS}$ of (S.2) keeps cross-ratios invariant

$$\begin{aligned} \dot{\boldsymbol{\lambda}} &= \mathrm{D}\boldsymbol{\Lambda}(\boldsymbol{\vartheta}) \cdot \dot{\boldsymbol{\vartheta}} \\ &= \mathrm{D}\boldsymbol{\Lambda}(\boldsymbol{\vartheta}) \cdot [\boldsymbol{G}_\mathrm{WS}(\boldsymbol{\vartheta}) + \epsilon \boldsymbol{h}(\boldsymbol{\vartheta})] \\ &= \underbrace{\mathrm{D}\boldsymbol{\Lambda}(\boldsymbol{\vartheta}) \cdot \boldsymbol{G}_\mathrm{WS}(\boldsymbol{\vartheta})}_{=0} + \epsilon \mathrm{D}\boldsymbol{\Lambda}(\boldsymbol{\vartheta}) \cdot \boldsymbol{h}(\boldsymbol{\vartheta}) \\ &= \epsilon \mathrm{D}\boldsymbol{\Lambda}(\boldsymbol{\theta}) \cdot \boldsymbol{h}(\boldsymbol{\theta}) + \mathcal{O}(\epsilon^2). \end{aligned} \tag{S.6}$$

We close equations (S.5) and (S.6) by writing them in terms of $s$ and $\boldsymbol{\lambda}$. For this, we need to express $\boldsymbol{\theta}$ in terms of $s = S(\boldsymbol{\vartheta})$ and $\boldsymbol{\lambda} = \boldsymbol{\Lambda}(\boldsymbol{\vartheta})$. (Note that $\boldsymbol{\theta}$ and $\boldsymbol{\vartheta}$ will generally have different phase and cross-ratios.) We note first that because of (S.1), we have

$$\begin{aligned} \boldsymbol{\Lambda}(\boldsymbol{\theta}) &= \boldsymbol{\Lambda}(\boldsymbol{\vartheta}) + \mathcal{O}(\epsilon) = \boldsymbol{\lambda} + \mathcal{O}(\epsilon) \\ S(\boldsymbol{\theta}) &= S(\boldsymbol{\vartheta}) + \mathcal{O}(\epsilon) = s + \mathcal{O}(\epsilon). \end{aligned}$$

Let
$$\boldsymbol{\theta}_0 : W \to \Sigma$$
$$\boldsymbol{\theta}_0 : \boldsymbol{\lambda} \mapsto \chi^{-1}(0, \boldsymbol{\lambda})$$

denote the point in $\mathcal{C}_{\boldsymbol{\lambda}} \subset \mathcal{M}$ of phase zero, i.e., $S(\boldsymbol{\theta}_0(\boldsymbol{\lambda})) = 0$. In particular, $\boldsymbol{\theta}_0$ is smooth since $\chi$ is smooth. Denoting the flow of the unperturbed equations (1), restricted to $\mathcal{M}$, by $\Phi_{\text{WS}}|_{\mathcal{M}} : \mathbb{R} \times \mathcal{M} \to \mathcal{M}$, we can write $\boldsymbol{\theta}$ in terms of $s$ and $\boldsymbol{\lambda}$ as

$$\boldsymbol{\theta} = \Phi_{\text{WS}}|_{\mathcal{M}} \left( S(\boldsymbol{\theta}) \frac{T \circ \boldsymbol{\Lambda}(\boldsymbol{\theta})}{2\pi}, \boldsymbol{\theta}_0 \circ \boldsymbol{\Lambda}(\boldsymbol{\theta}) \right)$$

$$= \Phi_{\text{WS}}|_{\mathcal{M}} \left( (s + \mathcal{O}(\epsilon)) \frac{T(\boldsymbol{\lambda} + \mathcal{O}(\epsilon))}{2\pi}, \boldsymbol{\theta}_0(\boldsymbol{\lambda} + \mathcal{O}(\epsilon)) \right)$$

$$= \Phi_{\text{WS}}|_{\mathcal{M}} \left( s \frac{T(\boldsymbol{\lambda})}{2\pi} + \mathcal{O}(\epsilon), \boldsymbol{\theta}_0(\boldsymbol{\lambda}) + \mathcal{O}(\epsilon) \right)$$

$$= \Phi_{\text{WS}}|_{\mathcal{M}} \left( s \frac{T(\boldsymbol{\lambda})}{2\pi}, \boldsymbol{\theta}_0(\boldsymbol{\lambda}) \right) + \mathcal{O}(\epsilon)$$

$$= \boldsymbol{\phi}_{\boldsymbol{\lambda}} \left( s \frac{T(\boldsymbol{\lambda})}{2\pi} \right) + \mathcal{O}(\epsilon)$$

where $\boldsymbol{\phi}_{\boldsymbol{\lambda}}(t) \equiv \Phi_{\text{WS}}|_{\mathcal{M}}(t, \boldsymbol{\theta}_0(\boldsymbol{\lambda}))$ is the solution of (1) with initial condition $\boldsymbol{\phi}_{\boldsymbol{\lambda}}(0) = \boldsymbol{\theta}_0(\boldsymbol{\lambda}) \in \mathcal{C}_{\boldsymbol{\lambda}}$. We arrive at the closed equations

$$\dot{s} = \frac{2\pi}{T(\boldsymbol{\lambda})} + \mathcal{O}(\epsilon)$$

$$\dot{\boldsymbol{\lambda}} = \epsilon \mathrm{D}\boldsymbol{\Lambda}\left( \boldsymbol{\phi}_{\boldsymbol{\lambda}}\left( s \frac{T(\boldsymbol{\lambda})}{2\pi} \right) + \mathcal{O}(\epsilon) \right) \cdot \boldsymbol{h}\left( \boldsymbol{\phi}_{\boldsymbol{\lambda}}\left( s \frac{T(\boldsymbol{\lambda})}{2\pi} \right) + \mathcal{O}(\epsilon) \right) + \mathcal{O}(\epsilon^2) \quad \text{(S.7)}$$

$$= \epsilon \left( \mathrm{D}\boldsymbol{\Lambda} \cdot \boldsymbol{h} \right) \circ \boldsymbol{\phi}_{\boldsymbol{\lambda}}\left( s \frac{T(\boldsymbol{\lambda})}{2\pi} \right) + \mathcal{O}(\epsilon^2)$$

and in particular, we can apply the averaging theorem 3.3 to (S.7). This yields the averaged system

$$\dot{\boldsymbol{\lambda}} = \epsilon \boldsymbol{F}_h(\boldsymbol{\lambda}) \quad \text{(S.8)}$$

with

$$\boldsymbol{F}_h(\boldsymbol{\lambda}) := \frac{1}{2\pi} \int_0^{2\pi} (\mathrm{D}\boldsymbol{\Lambda} \cdot \boldsymbol{h}) \circ \boldsymbol{\phi}_{\boldsymbol{\lambda}}\left( s \frac{T(\boldsymbol{\lambda})}{2\pi} \right) \mathrm{d}s$$

$$= \frac{1}{T(\boldsymbol{\lambda})} \int_0^{T(\boldsymbol{\lambda})} (\mathrm{D}\boldsymbol{\Lambda} \cdot \boldsymbol{h}) \circ \boldsymbol{\phi}_{\boldsymbol{\lambda}}(t) \, \mathrm{d}t$$

given by (24) where we average over the fast variable $s$ (or $t$, respectively).

Assuming that $T$ is uniformly bounded from above on $W$, and $\boldsymbol{\lambda}_0$ is a hyperbolic fixed point of (S.8) which possess $n_s$ stable and $n_u$ unstable directions, then for $\epsilon > 0$, the periodic orbit $\mathcal{C}_{\epsilon, \boldsymbol{\lambda}_0} \subset \mathcal{M}_\epsilon$ has equally $n_s$ stable and $n_u$ unstable directions in $\mathcal{M}_\epsilon$. Because the $(N-2)$-dimensional NAIM $\mathcal{M}_\epsilon$ itself possesses two attracting directions in $\mathbb{T}^N_{\text{ordered}}$, the orbit possesses in total $n_s + 2$ stable directions in $\mathbb{T}^N_{\text{ordered}}$. On the other hand, for $\epsilon_0 < \epsilon < 0$, the fixed point $\boldsymbol{\lambda}_0$ has $n_u$ stable and $n_s$ unstable directions

Continua and persistence of periodic orbits in ensembles of oscillators - Supplementary Material 5and by the same argument $\mathcal{C}_{\epsilon,\boldsymbol{\lambda}_0} \subset \mathcal{M}_\epsilon$ has $n_u + 2$ stable and $n_s$ unstable directions in $\mathbb{T}^N_{\text{ordered}}$. Further, $\mathcal{C}_{\epsilon,\boldsymbol{\lambda}_0} \subset \mathcal{M}_\epsilon$ lies $\mathcal{O}(\epsilon)$-close to $\mathcal{C}_{\boldsymbol{\lambda}_0} \subset \mathcal{M}$ so that, in particular, $||\boldsymbol{\Lambda}(\boldsymbol{\vartheta}) - \boldsymbol{\lambda}_0|| = \mathcal{O}(\epsilon)$ for every $\boldsymbol{\vartheta} \in \mathcal{C}_{\epsilon,\boldsymbol{\lambda}_0}$.

The remarks above concerning the meaning of $\boldsymbol{F}_h$ for determining robust orbits in $\mathcal{M}$ for a given perturbation function $h$ bring statements 2. and 3. of theorem 5.1 in perspective: For a given orbit $\mathcal{C}_{\boldsymbol{\lambda}_0} \subset \mathcal{M}$ to become a robust, i. e., hyperbolic orbit $\mathcal{C}_{\epsilon,\boldsymbol{\lambda}_0}$ of (2), its cross-ratios $\boldsymbol{\lambda}_0$ need to be (i) a zero of $\boldsymbol{F}_h$ and (ii) all eigenvalues of $\mathrm{D}\boldsymbol{F}_h(\boldsymbol{\lambda}_0)$ must have nonzero real part. Statements 2. and 3. of the theorem tell us that condition (i) is in fact always fulfilled for the splay state cross-ratios $\boldsymbol{\lambda}^*$. Whether a given setup (2) fulfills condition (ii) cannot be definitely answered with either yes or no for it depends intimately on the specifics form of the unperturbed system's dynamics (1) and the perturbation function $h$ as is discussed in section S.2.

However, as is shown in section 6, numerical evidence suggests that (ii) holds at least for the active rotator model and generic choices of $h$. As argued in S.2, we therefore conjecture that for general systems of type (2) and generic choices of $h$ the cross-ratios $\boldsymbol{\lambda}^*$ yield hyperbolic fixed points of the averaged system (S.8) so that consequently splay states should be generically robust solutions for systems of type (2). That this is not generally the case for perturbations to WS-integrable systems is illustrated in S.3. There, we show that $\boldsymbol{F}_h$ vanishes identically for the truncated system (22) and any smooth $h$ such that for this system, the averaging principle becomes sterile and theorem 5.1 trivial. Note however, that the truncated system does not belong to the class (1). With this, we end our discussion of the significance of $\boldsymbol{F}_h$ and proceed with the proof of theorem 5.1.

## S.2. Discussion of $\boldsymbol{F}_h$

In this appendix, we discuss the dependence of $\boldsymbol{F}_h$ on $h$ and its consequences for applying the averaging principle to determining the asymptotic dynamics of (2). In particular, we want to address the question whether the fixed point $\boldsymbol{\lambda}^*$ of (S.8) is hyperbolic for given $h$. This is important since in theorem 5.1, we only showed that $\boldsymbol{\lambda}^*$ is a fixed point which is necessary but not sufficient to conclude that splay states are robust solutions of (2).

From definition (24), we see that $\boldsymbol{F}_h$ is linear in $h$, i. e., for any smooth $h_1$ and $h_2$ and coefficients $a_1, a_2 \in \mathbb{R}$, we have

$$\boldsymbol{F}_{a_1 h_1 + a_2 h_2} = a_1 \boldsymbol{F}_{h_1} + a_2 \boldsymbol{F}_{h_2}. \tag{S.9}$$

The matrix $\mathrm{D}\boldsymbol{F}_h(\boldsymbol{\lambda}^*)$ is an element of the $(N-3)^2$-dimensional vector space of real-valued $(N-3) \times (N-3)$-matrices. Hence, every set of $(N-3)^2 + 1$ linearly independent perturbation functions $h_1, \ldots, h_{(N-3)^2+1}$ results in $(N-3)^2 + 1$ linearly dependent matrices $\mathrm{D}\boldsymbol{F}_{h_j}(\boldsymbol{\lambda}^*)$ so that there exists a nontrivial set of coefficients $a_1, \ldots, a_{(N-3)^2+1} \in \mathbb{R}$ such that (i) $h = \sum_j a_j h_j \neq 0$, (ii) we have

$$\boldsymbol{F}_h(\boldsymbol{\lambda}^*) = \sum_{j=1}^{(N-3)^2+1} a_j \boldsymbol{F}_{h_j}(\boldsymbol{\lambda}^*) = \boldsymbol{0}$$

from point 3. of theorem 5.1, and (iii)

$$\mathrm{D}\boldsymbol{F}_h(\boldsymbol{\lambda}^*) = \sum_{j=1}^{(N-3)^2+1} a_j \mathrm{D}\boldsymbol{F}_{h_j}(\boldsymbol{\lambda}^*) = \boldsymbol{0}$$

from (S.9). Hence, for every unperturbed system (1), one can always choose $h$ such that it does not result in a hyperbolic fixed point $\boldsymbol{\lambda}^*$ of the averaged equation (S.8) and thus one cannot apply the averaging principle to determine the existence of $\mathcal{C}_{\epsilon,\boldsymbol{\lambda}^*}$ for this perturbation function. On the other hand, if there exists a smooth $\tilde{h}$ for which the eigenvalues of $\mathrm{D}\boldsymbol{F}_{\tilde{h}}(\boldsymbol{\lambda}^*)$ all have nonzero real part, we can always add $\epsilon^2 \tilde{h}$ to $\epsilon h$ in (2) to make $\mathcal{C}_{\epsilon,\boldsymbol{\lambda}^*}$ hyperbolic. Since the numerical evidence presented in section 6 suggests that this is the case for the active rotator model and since this model is in no way special within the class of WS-integrable systems of type (1), we conjecture that hyperbolicity of the splay state solution for systems of type (2) is generic. Moreover, the numerical evidence suggest that hyperbolicity is in fact generic for all zeros of $\boldsymbol{F}_h$ (compare Panel (c) in figure 1.)

### S.3. Vanishing averages for the truncated system

We show that the averaged vector field (24) vanishes identically for the truncated system (22). Note that the solutions of (22) whose orbits $\mathcal{C}^{\text{closed}}_{\boldsymbol{\lambda}}$ form the NAIM $\mathcal{M}^{\text{closed}}_{\delta}$ are of the form

$$(\alpha(t), \psi(t), \boldsymbol{\lambda}(t)) = (\alpha, \Omega t, \boldsymbol{\lambda})$$

with constant $\alpha \in \mathbb{D}$ and $\boldsymbol{\lambda} \in V$ and that the corresponding solution in angular variables $\boldsymbol{\phi}$ reads

$$\boldsymbol{\phi}^{\text{closed}}_{\boldsymbol{\lambda}}(t) = \boldsymbol{m}(\alpha, \Omega t, \boldsymbol{\lambda}),$$

i. e., it is given via the Möbius action

$$\mathrm{e}^{\mathrm{i}\boldsymbol{\phi}^{\text{closed}}_{\boldsymbol{\lambda}}(t)} = G_{\alpha,\Omega t}(\mathrm{e}^{\mathrm{i}\boldsymbol{\Theta}(\boldsymbol{\lambda})}).$$

Hence, we have

$$\boldsymbol{F}_h(\boldsymbol{\lambda}) = \frac{\Omega}{2\pi} \int_0^{\frac{2\pi}{\Omega}} (\mathrm{D}\boldsymbol{\Lambda} \cdot \boldsymbol{h}) \circ \boldsymbol{\phi}^{\text{closed}}_{\boldsymbol{\lambda}}(t) \, \mathrm{d}t$$

$$= \frac{\Omega}{2\pi} \int_0^{\frac{2\pi}{\Omega}} (\mathrm{D}\boldsymbol{\Lambda} \cdot \boldsymbol{h}) \circ \boldsymbol{m}(\alpha, \Omega t, \boldsymbol{\lambda}) \, \mathrm{d}t$$

$$= \frac{1}{2\pi} \int_0^{2\pi} (\mathrm{D}\boldsymbol{\Lambda} \cdot \boldsymbol{h}) \circ \boldsymbol{m}(\alpha, \psi, \boldsymbol{\lambda}) \, \mathrm{d}\psi$$

$$= \frac{1}{2\pi} \int_0^{2\pi} \mathrm{D}\boldsymbol{\Lambda}(\boldsymbol{\theta}) \cdot \boldsymbol{h}(\boldsymbol{\theta}) \Big|_{\mathrm{e}^{\mathrm{i}\boldsymbol{\theta}} = G_{\alpha,\psi}(\boldsymbol{z})} \, \mathrm{d}\psi.$$

In other words, $\boldsymbol{F}_h(\boldsymbol{\lambda})$ is given by the zeroth Fourier mode of the function

$$\psi \mapsto \mathrm{D}\boldsymbol{\Lambda}(\boldsymbol{\theta}) \cdot \boldsymbol{h}(\boldsymbol{\theta}) \Big|_{\mathrm{e}^{\mathrm{i}\boldsymbol{\theta}} = G_{\alpha,\psi}(\boldsymbol{z})}$$

where we write $\boldsymbol{z} = \mathrm{e}^{\mathrm{i}\boldsymbol{\Theta}(\lambda)}$ for brevity. In what follows, we show that this zeroth mode vanishes identically for any $\alpha \in \mathbb{D}$, any $h$, and any $\boldsymbol{z} = (z_1, \ldots, z_N)$ of mutually distinct $z_j$ with $|z_j| = 1$, $j = 1, \ldots, N$. To do this, consider any nontrivial cross-ratio function $\Lambda_{p,q,r,s}$ (i.e., where $p, q, r, s \in \{1, \ldots, N\}$ are mutually distinct). We claim that the function

$$H : \psi \mapsto \mathrm{D}\Lambda_{p,q,r,s}(\boldsymbol{\theta}) \cdot \boldsymbol{h}(\boldsymbol{\theta})\Big|_{\mathrm{e}^{\mathrm{i}\boldsymbol{\theta}} = G_{\alpha,\psi}(\boldsymbol{z})}$$

has vanishing zeroth Fourier mode. By linearity of (24) in $h$, it suffices to show that this holds true for any

$$h : \theta \mapsto \mathrm{e}^{\mathrm{i}\nu\theta}$$

with $\nu \in \mathbb{Z}$.

The only nonzero entries of $\mathrm{D}\Lambda_{p,q,r,s}$ are

$$\begin{aligned}
\mathrm{D}_{\theta_p}\Lambda_{p,q,r,s}(\boldsymbol{\theta}) &= \mathrm{i}\frac{\mathrm{e}^{\mathrm{i}\theta_p}(\mathrm{e}^{\mathrm{i}\theta_r} - \mathrm{e}^{\mathrm{i}\theta_q})(\mathrm{e}^{\mathrm{i}\theta_r} - \mathrm{e}^{\mathrm{i}\theta_s})}{(\mathrm{e}^{\mathrm{i}\theta_p} - \mathrm{e}^{\mathrm{i}\theta_r})^2(\mathrm{e}^{\mathrm{i}\theta_q} - \mathrm{e}^{\mathrm{i}\theta_s})} \\
\mathrm{D}_{\theta_q}\Lambda_{p,q,r,s}(\boldsymbol{\theta}) &= \mathrm{i}\frac{\mathrm{e}^{\mathrm{i}\theta_q}(\mathrm{e}^{\mathrm{i}\theta_s} - \mathrm{e}^{\mathrm{i}\theta_p})(\mathrm{e}^{\mathrm{i}\theta_s} - \mathrm{e}^{\mathrm{i}\theta_r})}{(\mathrm{e}^{\mathrm{i}\theta_q} - \mathrm{e}^{\mathrm{i}\theta_s})^2(\mathrm{e}^{\mathrm{i}\theta_p} - \mathrm{e}^{\mathrm{i}\theta_r})} \\
\mathrm{D}_{\theta_r}\Lambda_{p,q,r,s}(\boldsymbol{\theta}) &= \mathrm{i}\frac{\mathrm{e}^{\mathrm{i}\theta_r}(\mathrm{e}^{\mathrm{i}\theta_q} - \mathrm{e}^{\mathrm{i}\theta_p})(\mathrm{e}^{\mathrm{i}\theta_p} - \mathrm{e}^{\mathrm{i}\theta_s})}{(\mathrm{e}^{\mathrm{i}\theta_p} - \mathrm{e}^{\mathrm{i}\theta_r})^2(\mathrm{e}^{\mathrm{i}\theta_q} - \mathrm{e}^{\mathrm{i}\theta_s})} \\
\mathrm{D}_{\theta_s}\Lambda_{p,q,r,s}(\boldsymbol{\theta}) &= \mathrm{i}\frac{\mathrm{e}^{\mathrm{i}\theta_s}(\mathrm{e}^{\mathrm{i}\theta_p} - \mathrm{e}^{\mathrm{i}\theta_q})(\mathrm{e}^{\mathrm{i}\theta_q} - \mathrm{e}^{\mathrm{i}\theta_r})}{(\mathrm{e}^{\mathrm{i}\theta_q} - \mathrm{e}^{\mathrm{i}\theta_s})^2(\mathrm{e}^{\mathrm{i}\theta_p} - \mathrm{e}^{\mathrm{i}\theta_r})}.
\end{aligned} \qquad (\mathrm{S.10})$$

We further find

$$\begin{aligned}
\mathrm{e}^{\mathrm{i}\theta_j} &= G_{\alpha,\psi}(z_j) \\
&= \frac{\alpha + \mathrm{e}^{\mathrm{i}\psi}z_j}{1 + \bar{\alpha}\mathrm{e}^{\mathrm{i}\psi}z_j} \\
&= \alpha + (1 - |\alpha|^2)\sum_{n=1}^{\infty}(-\bar{\alpha})^{n-1}z_j^n\mathrm{e}^{\mathrm{i}n\psi}
\end{aligned} \qquad (\mathrm{S.11})$$

where in the last line, we used again the geometric series identity $\sum_{n=0}^{\infty}\zeta^n = 1/(1-\zeta)$ for $|\zeta| < 1$. The last line of (S.11) is the Fourier series of the function $\psi \mapsto G_{\alpha,\psi}(z_j)$. Subsequently, we find for the powers of $\mathrm{e}^{\mathrm{i}\theta_j}$ and $^{-\mathrm{i}\theta_j}$ with $\nu \geq 0$ the Fourier series

$$\begin{aligned}
\mathrm{e}^{\mathrm{i}\nu\theta_j} &= \alpha^\nu + \nu z_j(1 - |\alpha|^2)\alpha^{\nu-1}\mathrm{e}^{\mathrm{i}\psi} + \text{ higher modes} \\
\mathrm{e}^{-\mathrm{i}\nu\theta_j} &= \bar{\alpha}^\nu + \nu\bar{z}_j(1 - |\alpha|^2)\bar{\alpha}^{\nu-1}\mathrm{e}^{-\mathrm{i}\psi} + \text{ higher modes}
\end{aligned}$$

up to their respective first modes. Inserting (S.11) in (S.10) and collecting powers of $\mathrm{e}^{\mathrm{i}\psi}$ yields

$$\begin{aligned}
\mathrm{D}_{\theta_p}\Lambda_{p,q,r,s}(\boldsymbol{\theta}) &= A_{p,q,r,s} + B_{p,q,r,s}\,\mathrm{e}^{-\mathrm{i}\psi} + C_{p,q,r,s}\,\mathrm{e}^{\mathrm{i}\psi} \\
\mathrm{D}_{\theta_q}\Lambda_{p,q,r,s}(\boldsymbol{\theta}) &= A_{q,p,s,r} + B_{q,p,s,r}\,\mathrm{e}^{-\mathrm{i}\psi} + C_{q,p,s,r}\,\mathrm{e}^{\mathrm{i}\psi} \\
\mathrm{D}_{\theta_r}\Lambda_{p,q,r,s}(\boldsymbol{\theta}) &= A_{r,s,p,q} + B_{r,s,p,q}\,\mathrm{e}^{-\mathrm{i}\psi} + C_{r,s,p,q}\,\mathrm{e}^{\mathrm{i}\psi} \\
\mathrm{D}_{\theta_s}\Lambda_{p,q,r,s}(\boldsymbol{\theta}) &= A_{s,r,q,p} + B_{s,r,q,p}\,\mathrm{e}^{-\mathrm{i}\psi} + C_{s,r,q,p}\,\mathrm{e}^{\mathrm{i}\psi}
\end{aligned}$$

with Fourier coefficients

$$A_{a,b,c,d} = -\mathrm{i}\frac{z_a(z_b - z_c)(z_c - z_d)(1 + |\alpha|^2)}{(z_a - z_c)^2(z_b - z_d)(1 - |\alpha|^2)}$$

$$B_{a,b,c,d} = -\mathrm{i}\frac{(z_b - z_c)(z_c - z_d)\alpha}{(z_a - z_c)^2(z_b - z_d)(1 - |\alpha|^2)}$$

$$C_{a,b,c,d} = -i\frac{z_a^2(z_b - z_c)(z_c - z_d)\bar{\alpha}}{(z_a - z_c)^2(z_b - z_d)(1 - |\alpha|^2)}$$

for which one checks that the identities

$$\begin{aligned}
A_{p,q,r,s} + A_{q,p,s,r} + A_{r,s,p,q} + A_{s,r,q,p} &= 0 \\
B_{p,q,r,s}\, z_p + B_{q,p,s,r}\, z_q + B_{r,s,p,q}\, z_r + B_{s,r,q,p}\, z_s &= 0 \quad\quad\text{(S.12)}\\
C_{p,q,r,s}\, \bar{z}_p + C_{q,p,s,r}\, \bar{z}_q + C_{r,s,p,q}\, \bar{z}_r + C_{s,r,q,p}\, \bar{z}_s &= 0
\end{aligned}$$

hold. We are now ready to determine the Fourier expansion of $H$. Beginning with $h : \theta \mapsto \mathrm{e}^{\mathrm{i}\nu\theta}$ and $\nu \geq 0$, we find

$$\begin{aligned}
H(\psi) &= \mathrm{D}\Lambda_{p,q,r,s}(\boldsymbol{\theta}) \cdot \boldsymbol{h}(\boldsymbol{\theta})\Big|_{\mathrm{e}^{\mathrm{i}\boldsymbol{\theta}} = G_{\alpha,\psi}(\boldsymbol{z})} \\
&= \mathrm{D}_{\theta_p}\Lambda_{p,q,r,s}(\boldsymbol{\theta})\mathrm{e}^{\mathrm{i}\nu\theta_p} + \mathrm{D}_{\theta_q}\Lambda_{p,q,r,s}(\boldsymbol{\theta})\mathrm{e}^{\mathrm{i}\nu\theta_q} + \\
&\quad + \mathrm{D}_{\theta_r}\Lambda_{p,q,r,s}(\boldsymbol{\theta})\mathrm{e}^{\mathrm{i}\nu\theta_r} + \mathrm{D}_{\theta_s}\Lambda_{p,q,r,s}(\boldsymbol{\theta})\mathrm{e}^{\mathrm{i}\nu\theta_s} \\
&= (A_{p,q,r,s} + A_{q,p,s,r} + A_{r,s,p,q} + A_{s,r,q,p})\alpha^\nu + \\
&\quad + (B_{p,q,r,s}\, z_p + B_{q,p,s,r}\, z_q + \\
&\quad\quad + B_{r,s,p,q}\, z_r + B_{s,r,q,p}\, z_s)\nu(1 - |\alpha|^2)\alpha^{\nu-1} + \\
&\quad + \text{higher modes}
\end{aligned}$$

for which the zeroth mode vanishes identically because of the identities (S.12). Similarly, we find for the case $h : \theta \mapsto \mathrm{e}^{-\mathrm{i}\nu\theta}$ with $\nu \geq 0$ the expansion

$$\begin{aligned}
H(\psi) &= \mathrm{D}\Lambda_{p,q,r,s}(\boldsymbol{\theta}) \cdot \boldsymbol{h}(\boldsymbol{\theta})\Big|_{\mathrm{e}^{\mathrm{i}\boldsymbol{\theta}} = G_{\alpha,\psi}(\boldsymbol{z})} \\
&= \mathrm{D}_{\theta_p}\Lambda_{p,q,r,s}(\boldsymbol{\theta})^{-\mathrm{i}\nu\theta_p} + \mathrm{D}_{\theta_q}\Lambda_{p,q,r,s}(\boldsymbol{\theta})^{-\mathrm{i}\nu\theta_q} + \\
&\quad + \mathrm{D}_{\theta_r}\Lambda_{p,q,r,s}(\boldsymbol{\theta})^{-\mathrm{i}\nu\theta_r} + \mathrm{D}_{\theta_s}\Lambda_{p,q,r,s}(\boldsymbol{\theta})^{-\mathrm{i}\nu\theta_s} \\
&= (A_{p,q,r,s} + A_{q,p,s,r} + A_{r,s,p,q} + A_{s,r,q,p})\bar{\alpha}^\nu + \\
&\quad + (C_{p,q,r,s}\, \bar{z}_p + C_{q,p,s,r}\, \bar{z}_q + \\
&\quad\quad + C_{r,s,p,q}\, \bar{z}_r + C_{s,r,q,p}\, \bar{z}_s)\nu(1 - |\alpha|^2)\bar{\alpha}^{\nu-1} + \\
&\quad + \text{higher modes}
\end{aligned}$$

which again vanishes because of (S.12). This finishes the proof.

# References

[1] Fenichel N 1977 *Indiana University Mathematics Journal* **26** 81–93




\section*{References}
\begin{thebibliography}{99}
	
	\bibitem{Cawthorne_1999} 
	Cawthorne A B, Barbara P, Shitov S V, Lobb C J, Wiesenfeld K and Zangwill A 
	\newblock{1999} 
	\newblock{\it \PR B} 
	\newblock{\bf 60} 
	\newblock{7575--7578} 
	
	\bibitem{Pikovsky_Rosenblum_Kurths_2001}
	Pikovsky A,	Rosenblum M and Kurths J 
	\newblock{2003} 
	\newblock{\it Synchronization} 
	\newblock{(Cambridge: Cambridge University Press)}
	
	\bibitem{Strogatz_2005}
	Strogatz S H, Abrams D M, McRobie A, Eckhardt B and Ott E
	\newblock{2005}
	\newblock{\it Nature}
	\newblock{\bf 438}
	\newblock{43--44}
	
	\bibitem{Motter_2013}
	Motter A E, Myers S A, Anghel M and Nishikawa T
	\newblock{2013}
	\newblock{\it Nature Physics}
	\newblock{\bf 9}
	\newblock{191--197}
	
	\bibitem{Singer_1993}
	Singer W
	\newblock{1993}
	\newblock{\it Annu. Rev. Physiol.}
	\newblock{\bf 55}
	\newblock{349--374}
	
	\bibitem{Fries_2015}
	Fries P
	\newblock{2015}
	\newblock{\it Neuron}
	\newblock{\bf 88}
	\newblock{220--235}
	
	\bibitem{Hegselmann_2002}
	Hegselmann R and Krause U
	\newblock{2002}
	\newblock{\it J. Artif. Soc. Soc. Simul.}
	\newblock{\bf 5}
	
	\bibitem{Pluchino_2006}
	Pluchino A, Latora V and Rapisarda A
	\newblock{2006}
	\newblock{\it Eur. Phys. J. B}
	\newblock{\bf 50}
	\newblock{169--176}
	
	\bibitem{Kuramoto_1975}
	Kuramoto Y
	\newblock{1975}
	\newblock{\it Int. Symp. on Mathematical Problems in Theoretical Physics (Lecture Notes in Physics)}
	\newblock{ed H Araki (Berlin: Springer)}
	\newblock{pp 420--2}
	
	\bibitem{Sakaguchi_1986}
	Sakaguchi H and Kuramoto Y
	\newblock{1986}
	\newblock{\it Prog. Theor. Exp. Phys.}
	\newblock{\bf 76}
	\newblock{576--581}
	
	\bibitem{Shinomoto_1986}
	Shinomoto S and Kuramoto Y
	\newblock{1986}
	\newblock{\it Prog. Theor. Exp. Phys.}
	\newblock{\bf 75}
	\newblock{1105--1110}
	
	\bibitem{Watanabe_Strogatz_1994}
	Watanabe S and Strogatz S H
	\newblock{1994}
	\newblock{\it Physica D}
	\newblock{\bf 74}
	\newblock{197--253}
	
	\bibitem{Marvel_Mirollo_Strogatz_2009}
	Marvel S A, Mirollo R E and Strogatz S H
	\newblock{2009}
	\newblock{\it Chaos}
	\newblock{\bf 19}
	\newblock{043104}
	
	\bibitem{Eldering_2021}
	Eldering J, Lamb J S W, Pereira T and Roque dos Santos E
	\newblock{2021}
	\newblock{\it \NL}
	\newblock{\bf 34}
	\newblock{5344--5374}
	
	\bibitem{Engelbrecht_Mirollo_2014}
	Engelbrecht J R and Mirollo R
	\newblock{2014}
	\newblock{\it Chaos}
	\newblock{\bf 24}
	\newblock{013114}

	\bibitem{Gong_2019_2}
	Gong C C, Zheng C, Toenjes R and Pikovsky A
	\newblock{2019}
	\newblock{\it Chaos}
	\newblock{\bf 29}
	\newblock{033127}
	
	\bibitem{Ronge_Zaks_2021}
	Ronge R and Zaks M A
	\newblock{2021}
	\newblock{\it \PR E}
	\newblock{\bf 103}
	\newblock{012206}	
	
	\bibitem{Aronson_1991}
	Aronson D G, Golubitsky M and Mallet-Paret J
	\newblock{1991}
	\newblock{\it \NL}
	\newblock{\bf 4}
	\newblock{903}
	
	\bibitem{Mirollo_1994}
	Mirollo R E
	\newblock{1994}
	\newblock{\it SIAM J. Math. Anal.}
	\newblock{\bf 25}
	\newblock{1176--1180}
	
	
	\bibitem{Acebron_2005}
	Acebr{\'o}n J A, Bonilla L L, Vicente C J P, Ritort F and Spigler R
	\newblock{2005}
	\newblock{\it RMP}
	\newblock{\bf 77}
	\newblock{137}
	
	\bibitem{Stankovski_2017}
	Stankovski T, Pereira T, McClintock P V E and Stefanovska A
	\newblock{2017}
	\newblock{\it \RMP}
	\newblock{\bf 89}
	\newblock{045001}
	
	\bibitem{Laing_2018}
	Laing C R
	\newblock{2018}
	\newblock{\it J. Math. Neurosci.}
	\newblock{\bf 8}
	\newblock{1--24}
	
	\bibitem{Bick_2020}
	Bick C, Goodfellow M, Laing C R and Martens E A
	\newblock{2020}
	\newblock{\it J. Math. Neurosci.}
	\newblock{\bf 10}
	\newblock{1--43}
	
	\bibitem{Izhikevich_2007}
	Izhikevich E M
	\newblock{2010}
	\newblock{\it Dynamical Systems in Neuroscience}
	\newblock{(Cambridge: MIT Press)}
	
	\bibitem{Zaks_Tomov_2016}
	Zaks M A and Tomov P
	\newblock{2016}
	\newblock{\it \PR E}
	\newblock{\bf 93}
	\newblock{020201}
	
	
	\bibitem{Ronge_Zaks_2021_2}
	Ronge R and Zaks M A
	\newblock{2021}
	\newblock{\it Eur. Phys. J. ST}
	\newblock{\bf 230}
	\newblock{2717--2724}
	
	\bibitem{Sanders_Verhulst_Murdock_2007}
	Sanders J A, Verhulst F and Murdock J
	\newblock{2007}
	\newblock{\it Averaging Methods in Nonlinear Dynamical Systems}
	\newblock{(\it{Applied Mathematical Sciences} \rm{vol 59})}
	\newblock{ed S S Antman, J E Marsden and L Sirovich}
	\newblock{(New York, Springer)}
	
	\bibitem{Ahlfors_1953}
	Ahlfors L V
	\newblock{1953}
	\newblock{\it Complex Analysis: an Introduction to the Theory of Analytic Functions of one Complex Variable}
	\newblock{(New York: McGraw-Hill)}
	
	\bibitem{Shilnikov2001}
	Shilnikov L P, Shilnikov A L, Turaev D V and Chua L O
	\newblock{1998}
	\newblock{\it Methods of Qualitative Theory in Nonlinear Dynamics}
	\newblock{(Singapore: World Scientific)}
	
	\bibitem{Eldering_2013}
	Eldering J
	\newblock{2013}
	\newblock{\it Normally Hyperbolic Invariant Manifolds: the Noncompact Case}
	\newblock{(\it{Atlantis Studies in Dynamical Systems} \rm{vol 2})}
	\newblock{Paris, Atlantis Press}
	
	
	\bibitem{HirschPughShub2006}
	Hirsch M W, Pugh C C and Shub M
	\newblock{2006}
	\newblock{\it Invariant Manifolds}
	\newblock{(\it{Lecture Notes in Mathematics} \rm{vol 583})}
	\newblock{Heidelberg,Berlin: Springer}

	\bibitem{Chicone_2006}
	Chicone C
	\newblock{1999}
	\newblock{\it Ordinary Differential Equations with Applications}
	\newblock{(\it{Texts in Applied Mathematics} \rm{vol 34})}
	\newblock{New York: Springer}
	

	\bibitem{Gantmacher_2000}
	Gantmacher F R
	\newblock{2000}
	\newblock{\it The theory of matrices I}
	\newblock{(Providence: AMS Chelsea Publishing)}
		
	
	\bibitem{Rudin1964}
	Rudin W
	\newblock{1964}
	\newblock{\it Principles of Mathematical Analysis}
	\newblock{(\it{International series in pure and applied mathematics})}
	\newblock{ed W T Martin, E H Spanier, G Springer and
		P J Davis}
	\newblock{(New York: McGraw-Hill)}

	\bibitem{Pikovsky_2008}
	Pikovsky A and Rosenblum M
	\newblock{2008}
	\newblock{\it \PRL}
	\newblock{\bf 101}
	\newblock{264103}
	
	
	\bibitem{GolubitskyStewart2003}
	Golubitsky M and Stewart I
	\newblock{2000}
	\newblock{The Symmetry Perspective}
	\newblock{(\it{Progress in Mathematics} vol 200)}
	\newblock{(Basel: Birkhäuser)}

	\bibitem{Fenichel1977}
	Fenichel N
	\newblock{1977}
	\newblock{\it Indiana University Mathematics Journal}
	\newblock{\bf 26}
	\newblock{81--93}
	
	\bibitem{Baibolatov_2009}
	Baibolatov Y, Rosenblum M, Zhanabaev Z Z, Kyzgarina M and Pikovsky A
	\newblock{2009}
	\newblock{\it \PR E}
	\newblock{\bf 80}
	\newblock{046211}

	\bibitem{Hong_2011}
	Hong H and Strogatz S H
	\newblock{2011}
	\newblock{\it \PR E}
	\newblock{\bf 84}
	\newblock{046202}	

\end{thebibliography}
\end{document}